\documentclass[10pt,a4paper]{article}
\usepackage{xcolor}
\usepackage[utf8]{inputenc}
\usepackage{amsmath, amsthm, amssymb, bbm, bm}
\usepackage{stmaryrd}
\usepackage{tikz-cd}
\usepackage{enumerate}
\usepackage{appendix}
\usepackage[hidelinks]{hyperref}
\usepackage{color}

\definecolor{mygray}{gray}{0.75}

\newtheorem{tvrz}{Proposition}[section]
\newtheorem{lemma}[tvrz]{Lemma}
\newtheorem{theorem}[tvrz]{Theorem}
\newtheorem{cor}[tvrz]{Corollary}
\theoremstyle{definition}
\newtheorem{definice}[tvrz]{Definition}
\theoremstyle{remark}
\newtheorem{rem}[tvrz]{Remark}
\newtheorem{example}[tvrz]{Example}
\theoremstyle{definition}

\def\^{\wedge}
\def\<{\langle}
\def\>{\rangle}
\def\M{\mathcal{M}}

\def\N{\mathbb{N}}

\def\X{\mathfrak{X}}

\def\frP{\mathfrak{P}}

\def\R{\mathbb{R}}

\def\C{\mathcal{C}}

\def\E{\mathcal{E}}

\def\F{\mathcal{F}}
\def\G{\mathcal{G}}

\def\V{\mathcal{V}}
\def\J{\mathcal{J}}

\def\Z{\mathbb{Z}}

\def\ol{\overline}

\def\fdelta{\bm{\delta}}

\def\fJ{\mathbf{J}}
\def\frJ{\mathfrak{J}}
\def\frj{\mathfrak{j}}

\def\frP{\mathfrak{P}}

\def\btr{\blacktriangleright}

\def\bbz{\mathbbm{z}}

\def\fK{\mathbf{K}}

\def\fD{\mathbf{D}}

\def\dr{\mathrm{d}}

\def\cD{\nabla}

\def\cX{\mathcal{X}}

\def\1{\mathbbm{1}}

\def\fF{\mathbf{F}}

\def\fI{\mathbf{I}}

\def\fJ{\mathbf{J}}
\def\fnula{\mathbf{0}}
\def\f1{\mathbf{1}}

\def\fJ{\mathbf{J}}

\def\~{\widetilde}

\def\tr{\triangleright}
\def\tl{\triangleleft}

\def\frD{\mathfrak{D}}
\def\frE{\mathfrak{E}}

\newcommand\ul[1]{\underline{#1}}

\DeclareMathOperator{\gdim}{gdim}

\DeclareMathOperator{\Loc}{Loc}
\DeclareMathOperator{\Dif}{Dif}
\DeclareMathOperator{\Jet}{Jet}
\DeclareMathOperator{\pJet}{pJet}
\DeclareMathOperator{\gpJet}{gpJet}

\DeclareMathOperator{\Sff}{Sff}
\DeclareMathOperator{\sff}{sff}
\DeclareMathOperator{\geo}{geo}
\DeclareMathOperator{\Geo}{Geo}

\DeclareMathOperator{\gVect}{\mathbf{gVec}}

\DeclareMathOperator{\PSh}{\mathbf{PSh}}

\DeclareMathOperator{\Sh}{\mathbf{Sh}}

\DeclareMathOperator{\Op}{\mathbf{Op}}

\DeclareMathOperator{\gVBun}{\mathbf{gVBun}}

\DeclareMathOperator{\Diff}{Diff}

\DeclareMathOperator{\Hom}{Hom}

\DeclareMathOperator{\op}{op}

\DeclareMathOperator{\Lin}{Lin}

\DeclareMathOperator{\At}{At}

\DeclareMathOperator{\rk}{rk}
\DeclareMathOperator{\grk}{grk}
\DeclareMathOperator{\trk}{trk}
\DeclareMathOperator{\supp}{supp}

\DeclareMathOperator{\Ilim}{\varprojlim}

\begin{document}
\begin{flushright}
\today
\end{flushright}
\vspace{0.7cm}
\begin{center}

\baselineskip=13pt {\Large \bf{Graded Jet Geometry}\\}
 \vskip0.5cm
 {\large{Jan Vysoký$^{1}$}}\\
 \vskip0.6cm
$^{1}$\textit{Faculty of Nuclear Sciences and Physical Engineering, Czech Technical University in Prague\\ Břehová 7, 115 19 Prague 1, Czech Republic, jan.vysoky@fjfi.cvut.cz}\\
\vskip0.3cm
\end{center}

\begin{abstract}
Jet manifolds and vector bundles allow one to employ tools of differential geometry to study differential equations, for example those arising as equations of motions in physics. They are necessary for a geometrical formulation of Lagrangian mechanics and the calculus of variations. It is thus only natural to require their generalization in  geometry of $\Z$-graded manifolds and vector bundles. 

Our aim is to construct the $k$-th order jet bundle $\frJ^{k}_{\E}$ of an arbitrary $\Z$-graded vector bundle $\E$ over an arbitrary $\Z$-graded manifold $\M$. We do so by directly constructing its sheaf of sections, which allows one to quickly prove all its usual properties. It turns out that it is convenient to start with the construction of the graded vector bundle of $k$-th order (linear) differential operators $\frD^{k}_{\E}$ on $\E$. In the process, we discuss (principal) symbol maps and a subclass of differential operators whose symbols correspond to completely symmetric $k$-vector fields, thus finding a graded version of Atiyah Lie algebroid. Necessary rudiments of geometry of $\Z$-graded vector bundles over $\Z$-graded manifolds are recalled. 
\end{abstract}

{\textit{Keywords}: Jet manifolds, graded manifolds, graded vector bundles, jet vector bundles, linear differential operators}.

\section*{Introduction}
Introduced by Ehresmann \cite{ehresmann1953introductiona}, jet geometry forms one of the pillars of modern differential geometry and mathematical physics. It proved to be an immeasurable tool to understand the geometry of differential equations and their solutions, and consequently of equations of classical and quantum physics. In particular, it is vital for a geometrical description of Lagrangian mechanics and the calculus of variations. It is of course not possible to provide a full list of references, hence let us at least point out standard monographs \cite{saunders1989geometry, kolar2013natural} for the geometry of jet bundles. For applications in theoretical physics, see e.g. \cite{bocharov1999symmetries} and \cite{giachetta1997new}. 

Linear differential operators on sections of vector bundles can be naturally described in terms of jet bundles \cite{palais1965chapter}. There is their fully algebraic description, tracing back to none other than Grothendieck \cite{grothendieck1966elements}. This viewpoint is fully utilized in a more conceptual algebraic approach to jet bundles, using their modules of sections \cite{krasil1986geometry, nestruev2003smooth}, which should be considered as a main source of inspiration for methods used in this paper. 

The need for a geometrical description of supersymmetric field theories and their variational calculus lead naturally to the notion of jet supermanifolds and jet supervector bundles. See \cite{hernandez1984global, monterde1992higher, monterde1993variational, monterde2006poincare} and a different approach in \cite{sardanashvily2007graded, sardanashvily2013graded, giachetta2005lagrangian}. Note that in the above list of references, the term ``graded'' always refers to the $\Z_{2}$-grading. 

In recent years, there was a reinvigoration of interest in $\Z$-graded manifolds with local coordinates of arbitrary degrees \cite{vysoky2022graded, kotov2021category, fairon2017introduction}, following on the non-negatively graded (or just $\N$-graded) manifolds \cite{Kontsevich:1997vb, severa2001some, Voronov:2019mav} and \cite{mehta2006supergroupoids,2011RvMaP..23..669C}. It is thus natural to consider a generalization of jet geometry to the (general) $\Z$-graded setting. There is a notion of graded jet bundles for $\N$-manifolds in the literature \cite{sardanashvily2017differential}, based on the explicit utilization of the Batchelor isomorphism for supermanifolds.

In this paper, we aim to provide a construction of jet bundles for general $\Z$-graded manifolds, following the definitions and formalism of our previous work \cite{Vysoky:2022gm}. Let us henceforth use the term ``graded'' \textit{exclusively} for $\Z$-graded (manifolds, vector bundles). However, it shall be noted that all of the constructions in this paper apply immediately to the $\Z_{2}$-graded case (Berezin-Leites supergeometry) and perhaps to some more exotic gradings as well, e.g. $\Z_{2}^{n}$. 

Let us sketch the main philosophy of the construction. Graded vector bundles over a graded manifold $\M$ can be fully described in terms of sheaves of graded $\C^{\infty}_{\M}$-modules, where $\C^{\infty}_{\M}$ is the structure sheaf of $\M$. Linear differential operators on graded vector bundles can be then readily defined using the graded version of the iterative algebraic definition of Grothendieck, see also \cite{nestruev2003smooth}. We prove that they form a locally freely and finitely generated sheaves of graded $\C^{\infty}_{\M}$-modules, hence new examples of graded vector bundles. 

This gave us the confidence that one can construct sheaves of sections of graded jet bundles directly, starting from a sheaf of sections of any given graded vector bundle. The procedure uses a graded analogue of jet modules in commutative algebra, see \cite{krasil1986geometry, nestruev2003smooth}, described also in \cite{sardanashvily2017differential}. However, the resulting presheaf of graded $\C^{\infty}_{\M}$-modules is way too big (it has non-trivial sections ``vanishing at all points'') at it does not form a sheaf. We address those technical details by introducing the notion of geometric presheaves of $\C^{\infty}_{\M}$-modules, inspired by \cite{krasil1986geometry, nestruev2003smooth}. One can then shape the problematic presheaf into a locally freely and finitely generated sheaf of graded $\C^{\infty}_{\M}$-modules, hence a sheaf of sections of a graded vector bundle, without loosing any of the functorial properties. In particular, it retains its natural relation to linear differential operators. Since every graded vector bundle (as a sheaf) can be used to construct the unique (up to a diffeomorphism) total space graded manifold, we obtain another tool for a construction of new graded manifolds. 

Let us also note that recently, there appeared a general construction of jet bundles in context of noncommutative geometry \cite{flood2022jet}, defining $k$-th order jet modules starting from any given $A$-module, where $A$ is any unital associative algebra. 

The paper is organized as follows:

In Section \ref{sec_local}, we start by recalling the necessary notions from graded geometry, namely the ones of a graded manifold and a graded vector bundle. We generalize the notion of a local operator on a graded module of sections of a graded vector bundle and show that those naturally form a sheaf of graded modules. We argue that endomorphisms of graded modules of sections (of any degree) form an example of local operators. 

In Section \ref{sec_diffop}, we define linear differential operators (of a given order) on graded vector bundles. We prove that they are local and it turns out that they can be viewed as sections of a sheaf of graded submodules of the sheaf of local operators. The fact that it is in fact a sheaf of sections of graded vector vector bundle is proved in Section \ref{sec_diffopVB}. We do so by explicitly by finding its local frame using a local frame of the graded vector bundle in question, together with a graded local chart of the underlying manifold. This also relates the algebraic definition of differential operators to their more self-explanatory explicit local form. 

In Section \ref{sec_symbol}, we first recall the notion of completely symmetric (differential) forms on a graded manifold, and the ``internal hom'' in the category of graded vector bundles over a given graded manifold. We do so in order to define a graded analogue of a (principal) symbol map. We prove that it fits into a canonical short exact sequence of graded vector bundles. 

The symbol map can be then used in Section \ref{sec_atiyah} to define a certain subclass of differential operators (their respective symbol map is determined by a single completely symmetric $k$-vector field). It turns out that this subset is in a certain sense closed under a graded commutator and the corresponding algebra of symbol maps is governed by a graded version of a Schouten-Nijenhuis bracket. In particular, for first-order differential operators, we obtain a graded Lie algebra on the graded module of sections of a so called Atiyah graded vector bundle, making it into a non-trivial example of a graded Lie algebroid. 

In Section \ref{sec_geometric}, we discuss a subclass of graded modules (over graded algebras of functions on graded manifolds) which contain no sections that in some sense ``vanish at all points''. More generally, we define a subclass of ``geometric'' presheaves of graded $\C^{\infty}_{\M}$-modules, motivated by the similar notion introduced in \cite{krasil1986geometry}. The name is justified by proving that sheaves of graded vector bundles are always geometric. We prove that there is a canonical ``geometrization'' functor into the subcategory of geometrical presheaves of graded $\C^{\infty}_{\M}$-modules. 

Section \ref{sec_gJetconst} is central to this paper. The sheaf of ($k$-th order) jets is constructed from a sheaf of sections of a given graded vector bundle. We prove that this sheaf is locally freely and finitely generated of a constant graded rank, hence a sheaf of sections of a graded vector bundle. We declare it to be the $k$-th order jet bundle of a graded vector bundle. By choosing a trivial ``line bundle'' over a given graded manifold, we obtain its $k$-th order jet manifold. 

Finally, Section \ref{sec_gJetprops} provides a justification of the previous section. In other words, we prove that the graded vector bundle we have constructed has all the usual properties of jet bundles in ordinary differential geometry. We show that there is a canonical jet prolongation map which becomes an isomorphism with the original graded vector bundle map for $k = 0$. We argue that $k$-th order differential operators on a graded vector bundle $\E$ can be equivalently described as graded vector bundle maps from the $k$-th jet bundle to $\E$. It is shown that jet bundles give raise to a canonical inverse system of sheaves, allowing one to define the sheaf $\Jet^{\infty}_{\E}$. We argue that fibers of graded jet bundles coincide with graded jet spaces defined at each point of the underlying smooth manifold in a usual way. We conclude by proving that in the ordinary (i.e. trivially graded) case, we obtain the usual definition of jet bundles. 

\section*{Acknowledgments}
First and foremost, I would like to express gratitude to my family for their patience and support.
I would also like to thank Branislav Jurčo and Rudolf Šmolka for helpful discussions. 

The author is grateful for a financial support from MŠMT under grant no. RVO 14000.

\section{Graded vector bundles, local operators} \label{sec_local}
In this section, we will briefly recall elementary definitions required in this paper. The notion of local operators on graded vector bundles is introduced and we prove that they actually form a sheaf of graded modules. 

Throughout this paper, $\M$ will be some fixed but otherwise general $\Z$-graded manifold. Its structure sheaf of smooth functions will be denoted as $\C^{\infty}_{\M}$ and the underlying smooth manifold as $M$. A sequence $(n_{j})_{j \in \Z}$ of non-negative integers is called a graded dimension of $\M$, if $\M$ is locally modeled on the graded domain $\hat{U}^{(n_{j})_{j}}$. In plain English, the number of its local coordinates of degree $j$ is precisely $n_{j}$. We assume that its \textit{total dimension} $n := \sum_{j \in \Z} n_{j}$ is finite. Note that then $\dim(M) = n_{0}$. For a detailed definition and related discussions see \S 3.3 of \cite{Vysoky:2022gm}. Since $M$ is a topological space, we can write $\Op(M)$ for its set of open subsets. For any $m \in M$, $\Op_{m}(M)$ denotes the set of open subsets of $M$ containing $m$.

For each $U \in \Op(M)$, elements of $\C^{\infty}_{\M}(U)$ all called \textbf{(smooth) functions on $\M$ over $U$}. To each $f \in \C^{\infty}_{\M}(U)$, there is an associated smooth function $\ul{f}: U \rightarrow \R$, called the \textbf{body of $f$}. For each $m \in U$, one defines $f(m) := \ul{f}(m)$. Note that whenever $|f| \neq	0$, one has $\ul{f} = 0$ and thus $f(m) = 0$ for all $m \in U$. 

\begin{definice}
By a \textbf{graded vector bundle} $\E$ over $\M$, we mean a locally freely and finitely generated sheaf $\Gamma_{\E}$ of graded $\C^{\infty}_{\M}$-modules on $M$ of a constant graded rank. $\Gamma_{\E}$ is called the \textbf{sheaf of sections of $\E$}. 
\end{definice}

A sheaf of graded $\C^{\infty}_{\M}$-modules $\Gamma_{\E}$ assigns to each $U \in \Op(M)$ a graded $\C^{\infty}_{\M}(U)$-module $\Gamma_{\E}(U)$ and the action of $\C^{\infty}_{\M}(U)$ has to be compatible with restrictions. By ``locally freely and finitely generated'' we mean the existence of a \textbf{local trivialization} $\{ (U_{\alpha},\phi_{\alpha}) \}_{\alpha \in I}$. More precisely, $\{ U_{\alpha} \}_{\alpha \in I}$ is an open cover of $M$ and each $\phi_{\alpha}$ is a $\C^{\infty}_{\M}|_{U_{\alpha}}$-linear sheaf isomorphism of $\Gamma_{\E}|_{U_{\alpha}}$ and a freely and finitely generated graded sheaf of $\C^{\infty}_{\M}|_{U_{\alpha}}$-modules $\C^{\infty}_{\M}|_{U_{\alpha}}[K]$.  To every $U \in \Op(U_{\alpha})$, it assigns a graded $\C^{\infty}_{\M}(U)$-module (with the obvious action)
\begin{equation} \label{eq_trivialbundle}
(\C^{\infty}_{\M}|_{U_{\alpha}}[K])(U) := \C^{\infty}_{\M}(U) \otimes_{\R} K,
\end{equation}
where $K \in \gVect$ is some fixed (for all $\alpha \in I$) finite-dimensional graded vector space called the \textbf{typical fiber} of $\E$. Its graded dimension $\gdim(K) =: (r_{j})_{j \in \Z}$ is called the \textbf{graded rank} of $\E$ and denoted as $\grk(\E)$. The number $r := \sum_{j \in \Z} r_{j}$ is called the \textbf{total rank} of $\E$ and denoted as $\trk(\E)$. 

Equivalently, for each $m \in M$, there exists an open subset $U \in \Op_{m}(M)$ and a collection $\{ \Phi_{\lambda} \}_{\lambda=1}^{r}$ of elements of $\Gamma_{\E}(U)$, such that  
\begin{enumerate}[(i)]
\item $|\Phi_{\lambda}| = |\vartheta_{\lambda}|$, where $\{ \vartheta_{\lambda} \}_{\lambda=1}^{r}$ is some fixed total basis of the typical fiber $K$. 
\item For every $V \in \Op(U)$, every $\psi \in \Gamma_{\E}(V)$ can be written as $\psi = \psi^{\lambda} \cdot \Phi_{\lambda}|_{V}$ for unique functions $\psi^{\lambda} \in \C^{\infty}_{\M}(V)$ such that $|\psi^{\lambda}| = |\psi| - |\Phi_{\lambda}|$. 
\end{enumerate}
Every such collection $\{ \Phi_{\lambda} \}_{\lambda=1}^{r}$ is called a \textbf{local frame for $\E$ over $U$}. Note that we usually denote the action of $f \in \C^{\infty}_{\M}(U)$ on $\psi \in \Gamma_{\E}(U)$ simply as $f \cdot \psi$. 

\begin{rem} \label{rem_totalspace}
There exists a more geometrical approach to graded vector bundles. In fact, one can always construct a \textit{total space} graded manifold $\E$ (unique up to a graded diffeomorphism) and a surjective submersion $\pi: \E \rightarrow \M$. The sheaf of sections $\Gamma_{\E}$ can be then fully recovered from $\C^{\infty}_{\E}$. See $\S 5.5$ in \cite{Vysoky:2022gm} for details. 
\end{rem}

The most vital property of graded vector bundles, necessary in everything what follows in this paper, is the ability to extend (in some sense) local sections to global sections. It is proved by using smooth bump functions, similarly to the ordinary case. Unless explicitly stated, $\E$ will always denote some fixed but arbitrary graded vector bundle over $\M$ of a graded rank $(r_{j})_{j \in \Z}$. 
\begin{tvrz}[\textbf{The Extension Lemma}]
Suppose $\psi \in \Gamma_{\E}(U)$ for some $U \in \Op(M)$. Then for any $V \in \Op(U)$ such that $\ol{V} \subseteq U$, there exists a section $\psi' \in \Gamma_{\E}(M)$, such that $\psi'|_{V} = \psi|_{V}$. 
\end{tvrz}

This property will allow us to restrict certain linear maps of sections of graded vector bundles from larger open subsets to smaller open subsets. We only have to limit ourselves to those which depend ``locally'' on sections. We thus arrive to the following important notion.
\begin{definice}
Let $U \in \Op(M)$. A graded $\R$-linear map $F: \Gamma_{\E}(U) \rightarrow \Gamma_{\E}(U)$ of any given degree $|F|$ is called a \textbf{local operator on $\E$ over $U$}, if for any $V \in \Op(U)$ and $\psi \in \Gamma_{\E}(V)$ satisfying $\psi|_{V} = 0$, one has $F(\psi)|_{V} = 0$. 

Local operators on $\E$ over $U$ form a graded vector space which we denote as $\Loc_{\E}(U)$. 
\end{definice}
As promised, local operators can be restricted in a unique way. In fact, these restrictions allow us to obtain a sheaf of graded $\C^{\infty}_{\M}$-modules of \textit{all} local operators on $\E$. 
\begin{tvrz} \label{tvrz_locsheaf} Let $U \in \Op(M)$ and $F \in \Loc_{\E}(U)$. Then for any $V \in \Op(U)$, there is the unique local operator $F|_{V} \in \Loc_{\E}(V)$ on $\E$ over $V$, such that 
\begin{equation} \label{eq_Flocalrestrictionproperty}
F(\psi)|_{V} = F|_{V}(\psi|_{V}),
\end{equation}
for all $\psi \in \Gamma_{\E}(U)$. There is a canonical graded $\C^{\infty}_{\M}(U)$-module structure on $\Loc_{\E}(U)$. With respect to aforementioned restrictions, the assignment $U \mapsto \Loc_{\E}(U)$ makes $\Loc_{\E}$ into a sheaf of graded $\C^{\infty}_{\M}$-modules, called the \textbf{sheaf of local operators on $\E$}. 
\end{tvrz}
\begin{proof}
Let $\psi \in \Gamma_{\E}(V)$. For each $m \in V$, find $V_{(m)} \in \Op_{m}(V)$ satisfying $\ol{V}_{(m)} \subseteq V$  and use the extension lemma to find $\psi'_{(m)} \in \Gamma_{\E}(U)$ satisfying $\psi|_{V_{(m)}} = \psi'_{(m)}|_{V_{(m)}}$. Since $\{ V_{(m)} \}_{m \in V}$ is an open cover of $V$ and $\Gamma_{\E}$ is a sheaf, the formula
\begin{equation}
F|_{V}(\psi)|_{V_{(m)}} := F(\psi'_{(m)})|_{V_{(m)}},
\end{equation}
imposed for every $m \in V$ defines a unique section $F|_{V}(\psi)$ of $\Gamma_{\E}(V)$, if the elements of $\Gamma_{\E}(V_{(m)})$ on the right-hand side agree on the overlaps. This follows from the fact that $F$ is $\R$-linear and local. Similarly, one can prove that the definition of $F|_{V}$ depends neither on the open cover $\{ V_{(m)} \}_{m \in V}$ or extensions $\psi'_{(m)}$ used. $F|_{V}$ is then easily shown to be a graded $\R$-linear map satisfying (\ref{eq_Flocalrestrictionproperty}). 

Let us argue that $F|_{V}$ is local. Suppose that there is $W \in \Op(V)$ and $\psi \in \Gamma_{\E}(V)$ satisfying $\psi|_{W} = 0$. In the construction of the open cover $\{ V_{(m)} \}_{m \in V}$, we can always assume that for $m \in W$, one has $V_{(m)} \subseteq W$. But then $\psi'_{(m)}|_{V_{(m)}} = 0$ and thus $F(\psi'_{(m)})|_{V_{(m)}} = 0$ for every $m \in W$, since $F$ is local. Since $\{ V_{(m)} \}_{m \in W}$ is an open cover of $W$, we find $F|_{V}(\psi)|_{W} = 0$. Hence $F|_{V}$ is local. 

To see that $F|_{V}$ is a unique local operator satisfying (\ref{eq_Flocalrestrictionproperty}), suppose that $G \in \Loc_{\E}(V)$ is another such operator. Let $\psi \in \Gamma_{\E}(V)$ be arbitrary and $\{ V_{(m)} \}_{m \in V}$ and $\psi'_{(m)} \in \Gamma_{\E}(U)$ be constructed as above. Since $G$ is local and satisfies (\ref{eq_Flocalrestrictionproperty}), one can for each $m \in V$ write
\begin{equation}
G(\psi)|_{V_{(m)}} = G(\psi_{(m)}|_{V})|_{V_{(m)}} = F(\psi_{(m)})|_{V_{(m)}} \equiv F|_{V}(\psi)|_{V_{(m)}}. 
\end{equation}
As $\{ V_{(m)} \}_{m \in V}$ covers $V$ and $\psi$ was arbitrary, this proves that $G = F|_{V}$. 

Let $U \in \Op(M)$. For each $f \in \C^{\infty}_{\M}(U)$, $F \in \Loc_{\E}(U)$ and $\psi \in \Gamma_{\E}(U)$, the formula
\begin{equation}
(f \cdot F)(\psi) := f \cdot F(\psi)
\end{equation}
clearly makes $\Loc_{\E}(U)$ into a graded $\C^{\infty}_{\M}(U)$-module. The unique property (\ref{eq_Flocalrestrictionproperty}) of restrictions can be now utilized to prove that for any open subsets $W \subseteq V \subseteq U$, every $F \in \Loc_{\E}(U)$ and any $f \in \C^{\infty}_{\M}(U)$, one has 
\begin{equation}
(F|_{V})|_{W} = F|_{W}, \; \; F|_{U} = F, \; \; (f \cdot F)|_{V} = f|_{V} \cdot F|_{V}. 
\end{equation}
This proves that $\Loc_{\E}$ is a presheaf of graded $\C^{\infty}_{\M}$-modules. Finally, suppose that $U \in \Op(M)$ and let $\{ U_{\alpha} \}_{\alpha \in I}$ be some open cover of $U$. Let $\{ F_{\alpha} \}_{\alpha \in I}$ be a collection of local operators on $\E$ over $U_{\alpha}$ agreeing on the overlaps. For each $\psi \in \Gamma_{\E}(U)$ and every $\alpha \in I$, define
\begin{equation}
F(\psi)|_{U_{\alpha}} := F_{\alpha}(\psi|_{U_{\alpha}}). 
\end{equation}
By using the assumption and (\ref{eq_Flocalrestrictionproperty}), the elements of $\Gamma_{\E}(U_{\alpha})$ agree on the overlaps, whence they glue to a unique element $\Gamma_{\E}(U)$. It is not difficult to see that this procedure defines an element of $\Loc_{\E}(U)$. Its defining formula shows that $F|_{U_{\alpha}} = F_{\alpha}$, for all $\alpha \in I$, and such $F$ is unique. This proves that $\Loc_{\E}$ is a sheaf of graded $\C^{\infty}_{\M}$-modules.
\end{proof}

\begin{example} \label{ex_CMlineararelocal}
Every $\C^{\infty}_{\M}(U)$-linear map $F: \Gamma_{\E}(U) \rightarrow \Gamma_{\E}(U)$ of any given degree $|F|$ is local. Indeed, suppose there is $V \in \Op(U)$ and $\psi \in \Gamma_{\E}(U)$ satisfying $\psi|_{V} = 0$. For any $m \in V$, find $V_{(m)} \in \Op_{m}(V)$ such that $\ol{V}_{(m)} \subseteq V$, and a smooth bump function $\eta \in \C^{\infty}_{\M}(U)$ satisfying $\supp(\eta) \subseteq V$ and $\eta|_{V_{(m)}} = 1$. It follows that $\psi = (1 - \eta) \cdot \psi$. But then
\begin{equation}
F(\psi)|_{V_{(m)}} = (1 - \eta)|_{V_{(m)}} \cdot F(\psi)|_{V_{(m)}} = 0,
\end{equation}
and since $\{ V_{(m)} \}_{m \in V}$ is an open cover of $V$, this proves that $F(\psi)|_{V} = 0$. Whence $F \in \Loc_{\E}(U)$. 
\end{example}
There is one important feature of local operators which follows from the property (\ref{eq_Flocalrestrictionproperty}). When we talk about sheaves, we usually talk about sheaf morphisms between them. It turns out that for graded vector bundles, it suffices to define a mapping of global sections. 
\begin{tvrz} \label{tvrz_Locglobalsheaf}
Let $F \in \Loc_{\E}(M)$. Then there exists a unique sheaf morphism $\ol{F}: \Gamma_{\E} \rightarrow \Gamma_{\E}$, such that $F = \ol{F}_{M}$. If $F$ is $\C^{\infty}_{\M}(M)$-linear, $\ol{F}$ is a $\C^{\infty}_{\M}$-linear sheaf morphism. We will usually abuse the notation and omit the bar over $F$. 
\end{tvrz}
\begin{proof}
Recall that a sheaf morphism is just a natural transformation of the two sheaf functors (possibly with a degree shift). $\ol{F}$ is unique, since for every $U \in \Op(M)$ and every $\psi \in \Gamma_{\E}(M)$, the naturality forces $\ol{F}_{U}(\psi|_{U}) = F(\psi)|_{U}$, that is $\ol{F}_{U} := F|_{U}$ by Proposition \ref{tvrz_locsheaf}. Since $\Loc_{\E}$ is a presheaf, $\ol{F}_{U}$ is natural in $U$. We thus obtain a sheaf morphism $\ol{F} = \{ \ol{F}_{U} \}_{U \in \Op(M)}$. It is not difficult to see from the proof of Proposition \ref{tvrz_locsheaf} that if $F$ is $\C^{\infty}_{\M}(M)$-linear, then $F|_{U}$ is $\C^{\infty}_{\M}(U)$-linear for every $U \in \Op(M)$, hence $\ol{F}$ is a $\C^{\infty}_{\M}$-linear sheaf morphism. 
\end{proof}
\section{Sheaves of differential operators} \label{sec_diffop}
In this section, we will produce a new class of local operators generalizing Example \ref{ex_CMlineararelocal}. We mimic standard definitions and prove that even in the graded case, we get expected properties. Again, $\E$ is still assumed to be a general graded vector bundle over a general graded manifold $\M$. 

For every $f \in \C^{\infty}_{\M}(U)$, one has a $\C^{\infty}_{\M}(U)$-linear operator $\lambda_{f}: \Gamma_{\E}(U) \rightarrow \Gamma_{\E}(U)$, defined as
\begin{equation}
\lambda_{f}(\psi) := f \cdot \psi,
\end{equation}
for all $\psi \in \Gamma_{\E}(U)$. Note that $|\lambda_{f}| = |f|$ and $\lambda_{f}|_{V} = \lambda_{f|_{V}}$ for any $V \in \Op(U)$. 

\begin{definice}
Let $k \in \N_{0}$ be a non-negative integer. A graded $\R$-linear map $D: \Gamma_{\E}(U) \rightarrow \Gamma_{\E}(U)$ is called a \textbf{$k$-th order differential operator on $\E$ over $U$}, if 
\begin{enumerate}[(i)]
\item it is $\C^{\infty}_{\M}(U)$-linear (when $k = 0$); 
\item for every $f \in \C^{\infty}_{\M}(U)$, the graded commutator $[D,\lambda_{f}] := D \circ \lambda_{f} - (-1)^{|D||f|} \lambda_{f} \circ D$ is a $(k-1)$-th order differential operator on $\E$ over $U$ (when $k > 0$).
\end{enumerate}
We denote the graded vector space of $k$-th order differential operators on $\E$ over $U$ as $\Dif_{\E}^{k}(U)$. 
\end{definice}
Let us introduce the following notation. For every graded $\R$-linear map $D$, every $j \in \N$ and every $j$-tuple $(f_{1}, \dots, f_{j})$ of functions from $\C^{\infty}_{\M}(U)$, we define a pair of graded $\R$-linear maps
\begin{equation} \label{eq_Djiteratedmaps}
D_{(f_{1},\dots,f_{j})}^{(j)} := [ \cdots [[D,\lambda_{f_{1}}], \lambda_{f_{2}}], \dots ],\lambda_{f_{j}}], \; \; \ol{D}^{(j)}_{(f_{1},\dots,f_{j})} := [\lambda_{f_{1}},[\dots, [\lambda_{f_{j-1}}, [\lambda_{f_{j}}, D]] \cdots ]. 
\end{equation}
They differ only by a sign but it turns out that it is convenient to keep both definitions. More explicitly, one finds the relation
\begin{equation} \label{eq_DjvshatDj}
D^{(j)}_{(f_{1},\dots,f_{j})} = (-1)^{j + |D|(|f_{1}| + \dots + |f_{j}|)} \ol{D}^{(j)}_{(f_{1},\dots,f_{j})}. 
\end{equation}
\begin{rem} \label{rem_localusingDkplus1}
It is easy to see that $D \in \Dif^{k}_{\E}(U)$, iff $D^{(k+1)}_{(f_{1},\dots,f_{k+1})} = 0$ for all $f_{1}, \dots, f_{k+1} \in \C^{\infty}_{\M}(U)$. 
\end{rem}
It turns out that differential operators are local. In fact, they form sections of a sheaf of graded $\C^{\infty}_{\M}$-submodules of the sheaf of local operators. 
\begin{tvrz}
For each $k \in \N_{0}$ and $U \in \Op(M)$, one has $\Dif^{k}_{\E}(U) \subseteq \Loc_{\E}(U)$. In fact, one can view $\Dif^{k}_{\E}$ as a sheaf of graded $\C^{\infty}_{\M}$-submodules of the sheaf $\Loc_{\E}$. $\Dif^{k}_{\E}$ is called the \textbf{sheaf of $k$-th order differential operators} on $\E$. 
\end{tvrz}
\begin{proof}
Let us proceed by induction on $k \in \N_{0}$. For $k = 0$, every $F \in \Dif^{0}_{\E}(U)$ is $\C^{\infty}_{\M}(U)$-linear, hence $\Dif^{0}_{\E}(U) \subseteq \Loc_{\E}(U)$ as proved in Example \ref{ex_CMlineararelocal}. Hence assume that $k > 0$ and that the statement holds for all lower order differential operators. Let $D \in \Dif^{k}_{\E}(U)$ and suppose there is $V \in \Op(U)$ and $\psi \in \Gamma_{\E}(V)$ satisfying $\psi|_{V} = 0$. For each $m \in V$, find $V_{(m)} \in \Op(V)$ and a smooth bump function $\eta$ as in Example \ref{ex_CMlineararelocal}, so one can write $\psi = (1 - \eta) \cdot \psi$. But then 
\begin{equation}
D(\psi)|_{V_{(m)}} = [D, \lambda_{1 - \eta}](\psi)|_{V_{(m)}} + (1 - \eta)|_{V_{(m)}} \cdot D(\psi)|_{V_{(m)}} = 0,
\end{equation}
where we have used the induction hypothesis and the definition of $\eta$. Since $\{ V_{(m)} \}_{m \in V}$ covers $V$, we see that $D(\psi)|_{V} = 0$ and conclude that $D \in \Loc_{\E}(U)$. 

To show that $\Dif^{k}_{\E}$ forms a sheaf of graded $\C^{\infty}_{\M}$-submodules, we must first prove $\Dif_{\E}^{k}(U)$ is a graded $\C^{\infty}_{\M}(U)$-submodule of $\Loc_{\E}(U)$. For any $f,g \in \C^{\infty}_{\M}(U)$ and $D \in \Dif_{\E}^{k}(U)$, one has 
\begin{equation}
[f \cdot D, \lambda_{g}] = [\lambda_{f} \circ D, \lambda_{g}] = \lambda_{f} \circ [D, \lambda_{g}] + (-1)^{||D||g|} [\lambda_{f},\lambda_{g}] \circ D = \lambda_{f} \circ [D,\lambda_{g}] = f \cdot D^{(1)}_{(g)},
\end{equation}
where we have used the graded Jacobi identity for the graded commutator and the fact that $[\lambda_{f},\lambda_{g}] = 0$ for all $f,g \in \C^{\infty}_{\M}(U)$. By iterating this equation, one finds the formula
\begin{equation} \label{eq_iteratedcommislinear}
(f \cdot D)^{(k+1)}_{(f_{1},\dots,f_{k+1})} = f \cdot D^{(k+1)}_{(f_{1},\dots,f_{k+1})},
\end{equation}
for all $f_{1}, \dots, f_{k+1} \in \C^{\infty}_{\M}(U)$. Using the Remark \ref{rem_localusingDkplus1} twice and the assumption on $D$, we see that $f \cdot D \in \Dif^{k}_{\E}(U)$. Hence $\Dif^{k}_{\E}(U)$ is a graded $\C^{\infty}_{\M}(U)$-submodule of $\Loc_{\E}(U)$. 

To prove that it forms a sheaf, we will proceed once more by induction on $k \in \N_{0}$. Let $U \in \Op(M)$ and $\{ U_{\alpha} \}_{\alpha \in I}$ be some its open cover. Suppose $D \in \Loc_{\E}(U)$ satisfies $D|_{U_{\alpha}} \in \Dif^{0}_{\E}(U_{\alpha})$. For any $f \in \C^{\infty}_{\M}(U)$ and $\psi \in \Gamma_{\E}(U)$, one can use (\ref{eq_Flocalrestrictionproperty}) and the assumption to write 
\begin{equation}
\begin{split}
D(f \cdot \psi)|_{U_{\alpha}} = & \ D|_{U_{\alpha}}( f|_{U_{\alpha}} \cdot \psi|_{U_{\alpha}}) = (-1)^{|D||f|} f|_{U_{\alpha}} \cdot D|_{U_{\alpha}}(\psi|_{U_{\alpha}}) \\
= & \ (-1)^{|D||f|} (f \cdot D(\psi))|_{U_{\alpha}}. 
\end{split}
\end{equation}
Since $\Gamma_{\E}$ is a sheaf, this proves that $D(f \cdot \psi) = (-1)^{|D||f|} f \cdot D(\psi)$, that is $D \in \Dif^{0}_{\E}(U)$. This proves that $\Dif^{0}_{\E}$ is a subsheaf of graded $\C^{\infty}_{\M}$-submodules. 

Next, assume that $k > 0$ and that the differential operators form sheaf of graded $\C^{\infty}_{\M}$-submodules for all orders lower then $k$. Suppose that $D \in \Loc_{\E}(U)$ satisfies $D|_{U_{\alpha}} \in \Dif^{k}_{\E}(U_{\alpha})$ for some $U \in \Op(M)$ and its open cover $\{ U_{\alpha} \}_{\alpha \in I}$. Let $f \in \C^{\infty}_{\M}(U)$ be an arbitrary function. Then
\begin{equation}
[D,\lambda_{f}]|_{U_{\alpha}} = [D|_{U_{\alpha}}, \lambda_{f|_{U_{\alpha}}}] \in \Dif^{k-1}_{\E}(U_{\alpha}). 
\end{equation}
But then $[D,\lambda_{f}] \in \Dif^{k-1}_{\E}(U)$ by the induction hypothesis. Since $f$ was arbitrary, this proves that $D \in \Dif^{k}_{\E}(U)$. Hence $\Dif^{k}_{\E}$ is a sheaf of graded $\C^{\infty}_{\M}$-submodules of $\Loc_{\E}$. 
\end{proof}
To conclude this section, we have to derive some important properties of the maps (\ref{eq_Djiteratedmaps}). 
\begin{lemma} \label{lem_olDjprops}
For every graded $\R$-linear map $D: \Gamma_{\E}(U) \rightarrow \Gamma_{\E}(U)$ and every $j \in \N$, one has 
\begin{equation} \label{eq_olDjgsymmetry}
\ol{D}^{(j)}_{(f_{1}, \dots, f_{i},f_{i+1}, \dots, f_{j})} = (-1)^{|f_{i}||f_{i+1}|} \ol{D}^{(j)}_{(f_{1},\dots,f_{i+1},f_{i},\dots,f_{j})}, 
\end{equation}
for every $j$-tuple $(f_{1}, \dots, f_{j})$ of functions in $\C^{\infty}_{\M}(U)$ and $i \in \{1, \dots, j-1\}$. Moreover, for any $f,g \in \C^{\infty}_{\M}(U)$ and any $(j-1)$-tuple $(f_{2},\dots,f_{j})$ of functions in $\C^{\infty}_{\M}(U)$, one has
\begin{equation} \label{eq_olDjgLeibniz}
\ol{D}^{(j)}_{(f \cdot g, f_{2}, \dots, f_{j})} = f \cdot \ol{D}^{(j)}_{(g,f_{2},\dots,f_{j})} + (-1)^{|f||g|} g \cdot \ol{D}^{(j)}_{(f,f_{2},\dots,f_{j})} - \ol{D}^{(j+1)}_{(f,g,f_{2},\dots,f_{j})}.
\end{equation} 
Similar equations can be derived for the maps $D^{(j)}_{(f_{1},\dots,f_{j})}$. 
\end{lemma}
\begin{proof}
The property (\ref{eq_olDjgsymmetry}) can be easily derived using definitions, the graded Jacobi identity for the graded commutator, and the fact that $[\lambda_{f},\lambda_{g}] = 0$ for all $f,g \in \C^{\infty}_{\M}(U)$. The property (\ref{eq_olDjgLeibniz}) follows from definitions, the graded Leibniz rule for the graded commutator, and the fact that $\lambda_{f \cdot g} = \lambda_{f} \circ \lambda_{g}$. We leave the detailed proof as an exercise. 
\end{proof}
\section{Graded vector bundles of differential operators} \label{sec_diffopVB}
In this section, we will prove that $\Dif^{k}_{\E}$ is locally freely and finitely generated, hence a sheaf of sections of a graded vector bundle. To do so, we must derive local expressions for differential operators. Let us first establish some notation. 

Suppose that $(U,\varphi)$ is a graded local chart for $\M$ inducing a collection of local coordinate functions $\{ \bbz^{A} \}_{A=1}^{n} \subseteq \C^{\infty}_{\M}(U)$. It is not really important how they are ordered in what follows. We will need some important subsets of sets all $n$-indices of non-negative integers. Define
\begin{equation}
\ol{\N}{}^{n} := \{ \fI \equiv (i_{1}, \dots, i_{n}) \in (\N_{0})^{n} \; | \; i_{A} \in \{0,1\} \text{ if $|\bbz^{A}|$ is odd} \}.
\end{equation}
This ensures that for each $\fI \in \ol{\N}{}^{n}$, the monomial $\bbz^{\fI} := (\bbz^{1})^{i_{1}} \dots (\bbz^{n})^{i_{n}}$ does not vanish. 
For each $j \in \N_{0}$, we sometimes need to make sure that $\bbz^{\fI}$ is the product of $j$ coordinate functions:
\begin{equation} \label{eq_onN(j)}
\ol{\N}{}^{n}(j) := \{ \fI \in \ol{\N}{}^{n} \; | \; w(\fI) := \sum_{A=1}^{n} i_{A} = j \}. 
\end{equation}
Note that the set $\ol{\N}{}^{n}(j)$ is finite. We write $\fnula = (0, \dots, 0)$ for the $n$-index of all zeros. Let us also use the shorthand notation
\begin{equation} \label{eq_zImultiindex}
(\bbz^{\fI}_{(k)}) := (\underbrace{\bbz^{1}, \dots, \bbz^{1}}_{i_{1} \times}, \dots, \underbrace{\bbz^{n}, \dots, \bbz^{n}}_{i_{n} \times}).
\end{equation}

Next, we need the operator of the ``partial derivative acting from the right''. For each $A \in \{1, \dots, n\}$, define a degree $-|\bbz^{A}|$ graded $\R$-linear operator on $\C^{\infty}_{\M}(U)$ by the formula
\begin{equation} \label{eq_partialtl}
\partial^{\tl}_{A}(f) := (-1)^{|\bbz^{A}|(1 + |f|)} \frac{\partial f}{\partial \bbz^{A}}, 
\end{equation}
for all $f \in \C^{\infty}_{\M}(U)$, where $\frac{\partial f}{\partial \bbz^{A}}$ denotes the usual action of the coordinate vector field $\frac{\partial}{\partial \bbz^{A}} \in \X_{\M}(U)$ on $f$. These operators are designed to satisfy the graded Leibniz rule in the form
\begin{equation} \label{eq_partialtlLeibniz}
\partial^{\tl}_{A}(f \cdot g) = f \cdot \partial^{\tl}_{A}(g) + (-1)^{|\bbz^{A}||g|} \partial^{\tl}_{A}(f) \cdot g,
\end{equation}
for any $f,g \in \C^{\infty}_{\M}(U)$. Observe that $\partial^{\tl}_{A}(\bbz^{B}) = \delta^{B}_{A}$. More generally, we will write $\partial^{\tl}_{A_{1} \dots A_{j}} := \partial^{\tl}_{A_{1}} \circ \cdots \circ \partial^{\tl}_{A_{j}}$ for any $A_{1},\dots,A_{j} \in \{1, \dots, n\}$. 

Finally, suppose that $\{ \Phi_{\lambda} \}_{\lambda=1}^{r}$ is a local frame for $\E$ over $U$ and let $D: \Gamma_{\E}(U) \rightarrow \Gamma_{\E}(U)$ be an $\R$-linear operator of degree $|D|$. For each $j \in \N$ and each $j$-tuple $(f_{1},\dots,f_{j})$ of functions in $\C^{\infty}_{\M}(U)$, we can for each $\mu,\lambda \in \{1, \dots, r\}$ define the functions $[\ol{\fD}^{(j)}_{(f_{1},\dots,f_{j})}]^{\mu}{}_{\lambda}$ using the operators defined by (\ref{eq_Djiteratedmaps}), and the formula
\begin{equation} \label{eq_olfDdef}
\ol{D}^{(j)}_{(f_{1},\dots,f_{j})}( \Phi_{\lambda}) =: [\ol{\fD}^{(j)}_{(f_{1},\dots,f_{j})}]^{\mu}{}_{\lambda} \cdot \Phi_{\mu}. 
\end{equation}
Note that $| [\ol{\fD}^{(j)}_{(f_{1},\dots,f_{j})}]^{\mu}{}_{\lambda}| = |D| + |f_{1}| + \dots + |f_{j}| + |\vartheta_{\lambda}| - |\vartheta_{\mu}|$, where $\{ \vartheta_{\lambda} \}_{\lambda=1}^{r}$ is some fixed basis of the typical fiber $K$ of $\E$, such that $|\Phi_{\lambda}| = |\vartheta_{\lambda}|$ for each $\lambda \in \{1,\dots,r\}$. We will now argue that for $D \in \Dif^{k}_{\E}(U)$, its action on a general section can be fully expressed in terms of these functions. The crucial observation is the following consequence of the graded Leibniz rule (\ref{eq_olDjgLeibniz}).

\begin{tvrz} \label{tvrz_olfDformula}
Let $D \in \Dif^{k}_{\E}(U)$. Then for each $j \in \{0, \dots, k-1\}$, one has the formula
\begin{equation} \label{eq_olfDformula}
[\ol{\fD}^{(k-j)}_{(f,g_{2},\dots,g_{k-j})}]^{\mu}{}_{\lambda} = \sum_{q=1}^{j+1} (-1)^{q+1} \frac{1}{q!} (\partial^{\tl}_{A_{1} \dots A_{q}} f) \cdot [\ol{\fD}^{(q+k-j-1)}_{(\bbz^{A_{1}}, \dots, \bbz^{A_{q}}, g_{2}, \dots, g_{k-j})}]^{\mu}{}_{\lambda}, 
\end{equation}
for all $f,g_{2},\dots,g_{k-j} \in \C^{\infty}_{\M}(U)$ and $\mu,\lambda \in \{1, \dots, r\}$. 
\end{tvrz}
\begin{proof}
Let us establish some notation first. Since $(g_{2},\dots,g_{k-j})$ does not play any role in the formula (\ref{eq_olfDformula}), we will replace it with a symbol $\star$ for the remainder of the proof. We thus aim to prove that for any $j \in \{0,\dots,k-1\}$, $f \in \C^{\infty}_{\M}(U)$, $\mu,\lambda \in \{1,\dots,r\}$ and any $\star$, one has
\begin{equation} \label{eq_olfDformula2}
[\ol{\fD}^{(k-j)}_{(f,\star)}]^{\mu}{}_{\lambda} = [\hat{\fD}^{(k-j)}_{(f,\star)}]^{\mu}{}_{\lambda},
\end{equation}
where we use $[\hat{\fD}^{(k-j)}_{(f,\star)}]^{\mu}{}_{\lambda}$ to denote the right-hand side of (\ref{eq_olfDformula}). Next, observe that (\ref{eq_olDjgLeibniz}) immediately implies the graded Leibniz rule in the form 
\begin{equation} \label{eq_olfDjgLeibniz}
[\ol{\fD}^{(k-j)}_{(f \cdot g, \star)}]^{\mu}{}_{\lambda} = f \cdot [\ol{\fD}^{(k-j)}_{(g,\star)}]^{\mu}{}_{\lambda} + (-1)^{|f||g|} g \cdot [\ol{\fD}^{(k-j)}_{(f,\star)}]^{\mu}{}_{\lambda} - [\ol{\fD}^{(k-j+1)}_{(f,g,\star)}]^{\mu}{}_{\lambda}.
\end{equation}
Let us now prove (\ref{eq_olfDformula2}) by induction in $j$. For $j = 0$, the last term in (\ref{eq_olfDjgLeibniz}) vanishes as $D \in \Dif^{k}_{\E}(U)$. This also immediately implies that 
\begin{equation} \label{eq_olfDkonconstant}
[\ol{\fD}^{(k)}_{(1, \star)}]^{\mu}{}_{\lambda} = 0.
\end{equation}
Let $a \in U$ be arbitrary and write $\bbz^{A}_{a} := \bbz^{A} - \bbz^{A}(a)$ for every $A \in \{1,\dots,n\}$. Note that $\bbz^{A}_{a} = \bbz^{A}$ whenever $|\bbz^{A}| \neq 0$. We claim that (\ref{eq_olfDformula2}) holds for all monomials in the coordinates $\{ \bbz^{A}_{a} \}_{A=1}^{n}$ of any order $q \in \N_{0}$ and for arbitrary $a \in U$. This can be proved by a simple induction in $q$. We will provide more details in the $j > 0$ case and thus skip the detailed discussion here. 

Next, recall that for every $a \in U$, there is a graded ideal $\J^{a}_{\M}(U) \subseteq \C^{\infty}_{\M}(U)$ of functions vanishing at $a \in U$. It is generated by the graded set $\{ \bbz^{A}_{a} \}_{A=1}^{n}$. Let $f \in \C^{\infty}_{\M}(U)$. One can show that
\begin{equation} \label{eq_fzerocondition}
f = 0 \Leftrightarrow f \in \bigcap_{q \in \N} \bigcap_{a \in U} (\J^{a}_{\M}(U))^{q},
\end{equation}
see Proposition 3.5 in \cite{Vysoky:2022gm}. It follows from (\ref{eq_olfDjgLeibniz}) that if $f \in (\J^{a}_{\M}(U))^{q+1}$, then 
\begin{equation}
[\ol{\fD}^{(k)}_{(f,\star)}]^{\mu}{}_{\lambda} \in (\J^{a}_{\M}(U))^{q},
\end{equation}
and the same observation is valid also for $[\hat{\fD}^{(k)}_{(f,\star)}]^{\mu}{}_{\lambda}$. Let us finish the proof of (\ref{eq_olfDformula}) for $j = 0$ and any $f \in \C^{\infty}_{\M}(U)$. Choose any $q \in \N$ and $a \in U$. One can write $f = T^{q}_{a}(f) + R^{q}_{a}(f)$, where $T^{q}_{a}(f)$ is a Taylor polynomial of $f$ of order $q$ and $R^{q}_{a}(f) \in (\J^{a}_{\M}(U))^{q+1}$, see Lemma 3.4 in \cite{Vysoky:2022gm}. Since $T^{q}_{a}(f)$ is a polynomial of order $q$ in $\{ \bbz^{A}_{a} \}_{A=1}^{n}$ and (\ref{eq_olfDformula2}) is $\R$-linear in $f$, we conclude that 
\begin{equation}
[\ol{\fD}^{(k)}_{(f,\star)}]^{\mu}{}_{\lambda} - [\hat{\fD}^{(k)}_{(f,\star)}]^{\mu}{}_{\lambda} = [\ol{\fD}^{(k)}_{(R^{q}_{a}(f),\star)}]^{\mu}{}_{\lambda} - [\hat{\fD}^{(k)}_{(R^{q}_{a}(f),\star)}]^{\mu}{}_{\lambda} \in (\J^{a}_{\M}(U))^{q}. 
\end{equation}
Since $q \in \N$ and $a \in U$ were arbitrary, the left-hand side must vanish due to (\ref{eq_fzerocondition}), and we conclude that (\ref{eq_olfDformula2}) holds for $j = 0$. 

Next, let us assume that $j > 0$ and the formula (\ref{eq_olfDformula2}) holds for all $j' \in \{0,\dots,j-1\}$. The graded Leibniz rule (\ref{eq_olfDjgLeibniz}) together with the induction hypothesis imply 
\begin{equation} \label{eq_olfDkonconstant2}
[\ol{\fD}^{(k-j)}_{(1, \star)}]^{\mu}{}_{\lambda} = 0.
\end{equation}
Let us prove (\ref{eq_olfDformula2}) for $f = \bbz^{B_{1}}_{a} \cdots \bbz^{B_{q}}_{a}$, where $a \in U$, $q \in \N_{0}$ and $B_{1},\dots,B_{q} \in \{1, \dots, n\}$ are arbitrary. Let us proceed by induction in $q$. For $q = 0$, this reduces to (\ref{eq_olfDkonconstant2}). It now takes a bit of combinatorics together with definitions and (\ref{eq_partialtlLeibniz}) to verify that 
\begin{equation}
\begin{split}
[\hat{\fD}^{(k-j)}_{(\bbz^{B_{1}}_{a} \cdots \bbz^{B_{q}}_{a},\star)}]^{\mu}{}_{\lambda} = & \ ( \bbz^{B_{1}} \cdots \bbz^{B_{q-1}}) \cdot [\hat{\fD}^{(k-j)}_{(\bbz^{B_{q}}_{a},\star)}]^{\mu}{}_{\lambda} \\
& + (-1)^{(|\bbz^{B_{1}}| + \dots + |\bbz^{B_{q-1}}|)|\bbz^{B_{q}}|} \bbz^{B_{q}}_{a} \cdot [\hat{\fD}^{(k-j)}_{(\bbz_{a}^{B_{1}} \cdots \bbz_{a}^{B_{q-1}}, \star)}]^{\mu}{}_{\lambda} \\
& - [\hat{\fD}^{(k-j+1)}_{(\bbz^{B_{1}}_{a} \cdots \bbz^{B_{q-1}}_{a}, \bbz^{B_{q}}, \star)}]^{\mu}{}_{\lambda}.
\end{split}
\end{equation}
In the first two terms, one can replace $\hat{\fD}$ by $\ol{\fD}$ by the induction hypothesis in $q$. In the last term, one can replace $\hat{\fD}$ by $\ol{\fD}$ by the induction hypothesis in $j$. But then one can then use (\ref{eq_olDjgLeibniz}) to obtain
\begin{equation}
[\hat{\fD}^{(k-j)}_{(\bbz^{B_{1}}_{a} \cdots \bbz^{B_{q}}_{a},\star)}]^{\mu}{}_{\lambda} = [\ol{\fD}^{(k-j)}_{(\bbz^{B_{1}}_{a} \cdots \bbz^{B_{q}}_{a},\star)}]^{\mu}{}_{\lambda}.
\end{equation}
This finishes the induction step in $q$ and thus the proof of (\ref{eq_olfDformula2}) for monomials in $\{ \bbz^{A}_{a} \}_{A=1}^{n}$ for any order $q \in \N_{0}$ and $a \in A$. The rest goes similarly to the $j = 0$ case. The graded Leibniz rule (\ref{eq_olDjgLeibniz}) together with the induction hypothesis imply that for any $q \in \N$, $a \in U$ and $f \in (\J^{a}_{\M}(U))^{q + 1 + j}$, one has 
\begin{equation}
[\ol{\fD}^{(k-j)}_{(f,\star)}]^{\mu}{}_{\lambda} \in (\J^{a}_{\M}(U))^{q},
\end{equation}
and the same is true for $[\hat{\fD}^{(k-j)}_{(f,\star)}]^{\mu}{}_{\lambda}$. For any $f \in \C^{\infty}_{\M}(U)$, one can then choose any $q \in \N$, $a \in U$ and write $f = T^{q+j}_{a}(f) + R^{q+j}_{a}(f)$. The same trick as for $j = 0$ case is then utilized to prove that (\ref{eq_olDjgLeibniz}) holds for any $f \in \C^{\infty}_{\M}(U)$. This finishes the induction step and the proof is finished. 
\end{proof}
As a simple consequence, one obtains a local expression for every differential operator. 
\begin{cor} \label{cor_Dactingformula1}
Let $D \in \Dif^{k}_{\E}(U)$. Write $D(\Phi_{\lambda}) =: \fD^{\mu}{}_{\lambda} \cdot \Phi_{\mu}$. Then the action of $D$ on a general section $\psi = \psi^{\lambda} \cdot \Phi_{\lambda} \in \Gamma_{\E}(U)$ can be written in terms of functions (\ref{eq_olfDdef}) as 
\begin{equation} \label{eq_Dactingformula1}
D(\psi) = \{ (-1)^{|D||\psi^{\lambda}|} \psi^{\lambda} \cdot \fD^{\mu}{}_{\lambda} + \sum_{q=1}^{k} (-1)^{q + |D||\psi^{\lambda}|} \frac{1}{q!} (\partial^{\tl}_{A_{1} \dots A_{q}} \psi^{\lambda}) \cdot [\ol{\fD}^{(q)}_{(\bbz^{A_{1}}, \dots, \bbz^{A_{q}})}]^{\mu}{}_{\lambda} \} \cdot \Phi_{\mu}.
\end{equation}
\begin{proof}
Starting from the left-hand side, one can write
\begin{equation}
\begin{split}
D(\psi) = & \ (D \circ \lambda_{\psi^{\lambda}})(\Phi_{\lambda}) = (-1)^{|D||\psi^{\lambda}|} (\lambda_{\psi^{\lambda}} \circ D - \ol{D}^{(1)}_{(\psi^{\lambda})})(\Phi_{\lambda}) \\
= & \ \{ (-1)^{|D||\psi^{\lambda}|} \psi^{\lambda} \cdot \fD^{\mu}{}_{\lambda} - [\ol{\fD}^{(1)}_{(\psi^{\lambda})}]^{\mu}{}_{\lambda} \} \cdot \Phi_{\mu}. 
\end{split}
\end{equation}
The rest is just the formula (\ref{eq_olfDformula}) for $j = k-1$ and $f = \psi^{\lambda}$. 
\end{proof}
\end{cor}
In the following, we will need the following technical lemma.
\begin{lemma}
Let $\fI \in \ol{\N}{}^{n}$ be arbitrary. Let $\partial^{\op}_{\fI}$ be an operator on functions defined as 
\begin{equation}
\partial^{\op}_{\fI} := (\partial_{n})^{i_{n}} \circ \dots \circ (\partial_{1})^{i_{1}},
\end{equation}
where $\partial_{A}(f) := \frac{\partial f}{\partial \bbz^{A}}$ for each $A \in \{1,\dots,n\}$. The superscript $\op$ just indicates that the partial derivatives act in the order opposite to the $n$-index $\fI$. Then it has the following properties:
\begin{enumerate}[(i)]
\item For every $\fJ \in \ol{\N}{}^{n}$ such that $w(\fJ) \leq w(\fI)$, one has 
\begin{equation} \label{eq_partialIoponzJ}
\partial^{\op}_{\fI}( \bbz^{\fJ}) = \fI! \cdot \delta_{\fI}^{\fJ},
\end{equation}
where $\fI! := i_{1}! \cdots i_{n}!$. In particular, one has $\partial^{\op}_{\fI}(\bbz^{\fJ}) = 0$ whenever $w(\fJ) < w(\fI)$. 
\item For every $f,g \in \C^{\infty}_{\M}(U)$, there holds the Leibniz rule 
\begin{equation} \label{eq_partialIopLeibniz}
\partial_{\fI}^{\op}(f \cdot g) = \sum_{\fK \leq \fI} \binom{\fI}{\fK} (-1)^{\sigma(\fI,\fK) + |f||\bbz^{\fI - \fK}|  }  (\partial_{\fK} f) \cdot (\partial_{\fI - \fK}g),
\end{equation}
where $\binom{\fI}{\fK} := \binom{i_{1}}{k_{1}} \cdots \binom{i_{n}}{k_{n}}$ and
\begin{equation}
\sigma(\fI,\fK) := \sum_{A=2}^{n} (i_{A}-k_{A})|\bbz^{A}|\{ k_{1}|\bbz^{1}| + \cdots + k_{A-1} |\bbz^{A-1}| \}.
\end{equation}
\end{enumerate}
\end{lemma}
\begin{proof}
This is an easy verification. The interpretation of $\sigma(\fI,\fK)$ is the following. For each $A \in \{2,\dots,n\}$, the contribution to the sign is the Koszul sign obtained by commuting $(\partial_{A})^{i_{A} - k_{A}}$ through the operators $(\partial_{A-1})^{k_{A-1}} \circ \cdots \circ (\partial_{1})^{k_{1}}$. We leave the rest to the reader. 
\end{proof}
Corollary \ref{cor_Dactingformula1} is the backbone of the main theorem of this section. Note that in the following, we use the functions $[\fD^{(j)}_{(f_{1},\dots,f_{j})}]^{\mu}{}_{\lambda}$ defined in the same way as in (\ref{eq_olfDdef}), but using the differential operators $D^{(j)}_{(f_{1},\dots,f_{j})}$ instead.
\begin{theorem} \label{thm_decompositionofdiffop}
Let $k \in \N_{0}$. Let $D \in \Dif^{k}_{\E}(U)$ for $U \in \Op(M)$ where we have a graded local chart $(U,\varphi)$ for $\M$ and a local frame $\{ \Phi_{\lambda} \}_{\lambda=1}^{r}$ for $\E$ over $U$. Then $D$ can be uniquely decomposed as 
\begin{equation} \label{eq_Ddecompositionfinal}
D = \sum_{q = 0}^{k} \sum_{\fI \in \ol{\N}{}^{n}(q)} \frac{1}{\fI!} [\fD^{\fI}]^{\mu}{}_{\lambda} \cdot \frP_{\fI}{}^{\lambda}{}_{\mu}, 
\end{equation}
where the functions $[\fD^{\fI}]^{\mu}{}_{\lambda} \in \C^{\infty}_{\M}(U)$ are obtained from $D$ via the formulas
\begin{equation}
[\fD^{\fnula}]^{\mu}{}_{\lambda} := \fD^{\mu}{}_{\lambda}, \; \; [\fD^{\fI}]^{\mu}{}_{\lambda} := [\fD^{(w(\fI))}_{(\bbz^{\fI}_{(w(\fI))})}]^{\mu}{}_{\lambda}.
\end{equation}
$\frP_{\fI}{}^{\lambda}{}_{\mu} \in \Diff^{k}_{\E}(U)$ are differential operators of degree $|\vartheta_{\mu}| - |\vartheta_{\lambda}| - |\bbz^{\fI}|$ defined by 
\begin{equation} \label{eq_operatorsframe}
\frP_{\fI}{}^{\lambda}{}_{\mu}(\psi) = (-1)^{(|\vartheta_{\mu}| - |\vartheta_{\lambda}|)(|\psi^{\lambda}|-|\bbz^{\fI}|)} (\partial^{\op}_{\fI} \psi^{\lambda}) \cdot \Phi_{\mu}, 
\end{equation} 
for all sections $\psi = \psi^{\lambda} \cdot \Phi_{\lambda} \in \Gamma_{\E}(U)$.

Consequently, $\Dif^{k}_{\E}$ forms a sheaf of sections of a \textbf{graded vector bundle $\frD^{k}_{\E}$ of $k$-th order differential operators}. If $( \ell_{j} )_{j \in \Z} := \grk(\frD^{k}_{\E})$, then 
\begin{equation}
\ell_{j} = \# \{ (\fI,\mu,\nu) \in \cup_{q=0}^{k}\ol{\N}{}^{n}(q) \times \{1, \dots, r\}^{2} \; | \; |\vartheta_{\mu}| - |\vartheta_{\lambda}| - |\bbz^{\fI}| = j \}.
\end{equation}
The collection $\{ \frP_{\fI}{}^{\lambda}{}_{\mu} \}$, for all $\fI \in \ol{\N}{}^{n}(q)$ with $q \in \{0,\dots,k\}$ and $\mu,\lambda \in \{1,\dots,r\}$, forms a local frame for $\frD^{k}_{\E}$ over $U$. 
\end{theorem}
\begin{proof}
In a nutshell, the proof of (\ref{eq_Ddecompositionfinal}) consists only of a slight reshuffling of the formula (\ref{eq_Dactingformula1}). One employs (\ref{eq_DjvshatDj}) and the fact that for every $f \in \C^{\infty}_{\M}(U)$ and $A_{1},\dots,A_{q} \in \{1, \dots, n\}$, one has
\begin{equation}
\partial^{\tl}_{A_{1} \dots A_{q}}(f) = (-1)^{(|\bbz^{A_{1}}| + \dots + |\bbz^{A_{q}}|)(|f|+1)} \partial_{A_{q} \dots A_{1}}(f),
\end{equation}
where $\partial_{A_{q} \dots A_{1}} := \partial_{A_{q}} \circ \dots \circ \partial_{A_{1}}$. This allows one to rewrite the equation (\ref{eq_Dactingformula1}) as 
\begin{equation} \label{eq_Ddecompositionamostfinal}
\begin{split}
D(\psi) = & \ \fD^{\mu}{}_{\lambda} \cdot \{ (-1)^{(|\vartheta_{\mu}| - |\vartheta_{\lambda}|)|\psi^{\lambda}|} \psi^{\lambda} \cdot \Phi_{\mu} \} \\
& + \sum_{q=1}^{k} \frac{1}{q!} [\fD^{(q)}_{(\bbz^{A_{1}}, \dots, \bbz^{A_{q}})}]^{\mu}{}_{\lambda} \cdot \{ (-1)^{(|\vartheta_{\mu}| - |\vartheta_{\lambda}|)(|\psi^{\lambda}| - |\bbz^{A_{1}}| \dots  -|\bbz^{A_{q}}|)} \partial_{A_{q} \dots A_{1}}(\psi^{\lambda}) \cdot \Phi_{\mu} \}.
\end{split}
\end{equation}
Next, observe that for each $q \in \N$ and any expression $E_{A_{1} \dots A_{q}}$ completely symmetric in indices $A_{1},\dots,A_{q} \in \{1, \dots, n\}$, one can write
\begin{equation} \label{eq_twosummations}
\sum_{A_{1},\dots,A_{q}} \frac{1}{q!} E_{A_{1} \dots A_{q}} = \sum_{\substack{\fI \in (\N_{0})^{n} \\ w(\fI) = q }} \frac{1}{\fI!} E_{\underbrace{1 \dots 1}_{i_{1} \times }, \dots, \underbrace{n \dots n}_{i_{n} \times }}.
\end{equation} 
In our case, we use the fact that 
\begin{equation}
E_{A_{1} \dots A_{q}} := [\fD^{(q)}_{(\bbz^{A_{1}}, \dots, \bbz^{A_{q}})}]^{\mu}{}_{\lambda} \cdot \{ (-1)^{(|\vartheta_{\mu}| - |\vartheta_{\lambda}|)(|\psi^{\lambda}| - |\bbz^{A_{1}}| \dots  -|\bbz^{A_{q}}|)} \partial_{A_{q} \dots A_{1}}(\psi^{\lambda}) \cdot \Phi_{\mu} \}
\end{equation}
is completely symmetric in $(A_{1},\dots,A_{q})$ thanks to (\ref{eq_olDjgsymmetry}) and the fact that the graded commutator of the coordinate vector fields vanishes. Finally, observe that in this particular case, one only has to consider $\fI \in \ol{\N}{}^{n}(q)$ in the right-hand side sum of (\ref{eq_twosummations}). It follows that (\ref{eq_Ddecompositionamostfinal}) can be written as 
\begin{equation}
D(\psi) = \sum_{q=0}^{k} \sum_{\fI \in \ol{\N}{}^{n}(q)} \frac{1}{\fI!} [\fD^{\fI}]^{\mu}{}_{\lambda} \cdot \{ (-1)^{(|\vartheta_{\mu}| - |\vartheta_{\lambda}|)(|\psi^{\lambda}| - |\bbz^{\fI}|)} \partial^{\op}_{\fI}(\psi^{\lambda}) \cdot \Phi_{\mu} \},
\end{equation}
which is precisely the formula (\ref{eq_Ddecompositionfinal}). We have to verify two facts. Each $\frP_{\fI}{}^{\lambda}{}_{\mu}$ defined by (\ref{eq_operatorsframe}) is to be a $k$-th order differential operator on $\E$ over $U$, and the decomposition (\ref{eq_Ddecompositionfinal}) has to be unique. 

To prove the first claim, let us show that $\frP_{\fI}{}^{\lambda}{}_{\mu}$ is $w(\fI)$-th order (and thus $k$-th order) differential operator for any $\fI \in \ol{\N}{}^{n}$ with $w(\fI) \leq k$. We proceed by induction in $w(\fI)$. For $w(\fI) = 0$, the only possibility is $\fI = \fnula$. But it is easy to check that 
\begin{equation} \label{eq_frPfnula}
\frP_{\fnula}{}^{\lambda}{}_{\mu}(\psi) = (-1)^{(|\vartheta_{\mu}| - |\vartheta_{\lambda}|)|\psi^{\lambda}|} \psi^{\lambda} \cdot \Phi_{\mu}
\end{equation}
is indeed $\C^{\infty}_{\M}(U)$-linear of degree $|\vartheta_{\mu}| - |\vartheta_{\lambda}|$. Hence suppose that $w(\fI) > 0$ and $\frP_{\fJ}{}^{\lambda}{}_{\mu}$ are differential operators of the $w(\fJ)$-th order whenever $w(\fJ) < w(\fI)$. For any $f \in \C^{\infty}_{\M}(U)$, one finds
\begin{equation}
\begin{split}
[\frP_{\fI}{}^{\lambda}{}_{\mu}, \lambda_{f}](\psi) = & \ (-1)^{(|\vartheta_{\mu}| - |\vartheta_{\lambda}|)(|f| + |\psi^{\lambda}| - |\bbz^{\fI}|)} \{ \partial_{\fI}^{\op}(f \cdot \psi^{\lambda}) - (-1)^{|\bbz^{\fI}||f|} f \cdot \partial_{\fI}^{\op}(\psi^{\lambda}) \} \cdot \Phi_{\mu} \\
= & \ \sum_{\fnula < \fK \leq \fI} \binom{\fI}{\fK} (-1)^{\sigma'(\fI,\fK,|f|,\lambda,\mu)} \partial^{\op}_{\fK}(f) \cdot \frP_{\fI-\fK}{}^{\lambda}{}_{\mu}(\psi),
\end{split}
\end{equation}
where $\sigma'(\fI,\fK,|f|,\lambda,\mu)$ is some integer which does not depend on $\psi$ and we have used (\ref{eq_partialIopLeibniz}). This shows that $[\frP_{\fI}{}^{\lambda}{}_{\mu}, \lambda_{f}]$ can be written as a $\C^{\infty}_{\M}(U)$-linear combination of the operators $\frP_{\fI-\fK}{}^{\lambda}{}_{\mu}$ for $\fnula < \fK \leq \fI$. By the induction hypothesis, one has 
\begin{equation}
\frP_{\fI-\fK}{}^{\lambda}{}_{\mu} \in \Dif^{w(\fI - \fK)}_{\E}(U) \subseteq \Dif^{w(\fI)-1}_{\E}(U).
\end{equation}
Since $f \in \C^{\infty}_{\M}(U)$ was arbitrary, this proves that $\frP_{\fI}{}^{\lambda}{}_{\mu} \in \Dif^{w(\fI)}_{\E}(U)$ and the induction step is finished. We conclude that $\frP_{\fI}{}^{\lambda}{}_{\mu}$ are indeed $k$-th order differential operators on $\E$ over $U$. 

It remains to prove the uniqueness of the decomposition (\ref{eq_Ddecompositionfinal}). Let us define a graded $\R$-linear mapping $\fK^{\fJ}{}^{\kappa}{}_{\rho}: \Lin(\Gamma_{\E}(U)) \rightarrow \C^{\infty}_{\M}(U)$ as 
\begin{equation} \label{eq_Kmap}
\fK^{\fJ}{}^{\kappa}{}_{\rho}(D) := [\fD^{\fJ}]^{\kappa}{}_{\rho}, 
\end{equation}
for each $\fJ \in \ol{\N}{}^{n}$, $\kappa,\rho \in \{1, \dots, r\}$, and every graded $\R$-linear map $D: \Gamma_{\E}(U) \rightarrow \Gamma_{\E}(U)$. Note that
\begin{equation} \label{eq_Kmapslinear}
\fK^{\fJ}{}^{\kappa}{}_{\rho}(f \cdot D) = f \cdot \fK^{\fJ}{}^{\kappa}{}_{\rho}(D),
\end{equation}
for any $f \in \C^{\infty}_{\M}(U)$ and $D \in \Lin(\Gamma_{\E}(U))$. This follows from (\ref{eq_iteratedcommislinear}). Moreover, we claim that for every $\fI \in \ol{\N}{}^{n}$ and every $\mu,\lambda \in \{1,\dots,r\}$, one has 
\begin{equation} \label{eq_Kmapsdual}
\fK^{\fJ}{}^{\kappa}{}_{\rho}( \frP_{\fI}{}^{\lambda}{}_{\mu}) = \fJ! \cdot \delta^{\fJ}_{\fI} \delta^{\kappa}_{\mu} \delta^{\lambda}_{\rho}.  
\end{equation}
It is easy to observe that the left-hand vanishes whenever $w(\fJ) < w(\fI)$. This is because each its term is proportional to the expression $\partial_{\fI}^{\op}( \bbz^{B_{1}} \cdots \bbz^{B_{a}})$, where $a \leq w(\fJ) < w(\fI)$. But every such term vanishes. It also vanishes whenever $w(\fJ) > w(\fI)$ since we have proved that $\frP_{\fI}{}^{\lambda}{}_{\mu}$ is a $w(\fI)$-th order differential operator. We thus only have to consider the case $w(\fJ) = w(\fI)$. Using the same reasoning, we only have to deal with the term of the iterated graded commutator, where $\frP_{\fI}{}^{\lambda}{}_{\mu}$ acts on the section $\bbz^{\fJ} \cdot \Phi_{\rho}$. We have 
\begin{equation}
\frP_{\fI}{}^{\lambda}{}_{\mu}( \bbz^{\fJ} \cdot \Phi_{\rho}) = (-1)^{(|\vartheta_{\mu}|-|\vartheta_{\lambda}|)(|\bbz^{\fJ}| - |\bbz^{\fI}|)} \partial_{\fI}^{\op}(\delta^{\lambda}_{\rho} \cdot \bbz^{\fJ}) \cdot \Phi_{\mu} = \{ \fJ! \cdot \delta^{\fJ}_{\fI} \delta^{\kappa}_{\mu} \delta^{\lambda}_{\rho} \} \cdot \Phi_{\kappa},
\end{equation}
where we have used (\ref{eq_partialIoponzJ}). It follows that the function in the curly brackets must be $\fK^{\fJ}{}^{\kappa}{}_{\rho}(\frP_{\fI}{}^{\lambda}{}_{\mu})$. This proves (\ref{eq_Kmapsdual}). Finally, suppose that a given $D \in \Dif^{k}_{\E}(U)$ can be decomposed as 
\begin{equation}
D = \sum_{q=0}^{k} \sum_{\fI \in \ol{\N}{}^{n}(q)} \frac{1}{\fI!} f^{\fI}{}^{\mu}{}_{\lambda} \cdot \frP_{\fI}{}^{\lambda}{}_{\mu},
\end{equation}
for some functions $f^{\fI}{}^{\mu}{}_{\lambda} \in \C^{\infty}_{\M}(U)$, where $|f^{\fI}{}^{\mu}{}_{\lambda}| = |D| + |\bbz^{\fI}| + |\vartheta_{\lambda}| - |\vartheta_{\mu}|$. Choose any $\fJ \in \ol{\N}{}^{n}$ with $w(\fJ) \leq k$ and $\kappa,\rho \in \{1, \dots, r\}$ and apply $\fK^{J}{}^{\kappa}{}_{\rho}$ on both sides of this expression. It follows immediately from (\ref{eq_Kmapslinear}) and (\ref{eq_Kmapsdual}) that 
\begin{equation}
[\fD^{\fJ}]^{\kappa}_{\rho} \equiv K^{\fJ}{}^{\kappa}{}_{\rho}(D) = f^{\fJ}{}^{\kappa}{}_{\rho}. 
\end{equation}
This proves the uniqueness of the decomposition (\ref{eq_Ddecompositionfinal}). The remaining statements of the proposition now follow immediately. 
\end{proof}
Recall that we work in the category $\gVBun^{\infty}_{\M}$ of graded vector bundles over a fixed base manifold $\M$. Its morphisms are $\C^{\infty}_{\M}$-linear sheaf morphisms of the respective sheaves of sections. It turns out that the assignment of the graded vector bundle of $k$-th order differential operators respects the categorical structure.
\begin{tvrz}
For each $k \in \N_{0}$, the assignment $\E \mapsto \frD^{k}_{\E}$ defines a functor 
\begin{equation}
\frD^{k}: (\gVBun^{\infty}_{\M})^{\op} \rightarrow \gVBun^{\infty}_{\M}. 
\end{equation}
\end{tvrz}
\begin{proof}
Let $F: \E \rightarrow \E'$ be a graded vector bundle map, that is a $\C^{\infty}_{\M}$-linear sheaf morphism $F: \Gamma_{\E} \rightarrow \Gamma_{\E}$ of any given degree $|F|$. We must produce a graded vector bundle map
\begin{equation}
\frD^{k}F: \frD^{k}_{\E'} \rightarrow \frD^{k}_{\E}. 
\end{equation}
For each $U \in \Op(M)$, we thus have to find a degree $|F|$ graded $\C^{\infty}_{\M}(U)$-linear map 
\begin{equation}
(\frD^{k}F)_{U}: \Dif^{k}_{\E'}(U) \rightarrow \Dif^{k}_{\E}(U).
\end{equation}
For every $D' \in \Dif^{k}_{\E'}(U)$ and any $\psi \in \Gamma_{\E}(U)$, define 
\begin{equation}
[(\frD^{k}F)_{U}(D')](\psi) := (-1^{|F||D'|} D'( F_{U}(\psi)). 
\end{equation}
We must argue that this defines a $k$-th order differential operator on $\E$ over $U$ of degree $|D'| + |F|$. It is clearly $\R$-linear of degree $|D'| + |F|$ and for any $f \in \C^{\infty}_{\M}(U)$, one finds the relation
\begin{equation}
[(\frD^{k}F)_{U}(D')]^{(1)}_{(f)} = (\frD^{k}F)_{U}( D'^{(1)}_{(f)}). 
\end{equation}
By iterating this procedure, one can argue that $(\frD^{k}F)_{U}(D') \in \Jet^{k}_{\E}(U)$. It is easy to see that $(\frD^{k}F)_{U}$ is $\C^{\infty}_{\M}(U)$-linear of degree $|F|$. For any $V \in \Op(U)$ and $\psi \in \Gamma_{\E}(U)$, one finds
\begin{equation}
\begin{split}
[(\frD^{k}F)_{V}(D'|_{V})](\psi|_{V}) = & \ (-1)^{|F||D'|} D'|_{V}(F_{V}(\psi|_{V})) = (-1)^{|F||D'|} D'|_{V}( F_{U}(\psi)|_{V}) \\
= & \ (-1)^{|F||D'|} D'(F_{U}(\psi))|_{V} = [(\frD^{k}F)_{U}(D')](\psi)|_{V} 
\end{split}
\end{equation}
But it now follows from Proposition \ref{tvrz_locsheaf} that necessarily
\begin{equation}
(\frD^{k}F)_{V}(D'|_{V}) = (\frD^{k}F)_{U}(D')|_{V}. 
\end{equation}
This proves the naturality in $U$ and we conclude that $\frD^{k}F: \Dif^{k}_{\E'} \rightarrow \Dif^{k}_{\E}$ is indeed a $\C^{\infty}_{\M}$-linear sheaf morphism of degree $|F|$. The facts that 
\begin{equation}
\frD^{k}\1_{\E} = \1_{\frD^{k}_{\E}}, \; \; \frD^{k}(G \circ F) = \frD^{k}F \circ \frD^{k}G
\end{equation}
are obvious. This concludes the proof. 
\end{proof}
\section{Symbol of a differential operator} \label{sec_symbol}
Recall that to any $k \in \N_{0}$, a \textbf{completely symmetric $k$-form on a graded manifold $\M$ over $U \in \Op(M)$} is a graded $k$-linear map from $\X_{\M}(U)$ to $\C^{\infty}_{\M}(U)$ satisfying the conditions
\begin{align}
\omega(f \cdot X_{1}, \dots, X_{k}) = & \ (-1)^{|f||\omega|} f \cdot \omega(X_{1},\dots,X_{k}), \\
\omega(X_{1}, \dots, X_{i},X_{i+1}, \dots, X_{k}) = & \  (-1)^{|X_{i}||X_{i+1}|} \omega(X_{1}, \dots, X_{i+1},X_{i}, \dots, X_{k}),
\end{align}
for all $f \in \C^{\infty}_{\M}(U)$ and $X_{1},\dots,X_{k} \in \X_{\M}(U)$. Such maps form a graded vector space $\~\Omega^{k}_{\M}(U)$ and they are local in the sense completely analogous to Section \ref{sec_local}. Consequently, they can be naturally restricted to smaller open subsets. This provides one with a sheaf $\~\Omega^{k}_{\M}$ of graded $\C^{\infty}_{\M}$-modules called the \textbf{sheaf of completely symmetric $k$-forms on $\M$}. There is a canonical (degree zero) bilinear, graded symmetric, and associative product 
\begin{equation}
\odot: \~\Omega^{k}_{\M}(U) \times \~\Omega^{\ell}_{\M}(U) \rightarrow \~\Omega^{k+\ell}_{\M}(U).
\end{equation}
If $(U,\varphi)$ is a graded local chart for $\M$ with local coordinate functions $\{ \bbz^{A}\}_{A=1}^{n}$, one can show that $\~\Omega^{k}_{\M}(U)$ is freely and finitely generated by the the collection $\{ \dr{\bbz}^{\fI} \}_{\fI \in \ol{\N}{}^{n}(k)}$, where
\begin{equation}
\dr{\bbz}^{\fI} := \underbrace{\dr{\bbz}^{1} \odot \cdots \odot \dr{\bbz}^{1}}_{i_{1} \times} \odot \cdots \odot \underbrace{\dr{\bbz}^{n} \odot \cdots \odot \dr{\bbz}^{n}}_{i_{n} \times},
\end{equation}
and $\dr{\bbz}^{A} \in \~\Omega^{1}_{\M}(U) \equiv \Omega^{1}_{\M}(U)$ are the usual coordinate $1$-forms. It follows that $\~\Omega^{k}_{\M}$ is a sheaf of sections of a graded vector bundle which we shall denote simply as $S^{k}(T^{\ast}\M)$. 

Before we proceed, note that to a pair of graded vector bundles $\E$ and $\E'$ over a graded manifold $\M$, one can assign a graded vector bundle $\ul{\Hom}(\E,\E')$. Its space of local 
sections over $U \in \Op(M)$ is defined as to be the graded space of $\C^{\infty}_{\M}(U)$-linear maps from $\Gamma_{\E}(U)$ to $\Gamma_{\E'}(U)$, that is
\begin{equation}
\Gamma_{\ul{\Hom}(\E,\E')}(U) := \Lin^{\C^{\infty}_{\M}(U)}( \Gamma_{\E}(U), \Gamma_{\E'}(U)). 
\end{equation}
Analogously to Section \ref{sec_local}, these form into a graded sheaf of graded $\C^{\infty}_{\M}$-modules. Suppose $\{ \Phi_{\lambda} \}_{\lambda = 1}^{r}$ is a local frame for $\E$ over $U$ and $\{ \Psi_{\kappa} \}_{\kappa=1}^{r'}$ is a local frame for $\E'$ over $U$. Then $\Gamma_{\ul{\Hom}(\E,\E')}(U)$ is easily shown to be freely and finitely generated by the collection $\{ \frE{}^{\lambda}{}_{\kappa} \}$ of $\C^{\infty}_{\M}(U)$-linear maps, each one defined for $\lambda \in \{1,\dots, r\}$ and $\kappa \in \{1, \dots, r'\}$ by
\begin{equation} \label{eq_frEstandardbasis}
\frE^{\lambda}{}_{\kappa}(\psi) = (-1)^{(|\Psi_{\kappa}| - |\Phi_{\lambda}|)|\psi^{\lambda}|} \psi^{\lambda} \cdot \Psi_{\kappa},
\end{equation}
for all $\psi = \psi^{\lambda} \cdot \Phi_{\lambda} \in \Gamma_{\E}(U)$.  Observe that $\ul{\Hom}(\E,\E) = \frD^{0}_{\E}$ and compare this to (\ref{eq_frPfnula}).
\begin{tvrz} \label{tvrz_morphisms}
Let $\E$ and $\E'$ be graded vector bundle bundles. Then graded vector bundle maps over $\1_{\M}$ can be identified with global sections of $\ul{\Hom}(\E,\E')$. 
\end{tvrz}
\begin{proof}
Recall that graded vector bundle maps from $\E$ to $\E'$ over $\1_{\M}$ are just $\C^{\infty}_{\M}$-linear sheaf morphisms from $\Gamma_{\E}$ to $\Gamma_{\E'}$. One usually considers just the degree zero case (that is they preserve the degree), but they can in principle be of any degree. 

A global section $F$ of $\ul{\Hom}(\E,\E')$ is a graded $\C^{\infty}_{\M}(M)$-linear map $F: \Gamma_{\E}(M) \rightarrow \Gamma_{\E'}(M)$. Similarly to Proposition \ref{tvrz_Locglobalsheaf}, there is a unique $\C^{\infty}_{\M}$-linear sheaf morphism $\ol{F}: \Gamma_{\E} \rightarrow \Gamma_{\E'}$, such that $F = \ol{F}_{M}$. This is the corresponding graded vector bundle map (of degree $|F|$). One usually does not distinguish between $F$ and $\ol{F}$ and writes simply $F: \E \rightarrow \E'$ for both equivalent notions. 
\end{proof}

We can now proceed to the main construction of this section.
\begin{tvrz} \label{tvrz_symbolmap}
For each $k \in \N_{0}$, there is a canonical surjective graded vector bundle map $\sigma: \frD^{k}_{\E} \rightarrow \ul{\Hom}(S^{k}(T^{\ast}\M), \frD^{0}_{\E})$ called the \textbf{symbol map}. For $k > 0$, its kernel can be identified with $\frD^{k-1}_{\E}$ and we thus have a short exact sequence
\begin{equation} \label{eq_SESdifdifhom}
\begin{tikzcd} 
0\arrow{r} & \frD^{k-1}_{\E} \arrow[hookrightarrow]{r} & \frD^{k}_{\E} \arrow{r}{\sigma} & \ul{\Hom}(S^{k}(T^{\ast}\M), \frD^{0}_{\E}) \arrow{r} & 0 
\end{tikzcd}
\end{equation}
in the category of graded vector bundles over $\M$. For $k = 0$, $\sigma$ is a graded vector bundle isomorphism, so the statement is still valid if we declare $\frD^{-1}_{\E} := 0$.  

For each $U \in \Op(M)$ and $D \in \Dif^{k}_{\E}(U)$, the graded $\C^{\infty}_{\M}(U)$-linear map $\sigma_{U}(D): \~\Omega^{k}_{\M}(U) \rightarrow \Dif^{0}_{\E}(U)$ of degree $|D|$ is called the \textbf{symbol of the differential operator $D$}. 
\end{tvrz}
\begin{proof}
Let $D \in \Dif^{k}_{\E}(U)$ for some $U \in \Op(M)$. For any $k$-tuple of functions $f_{1},\dots,f_{k} \in \C^{\infty}_{\M}(U)$, we want the symbol map to satisfy the equation
\begin{equation} \label{eq_symbolonktuple}
[\sigma(D)]( \dr{f}_{1} \odot \cdots \odot \dr{f}_{2}) = (-1)^{|f_{1}| + \cdots + |f_{k}|} D^{(k)}_{(f_{1},\dots,f_{k})} \in \Dif^{0}_{\E}(U),
\end{equation}
where we omit the explicit writing of the subscript $U$ throughout the proof. Now, suppose that $(U,\varphi)$ is a graded local chart, hence inducing a set $\{ \bbz^{A} \}_{A=1}^{n}$ of local coordinate functions. For each $\fI \in \ol{\N}{}^{n}(k)$, the formula (\ref{eq_symbolonktuple}) requires one to prescribe 
\begin{equation} \label{eq_sigmaUDonthebasis}
[\sigma(D)]( \dr{\bbz}^{\fI}) := (-1)^{|\bbz^{\fI}|} D^{(k)}_{(\bbz^{\fI}_{(k)})},
\end{equation}
see also (\ref{eq_zImultiindex}). Note that for $k = 0$, we declare $\dr{\bbz}^{\fnula} := 1$ and $D^{(0)}_{(\bbz^{\fnula}_{(0)})} := D$. 

The map $\sigma(D)$ is to be $\C^{\infty}_{\M}(U)$-linear of degree $|D|$. If we write a general $\omega \in \~\Omega^{k}_{\M}(U)$ as $\omega = \sum_{\fI \in \ol{\N}{}^{n}(k)} \frac{1}{\fI!} \omega_{\fI} \cdot \dr{\bbz}^{\fI}$, we thus have to define
\begin{equation} \label{eq_sigmaUDonomega}
[\sigma(D)](\omega) := \sum_{\fI \in \ol{\N}{}^{n}(k)} (-1)^{|D|(|\omega|-|\bbz^{\fI}|) + |\bbz^{\fI}|}\frac{1}{\fI!}  \omega_{\fI} \cdot D^{(k)}_{(\bbz^{\fI}_{(k)})}.
\end{equation}
It is obvious that such $\sigma(D)$ is $\C^{\infty}_{\M}(U)$-linear (in $\omega$) of degree $|D|$. Moreover, thanks to (\ref{eq_iteratedcommislinear}), this formula defines a degree zero $\C^{\infty}_{\M}(U)$-linear map
\begin{equation}
\sigma: \Dif^{k}_{\E}(U) \rightarrow \Lin^{\C^{\infty}_{\M}(U)}( \~\Omega^{k}_{\M}(U), \Dif^{0}_{\E}(U)). 
\end{equation}
Finally, if $(U',\varphi')$ is a different graded local chart, one can verify that the definitions agree on $U \cap U'$. This requires one to show that (\ref{eq_sigmaUDonomega}) is in fact independent of used local coordinates. We leave this for an interested reader. Note that one has to use the formula (\ref{eq_olfDformula}) together with (\ref{eq_olDjgsymmetry}) to find the transformation rules for the functions $\fK^{\fI}{}^{\mu}{}_{\lambda}(D)$. 

Since the formulas agree on the overlaps, they glue to
\begin{equation}
\sigma: \Gamma_{\frD^{k}_{\E}}(M) \rightarrow \Gamma_{\ul{\Hom}(S^{k}(T^{\ast}\M), \frD^{0}_{\E})}(M),
\end{equation}
which is of degree zero and $\C^{\infty}_{\M}(M)$. Hence by Proposition \ref{tvrz_morphisms}, this defines a degree zero graded vector bundle map $\sigma: \frD^{k}_{\E} \rightarrow \ul{\Hom}(S^{k}(T^{\ast}\M), \frD^{0}_{\E})$. Note that one can verify that such $\sigma$ satisfies (\ref{eq_symbolonktuple}) and for any graded local chart $(U,\varphi)$, the restriction of $\sigma$ to $U$ is given by the formula (\ref{eq_sigmaUDonomega}). 

Recall that $\sigma$ is surjective, if it is surjective as a sheaf morphism. This is equivalent to the surjectivity of the corresponding map of global sections. In fact, it suffices to show that each point $m \in M$ has a neighborhood $U \in \Op_{m}(M)$, such that the restriction of $\sigma$ to $U$ is surjective. Let $(U,\varphi)$ be a graded local chart for $\M$ and $\{ \Phi_{\lambda} \}_{\lambda=1}^{r}$ be a local frame for $\E$ over $U$, that is $\sigma$ is given by (\ref{eq_sigmaUDonomega}). Let $F \in \Lin^{\C^{\infty}_{\M}(U)}( \~\Omega^{k}_{\M}(U), \Dif^{0}_{\E}(U))$. For each $\fI \in \ol{\N}{}^{n}(k)$, one can thus write
\begin{equation}
F(\dr{\bbz}^{\fI}) = \fF^{\fI}{}^{\mu}{}_{\lambda} \cdot \frP_{\fnula}{}^{\lambda}{}_{\mu},
\end{equation}
where $\fF^{\fI}{}^{\mu}{}_{\lambda} \in \C^{\infty}_{\M}(U)$ have degree $|\fF^{\fI}{}^{\mu}{}_{\lambda}| = |F| + |\bbz^{\fI}| + |\vartheta_{\lambda}| - |\vartheta_{\mu}|$. Define $D \in \Dif^{k}_{\E}(U)$ by 
\begin{equation}
D := \sum_{\fJ \in \ol{\N}{}^{n}(k)} \frac{1}{\fJ!}(-1)^{|\bbz^{\fJ}|} \fF^{\fJ \mu}{}_{\lambda} \cdot \frP_{\fJ}{}^{\lambda}{}_{\mu}. 
\end{equation}
Note that $|D| = |F|$ and observe that $D^{(k)}_{(\bbz^{\fI}_{(k)})} = K^{\fI}{}^{\kappa}{}_{\rho}(D) \cdot \frP^{\fnula \rho}{}_{\kappa}$ for each $\fI \in \ol{\N}{}^{n}(k)$. 

It is then easy to utilize (\ref{eq_Kmapslinear}) and (\ref{eq_Kmapsdual}) to see that $[\sigma(D)](\dr{\bbz}^{\fI}) = F(\dr{\bbz}^{\fI})$. Hence the restriction of $\sigma$ to $U$ is surjective. 

Next, let $k > 0$. We have to prove that for every $U \in \Op(M)$, one has $\ker(\sigma) = \Dif^{k-1}_{\E}(U)$. Since we are comparing two subsheaves of $\Dif^{k}_{\E}$, one has to show that for each $m \in M$, there is $U \in \Op_{m}(M)$, such that this is true. But if $(U,\varphi)$ is a graded local chart, this statement is immediately verified thanks to the definition of $\Dif^{k-1}_{\E}(U)$ and the formulas (\ref{eq_symbolonktuple}, \ref{eq_sigmaUDonthebasis}). 

Finally, for $k = 0$, note that $\~\Omega^{0}_{\M} \equiv \C^{\infty}_{\M}$ and for each $D \in \Dif^{0}_{\E}(M)$ and $f \in \C^{\infty}_{\M}(M)$ one has $[\sigma(D)](f) = (-1)^{|D||f|} f \cdot D$, which is obviously an isomorphism. 
\end{proof} 
\begin{rem}
Let us rephrase the above result perhaps in a more standard way. Without going into details, there are canonical isomorphisms $\ul{\Hom}(\E,\E') \cong \E^{\ast} \otimes \E'$ and $S^{k}(T^{\ast}\M) \cong S^{k}(T\M)^{\ast}$, where the sheaf of sections of $S^{k}(T\M)$ is the sheaf of completely symmetric $k$-vector fields on $\M$. One can thus view the symbol map as a graded vector bundle map $\sigma: \frD^{k}_{\E} \rightarrow S^{k}(T\M) \otimes \E^{\ast} \otimes \E$. Since tensor products of graded vector bundles are a bit cumbersome to work with, we stick to the maps of sections.
\end{rem}
\section{Atiyah Lie algebroid of a graded vector bundle} \label{sec_atiyah}
In this section, we aim to construct a certain subbundle $\ol{\frD}{}^{k}_{\E} \subseteq \frD^{k}_{\E}$ for every $k \in \N_{0}$. It will come equipped with a canonical degree zero graded vector bundle map $\ell_{(k)}: \ol{\frD}{}^{k}_{\E} \rightarrow S^{k}(T\M)$. It turns out that this subclass of differential operators has a nice algebraic structure which is preserved by these maps. For $k = 1$, this will give us a non-trivial example of a graded Lie algebroid, see e.g. \S 6 of \cite{vysoky2022graded}. Let us start by some observations. 

\begin{tvrz}
Let $D \in \Dif^{k}_{\E}(U)$ and $D' \in \Dif^{m}_{\E}(U)$ for some $k,m \in \N_{0}$ and $U \in \Op(M)$. Then $D \circ D' \in \Dif^{k+m}_{\E}(U)$. 
\end{tvrz}
\begin{proof}
Let us prove this statement by induction in $k + m \in \N_{0}$. For $k + m = 0$, one has $k = m = 0$, and the claim follows from the fact that the composition of two graded $\C^{\infty}_{\M}(U)$-linear maps is graded $\C^{\infty}_{\M}(U)$-linear. Hence assume that $k + m > 0$ and the claim is true for all pairs $k',m' \in \N_{0}$ with $k' + m' < k + m$. For any $f \in \C^{\infty}_{\M}(U)$, one has
\begin{equation}
[D \circ D', \lambda_{f}] = D \circ [D',\lambda_{f}] + (-1)^{|D'||f|} [D,\lambda_{f}] \circ D'.
\end{equation}
where we have used the Leibniz rule for $[\cdot,\cdot]$. But both terms are $(k+m-1)$-th order differential operators by the induction hypothesis. Since $f$ was arbitrary, one obtains $D \circ D' \in \Dif^{k+m}_{\E}(U)$. 
\end{proof}

Recall that $S^{k}(T\M)$ is a graded vector bundle, such that it sheaf of sections is a sheaf $\~\X^{k}_{\M}$ of completely symmetric $k$-vector fields. If $(U,\varphi)$ is a graded local chart inducing a set of local coordinate functions $\{ \bbz^{A} \}_{A=1}^{n}$, then the graded $\C^{\infty}_{\M}(U)$-module $\~\X^{k}_{\M}(U)$ is freely and finitely generated by a collection $\{ \partial_{\fI}^{\op} \}_{\fI \in \ol{\N}{}^{n}(k)}$, where for each $\fI \in \ol{\N}{}^{n}(k)$, one defines
\begin{equation} \label{eq_partialIop}
\partial^{\op}_{\fI} := \underbrace{\frac{\partial}{\partial \bbz^{n}} \odot \cdots \odot \frac{\partial}{\partial \bbz^{n}}}_{i_{n} \times} \odot \cdots \odot \underbrace{\frac{\partial}{\partial \bbz^{1}} \odot \cdots \odot \frac{\partial}{\partial \bbz^{1}}}_{i_{1} \times},
\end{equation}
where $\frac{\partial}{\partial \bbz^{A}} \in \~\X^{1}_{\M}(U) \equiv \X^{1}_{\M}(U)$ are the coordinate vector fields. There is a canonical identification of $S^{k}(T^{\ast}\M)$ with the dual graded vector bundle to $S^{k}(T\M)$. If $(U,\varphi)$ is a graded local chart, the action of the local frame $\{ \dr{\bbz}^{\fI} \}_{\fI \in \ol{\N}{}^{n}(k)}$ for $S^{k}(T^{\ast}\M)$ over $U$ on the above generators satisfies
\begin{equation}
\dr{\bbz}^{\fI}( \partial_{\fJ}^{\op}) = \fI! \cdot \delta^{\fI}_{\fJ},
\end{equation}
for all $\fI,\fJ \in \ol{\N}{}^{n}(k)$. We can use this to prove the following statement.
\begin{tvrz}
For each $k \in \N_{0}$, there is a canonical fiber-wise injective graded vector bundle map $I_{(k)}: S^{k}(T\M) \rightarrow \ul{\Hom}(S^{k}(T^{\ast}\M), \frD^{0}_{\E})$. 
\end{tvrz}
\begin{proof}
For every $X \in \~\X^{k}_{\M}(M)$ and every $\omega \in \~\Omega^{k}_{\M}(M)$, we declare
\begin{equation}
[I_{(k)}(X)](\omega) := (-1)^{|X||\omega|} \lambda_{\omega(X)} \in \Dif^{0}_{\E}(M). 
\end{equation}
It is straightforward to verify that $I_{(k)}(X)$ defines a graded $\C^{\infty}_{\M}(M)$-linear map of degree $|X|$, that is a global section of $\ul{\Hom}(S^{k}(T^{\ast}\M), \frD^{0}_{\E})$. Moreover, $I_{(k)}(X)$ is graded $\C^{\infty}_{\M}(M)$-linear of degree zero in $X$, hence it defines a graded vector bundle map $I_{(k)}$ by Proposition \ref{tvrz_morphisms}. One only has to prove that $I_{(k)}$ is fiber-wise injective. Note that this \textit{does not} follow automatically from the injectivity of $I_{(k)}$. 

Let $(U,\varphi)$ be a graded local chart for $\M$ and $\{ \Phi_{\lambda} \}_{\lambda=1}^{r}$ a local frame for $\E$ over $U$. Then we have the induced local frames:  $\{ \partial^{\op}_{\fI} \}_{\fI \in \ol{\N}{}^{n}(k)}$ for $S^{k}(T\M)$, the $k$-forms $\{ \dr{\bbz}^{\fI} \}_{\fI \in \ol{\N}{}^{n}(k)}$ for $S^{k}(T^{\ast}\M)$ and $\{ \frP_{\fnula}{}^{\lambda}{}_{\mu} \}_{\mu,\lambda=1}^{r}$ for $\frD^{0}_{\E}$. Finally, there is the induced ``standard basis'' local frame $\{ \frE_{\fI}{}^{\lambda}{}_{\mu} \}$ for $\ul{\Hom}(S^{k}(T^{\ast}\M), \frD^{0}_{\E})$ defined for each $\fI \in \ol{\N}{}^{n}(k)$ and $\mu,\lambda \in \{1,\dots,r\}$ to satisfy
\begin{equation}
\frE_{\fI}{}^{\lambda}{}_{\mu}(\omega) := \frac{1}{\fI!} (-1)^{(|\vartheta_{\mu}|-|\vartheta_{\lambda}| - |\bbz^{\fI}|)(|\omega| - |\bbz^{\fI}|)} \omega_{\fI} \cdot \frP_{\fnula}{}^{\lambda}{}_{\mu},
\end{equation} 
for all $\omega = \sum_{\fJ \in \ol{\N}{}^{n}(k)} \frac{1}{\fJ!} \omega_{\fJ} \cdot \dr{\bbz}^{\fJ}$. This is an explicit example of the maps (\ref{eq_frEstandardbasis}). One finds
\begin{equation}
\begin{split}
[I_{(k)}|_{U}(\partial_{\fI}^{\op})](\omega) = & \ (-1)^{|\bbz^{\fI}||\omega|} \omega(\partial_{\fI}^{\op}) \cdot \1_{\Gamma_{\E}(U)} = (-1)^{|\bbz^{\fI}||\omega|} w_{\fI} \cdot \1_{\Gamma_{\E}(U)} \\
= & \ (-1)^{|\bbz^{\fI}||\omega|} \omega_{\fI} \delta^{\mu}_{\lambda} \cdot \frP_{\fnula}{}^{\lambda}{}_{\mu} = (-1)^{|\bbz^{\fI}|} \fI!  \delta^{\mu}_{\lambda} \cdot \frE_{\fI}{}^{\lambda}{}_{\mu}(\omega). 
\end{split}
\end{equation}
For every $\fI \in \ol{\N}{}^{n}(k)$, we have thus obtained the formula
\begin{equation}
I_{(k)}|_{U}( \partial^{\op}_{\fI}) = \sum_{\lambda=1}^{r} (-1)^{|\bbz^{\fI}|} \fI! \cdot \frE_{\fI}{}^{\lambda}{}_{\lambda}
\end{equation}
For every $m \in U$, the induced graded linear map on the fibers thus takes the form 
\begin{equation}
[I_{(k)}]_{m}( \partial_{\fI}^{\op}|_{m}) = \sum_{\lambda=1}^{r} (-1)^{|\bbz^{\fI}|} \fI! \cdot \frE_{\fI}{}^{\lambda}{}_{\lambda}|_{m}. 
\end{equation}
It is not difficult to see that it is injective (that is injective in each degree). Recall that if $\{ \Phi_{\lambda} \}_{\lambda=1}^{r}$ is a local frame for $\E$ over $U$, then $\{ \Phi_{\lambda}|_{m} \}_{\lambda=1}^{r}$ forms a total basis of the fiber $\E_{m}$ for any $m \in U$. 
\end{proof}
\begin{tvrz}
Let $k \in \N_{0}$ be arbitrary. Then there is a graded vector bundle $\ol{\frD}{}^{k}_{\E}$ over $\M$, together with a pair of graded vector bundle maps $J_{(k)}: \ol{\frD}{}^{k}_{\E} \rightarrow \frD^{k}_{\E}$ and $\ell_{(k)}: \ol{\frD}{}^{k}_{\E} \rightarrow S^{k}(T\M)$, fitting into the universal pullback diagram
\begin{equation} \label{eq_pullbackfromdiff}
\begin{tikzcd}
\ol{\frD}{}^{k}_{\E} \arrow{r}{\ell_{(k)}} \arrow{d}{J_{(k)}} & S^{k}(T\M) \arrow{d}{I_{(k)}} \\
\frD^{k}_{\E} \arrow{r}{\sigma} & \ul{\Hom}(S^{k}(T^{\ast}\M), \frD^{0}_{\E}). 
\end{tikzcd}
\end{equation}
$\ell_{(k)}$ is (fiber-wise) surjective and $J_{(k)}$ is fiber-wise injective. Consequently $\ol{\frD}{}^{k}_{\E}$ can be viewed as a subbundle of $\frD^{k}_{\E}$, the corresponding sheaf of graded $\C^{\infty}_{\M}$-submodules $\ol{\Dif}{}^{k}_{\E}$ being
\begin{equation} \label{eq_olDifkEsheaf}
\ol{\Dif}{}^{k}_{\E}(U) = \{ D \in \Dif^{k}_{\E}(U) \; | \; \sigma(D) = I_{(k)}(X) \text{ for some } X \in \~\X^{k}_{\M}(U) \},
\end{equation}
for each $U \in \Op(M)$. Moreover, the new graded vector bundle fits into the short exact sequence of graded vector bundles over $\M$:
\begin{equation} \label{eq_SESdidifSkTM}
\begin{tikzcd}
0 \arrow{r} & \frD^{k-1}_{\E} \arrow[hookrightarrow]{r}& \ol{\frD}{}^{k}_{\E} \arrow{r}{\ell_{(k)}} & S^{k}(T\M) \arrow{r} & 0. 
\end{tikzcd}
\end{equation}
\end{tvrz}
\begin{proof}
The construction of the universal pullback diagram is the same as for ordinary vector bundles. One constructs a graded vector bundle map 
\begin{equation}
K: \frD^{k}_{\E} \oplus S^{k}(T\M) \rightarrow \ul{\Hom}(S^{k}(T^{\ast}\M), \frD^{0}_{\E}),
\end{equation}
by declaring $K(D,X) := \sigma(D) - I_{(k)}(X)$ for all $(D,X) \in \Dif^{k}_{\E}(M) \oplus \~\X^{k}_{\M}(M)$. By Proposition \ref{tvrz_morphisms}, this extends to the morphism of the corresponding sheaves of sections, hence one can define 
\begin{equation}
\ol{\Dif}{}^{k}_{\E} := \ker(K). 
\end{equation}
Since $K$ is surjective thanks to the surjectivity of $\sigma$, this defines a sheaf of sections of a subbundle $\ol{\frD}{}^{k}_{\E}$ of $\frD^{k}_{\E} \oplus S^{k}(T\M)$. The maps $J_{(k)}$ and $\ell_{(k)}$ are then simply the restrictions of the corresponding projections. It is straightforward to prove that (\ref{eq_pullbackfromdiff}) now forms a universal pullback diagram, $\ell_{(k)}$ inherits the surjectivity from $\sigma$ and $J_{(k)}$ its fiber-wise injectivity from $I_{(k)}$. The description (\ref{eq_olDifkEsheaf}) then follows from the construction. The exactness of (\ref{eq_SESdidifSkTM}) follows immediately from the exactness of (\ref{eq_SESdifdifhom}) and the fact that $\ker(\sigma) \subseteq \ol{\Dif}{}^{k}_{\E}$ and it coincides with $\ker(\ell_{(k)})$. 
\end{proof}

\begin{rem} \label{rem_lkDformula}
Observe that it follows from the definitions and (\ref{eq_symbolonktuple}) that for $D \in \ol{\Dif}{}^{k}_{\E}(U)$, the $k$-vector field $\ell_{(k)}(D)$ is uniquely determined by the formula
\begin{equation} \label{eq_lkDformula}
D^{(k)}_{(f_{1},\dots,f_{k})} = (-1)^{|f_{1}| + \dots + |f_{k}|} [\ell_{(k)}(D)]( \dr{f}_{1}, \dots, \dr{f}_{k}) \cdot \1_{\Gamma_{\E}(U)}.
\end{equation}
In particular, for $k = 0$, one has $D = [\ell_{(0)}](D) \cdot \1_{\Gamma_{\E}(U)}$. This shows that $D \in \ol{\Dif}{}^{0}_{\E}(U$), iff $D = \lambda_{f}$ for some $f \in \C^{\infty}_{\M}(U)$. Note that $\ell_{(0)}(\lambda_{f}) = f$. 
\end{rem}

We can now prove the main theorem of this section, relating graded commutators of (a subclass of) differential operators to a canonical algebraic structure on completely symmetric multivector fields induced by the Schouten-Nijenhuis bracket. We will provide all necessary details in course of the proof, but we refer the interested reader to Proposition A.12 in \cite{vysoky2022graded}. 

\begin{theorem}
Let $k,m \in \N_{0}$, $U \in \Op(M)$, $D \in \ol{\Dif}{}^{k}_{\M}(U)$ and $D' \in \ol{\Dif}{}^{m}_{\M}(U)$ be arbitrary. 

Then $[D,D'] \in \ol{\Dif}{}^{k+m-1}_{\M}(U)$ and there holds the formula 
\begin{equation} \label{eq_ellmapsofgcommareSN}
\ell_{(k+m-1)}([D,D']) = [\ell_{(k)}(D), \ell_{(m)}(D')]_{S},
\end{equation}
where $[\cdot,\cdot]_{S}$ is the Schouten-Nijenhuis bracket of completely symmetric multivector fields on $\M$. 
\end{theorem}
\begin{proof}
Recall that for each $k,m \in \N_{0}$ and $U \in \Op(M)$, $[\cdot,\cdot]_{S}$ is a degree zero $\R$-bilinear bracket
\begin{equation}
[\cdot,\cdot]_{S}: \~\X^{k}_{\M}(U) \times \~\X^{m}_{\M}(U) \rightarrow \~\X^{k+m-1}_{\M}(U),
\end{equation}
having the following properties:
\begin{enumerate}[(i)]
\item For every $f \in \C^{\infty}_{\M}(M)$ and $X \in \~\X^{k}_{\M}(U)$, one has $[f,X]_{S} = (-1)^{|f|+1} j_{\dr{f}}(X)$, where $j$ is the interior product. For $k=m=1$, it coincides with the graded commutator of vector fields.
\item It is graded skew-symmetric, that is $[X,Y]_{S} = -(-1)^{|X||Y|}[Y,X]_{S}$.
\item It satisfies the graded Leibniz rule with respect to the symmetric product $\odot$: 
\begin{equation}
[X,Y \odot Z]_{S} = [X,Y]_{S} \odot Z + (-1)^{|X||Y|} Y \odot [X,Z]_{S}.
\end{equation}
\item It satisfies the graded Jacobi identity:
\begin{equation} \label{eq_SchoutengJI}
[X,[Y,Z]_{S}]_{S} = [[X,Y]_{S},Z]_{S} + (-1)^{|X||Y|} [Y, [X,Z]_{S}]_{S}. 
\end{equation}
\end{enumerate}
$[\cdot,\cdot]_{S}$ is called the \textbf{Schouten-Nijenhuis bracket} and it is uniquely determined by the properties (i)-(iv). Note that they are significantly simpler then for the analogous bracket for completely \textit{skew}-symmetric multivector fields.

Now, for any $f \in \C^{\infty}_{\M}(U)$, one has $\lambda_{f} \in \ol{\Dif}{}^{0}_{\E}(U)$ by Remark \ref{rem_lkDformula}. For any $D \in \ol{\Dif}{}^{k}_{\E}(U)$, we thus expect $D^{(1)}_{(f)} \equiv [D,\lambda_{f}] \in \ol{\Dif}{}^{k-1}_{\M}(U)$. For any $g_{1},\dots,g_{k-1} \in \C^{\infty}_{\M}(U)$, one has 
\begin{equation}
\begin{split}
[D^{(1)}_{(f)}]^{(k-1)}_{(g_{1},\dots,g_{k-1})} = & \  D^{(k)}_{(f,g_{1},\dots,g_{k-1})} \\
= & \  (-1)^{|f|+|g_{1}|+\dots+|g_{k-1}|} [\ell_{(k)}(D)]( \dr{f}, \dr{g}_{1}, \dots, \dr{g}_{k-1}) \cdot \1_{\Gamma_{\E}(U)} \\
= & \ (-1)^{|f|(1 + |D|) + |g_{1}| + \dots + |g_{k-1}|} [j_{\dr{f}}( \ell_{(k)}(D))]( \dr{g}_{1}, \dots, \dr{g}_{k-1}) \cdot \1_{\Gamma_{\E}(U)},
\end{split}
\end{equation}
where we have used (\ref{eq_lkDformula}) for $D$ and the definition of the interior product. But using (\ref{eq_lkDformula}) again for $D^{(1)}_{(f)}$, this at once proves that $D^{(1)}_{(f)} \in \ol{\Dif}{}^{k-1}_{\E}(U)$ and 
\begin{equation} \label{eq_ellmapsofcommareSNforknula}
\ell_{(k-1)}( D^{(1)}_{(f)}) = (-1)^{|f|(1+|D|)} j_{\dr{f}}( \ell_{(k)}(D)) = [\ell_{(k)}(D), f]_{S}. 
\end{equation} 
Note that this is in accordance with (\ref{eq_ellmapsofgcommareSN}) since $f = \ell_{(0)}(\lambda_{f})$ by Remark \ref{rem_lkDformula}. 

Let us now prove the claim by induction in $k + m \in \N_{0}$. For $k + m = 0$, one has $D = \lambda_{f}$ and  $D' = \lambda_{g}$ by Remark \ref{rem_lkDformula}. The formula (\ref{eq_ellmapsofgcommareSN}) thus follows from $[\lambda_{f}, \lambda_{g}] = 0$. Next, let $k + m = 1$. Since the equation (\ref{eq_ellmapsofgcommareSN}) is graded skew-symmetric in $(D,D')$, we may assume that $k = 1$ and $m = 0$. But we have already proved this case above. 

Finally, suppose $k + m > 1$ and assume that the formula holds for all differential operators whose sum of degrees is strictly less then $k + m$. Let $f_{1},\dots,f_{k+m-1} \in \C^{\infty}_{\M}(U)$. Then 
\begin{equation}
\begin{split}
[D,D']^{(k+m-1)}_{(f_{1},\dots,f_{k+m-1})} = & \ [[D,D'], \lambda_{f_{1}}]^{(k+m-2)}_{(f_{2},\dots,f_{k+m-1})} \\
= & \ [D, D'^{(1)}_{(f_{1})} ]_{(f_{2}, \dots, f_{k+m-1})}^{(k+m-2)} +(-1)^{|D'||f_{1}|} [D^{(1)}_{(f_{1})}, D']^{(k+m-2)}_{(f_{2},\dots,f_{k+m-1})}, \\
\end{split}
\end{equation}
where we have used the graded Jacobi identity. Using the induction hypothesis together with the formulas (\ref{eq_lkDformula}, \ref{eq_ellmapsofcommareSNforknula}) and the graded Jacobi identity (\ref{eq_SchoutengJI}) for $[\cdot,\cdot]_{S}$ now leads to 
\begin{equation}
\begin{split}
[D,D'&]^{(k+m-1)}_{(f_{1},\dots,f_{k+m-1})} = \\
= & \ (-1)^{|f_{2}| + \dots + |f_{k+m-1}|} [[\ell_{(k)}(D),\ell_{(m)}(D')]_{S},f]_{S}(\dr{f}_{2}, \dots, \dr{f}_{k+m-1}) \cdot \1_{\Gamma_{\E}(U)} \\
= & \ (-1)^{|f_{1}| + \dots + |f_{k+m-1}|} [\ell_{(k)}(D), \ell_{(m)}(D')]_{S}( \dr{f}_{1}, \dots, \dr{f}_{k+m-1}) \cdot \1_{\Gamma_{\E}(U)}.
\end{split}
\end{equation}
But this and the (\ref{eq_lkDformula}) prove at once that $[D,D'] \in \ol{\Dif}{}^{k+m-1}_{\E}(U)$ and the formula (\ref{eq_ellmapsofgcommareSN}). This finishes the induction step and thus concludes the proof.
\end{proof}
\begin{cor}
The elements of $\ol{\Dif}{}^{1}_{\E}(U)$ are called the \textbf{derivative endomorphisms of $\E$ over $U$}. We also call $\ol{\frD}{}^{1}_{\E}$ the \textbf{Atiyah bundle of a graded vector bundle $\E$} and denote it as $\At_{\E}$. 

The triple $(\At_{\E}, \ell_{(1)}, [\cdot,\cdot])$ forms an example of a transitive graded Lie algebroid of degree zero, called the \textbf{Atiyah Lie algebroid of $\E$}. We have a short exact sequence
\begin{equation} \label{eq_SESAtiyah}
\begin{tikzcd}
0 \arrow{r} & \frD^{0}_{\E} \arrow[hookrightarrow]{r} & \At_{\E} \arrow{r}{\ell_{(1)}} & T\M \arrow{r} & 0. 
\end{tikzcd}
\end{equation}
\end{cor}
\begin{proof}
See \S 6 of \cite{vysoky2022graded} for the definition of a graded Lie algebroid of degree zero. Being transitive means that its anchor map $\ell_{(1)}: \At_{\E} \rightarrow T\M$ is surjective. This will follow from the exactness of the sequence (\ref{eq_SESAtiyah}). $\ol{\Dif}{}^{1}_{\E}(M)$ forms a degree zero graded Lie algebra thanks to the above theorem and (\ref{eq_ellmapsofgcommareSN}). One only has to verify the graded Leibniz rule to finish the proof. For all $D,D' \in \ol{\Dif}{}^{1}_{\E}(M)$ and $f \in \C^{\infty}_{\M}(M)$, one has
\begin{equation}
\begin{split}
[D, f \cdot D'] = & \  [D, \lambda_{f} \circ D'] = [D,\lambda_{f}] \circ D' + (-1)^{|D||f|} \lambda_{f} \circ [D,D'] \\
= & \ \{ (-1)^{|f|} [\ell_{(1)}(D)](\dr{f}) \circ \1_{\Gamma_{\E}(U)} \} \circ D' + (-1)^{|D||f|} f \cdot [D,D'] \\
= & \ \ell_{1}(D)(f) \cdot D' + (-1)^{|D||f|} f \cdot [D,D']. 
\end{split}
\end{equation}
But this is precisely the graded Leibniz rule for $(\At_{\E}, \ell_{(1)}, [\cdot,\cdot])$. The sequence (\ref{eq_SESAtiyah}) is just (\ref{eq_SESdidifSkTM}). 
\end{proof}
As in ordinary differential geometry, splittings of (\ref{eq_SESAtiyah}) are of some importance. 
\begin{definice}
A \textbf{connection on $\E$} is a splitting $\cD: T\M \rightarrow \At_{\E}$ of the sequence (\ref{eq_SESAtiyah}). For each $X \in \X_{\M}(M)$, the first order differential operator $\cD_{X} := \cD(X) \in \ol{\Dif}{}^{1}_{\E}(M)$ is called a \textbf{covariant derivative} with with respect to $X$. 
\end{definice}

By definition, the covariant derivative operators adhere to the expected rules
\begin{equation}
\cD_{f \cdot X}(\psi) = f \cdot \cD_{X}\psi, \; \; \cD_{X}(f \cdot \psi) = X(f) \cdot \psi + (-1)^{|X||f|} f \cdot \cD_{X}\psi,
\end{equation}
for all $X \in \X_{\M}(M)$, $\psi \in \Gamma_{\E}(M)$ and $f \in \C^{\infty}_{\M}(M)$. Since the splittings of short exact sequences always exist in the category of graded vector bundles over $\M$, we immediately obtain the following statement:
\begin{tvrz}
There always exists some connection on $\E$. 
\end{tvrz}
To conclude this section, note that one can define the \textbf{curvature operator $R_{\cD}$ of $\cD$} in order to describe the failure of $\cD(T\M)$ to form an involutive subbundle of $\At_{\E}$, that is set 
\begin{equation}
[R_{\cD}(X,Y)](\psi) = [\cD_{X}, \cD_{Y}](\psi) - \cD_{[X,Y]}(\psi),
\end{equation}
for all $X,Y \in \X_{\M}(M)$ and $\psi \in \Gamma_{\E}(M)$. By exactness of (\ref{eq_SESAtiyah}), $R_{\cD}(X,Y)$ is in $\Dif^{0}_{\E}(M)$, that is $\C^{\infty}_{\M}(M)$-linear of degree $|X| + |Y|$. 
\section{Geometric presheaves} \label{sec_geometric}
In this section, we will identify and examine an important class of presheaves of graded $\C^{\infty}_{\M}$-modules. Those will be of vital importance in the following section.
Recall that for each $a \in M$, we have a sheaf of ideals $\J^{a}_{\M}$ of functions vanishing at $a$. For any $U \in \Op(M)$, it is defined as 
\begin{equation}
\J^{a}_{\M}(U) = \left\{ \begin{array}{cc} \{ f \in \C^{\infty}_{\M}(U) \; | \; f(a) = 0 \} & \text{ when } a \in U \\
\C^{\infty}_{\M}(U) & \text{ when } a \notin U \end{array} \right.
\end{equation}

These (sheaves of) ideals play an important role to characterize vanishing functions: $f \in \C^{\infty}_{\M}(U)$ is zero, iff $f \in (\J^{a}_{\M}(U))^{q}$ for all $q \in \N$ and $a \in U$. We have already used this, see (\ref{eq_fzerocondition}). 

Now, for any given graded $\C^{\infty}_{\M}(U)$-module $(P,\tr)$, let us write $P^{[q,a]} := (\J^{a}_{\M}(U))^{q} \tr P$ for the graded $\C^{\infty}_{\M}(U)$-submodule generated by the graded subset $\{ f \tr p \; | \; f \in (\J^{a}_{\M}(U))^{q}, \; p \in P \}$, and let 
\begin{equation}
P^{\bullet} := \bigcap_{q \in \N} \bigcap_{a \in U} P^{[q,a]}.
\end{equation}
We would like to have the following criterion: $p = 0$, iff $p \in P^{[q,a]} = 0$ for all $q \in \N$ and $a \in U$. Equivalently, $P^{\bullet} = 0$. However, for a general graded $\C^{\infty}_{\M}(U)$-module $P$, this is not true. This leads us to the following definition.
\begin{definice}
We say that a graded $\C^{\infty}_{\M}(U)$-module $P$ is \textbf{geometric}, if $P^{\bullet} = 0$. 

More generally, let $\F$ be a presheaf of graded $\C^{\infty}_{\M}$-modules. We say that $\F$ is \textbf{geometric}, if $\F(U)$ is a geometric graded $\C^{\infty}_{\M}(U)$-module for every $U \in \Op(M)$. 
\end{definice}
\begin{lemma} \label{lem_geometricislocal}
Let $\F$ be a sheaf of graded $\C^{\infty}_{\M}$-modules. Then $\F$ is geometric, iff there is an open cover $\{ U_{\alpha} \}_{\alpha \in I}$ of $M$, such that $\F|_{U_{\alpha}}$ is geometric for every $\alpha \in I$. 
\end{lemma}
\begin{proof}
The only if direction is trivial. To prove the if direction, suppose that there is an open cover $\{ U_{\alpha} \}_{\alpha \in I}$, such that $\F|_{U_{\alpha}}$ is a geometric. Let $U \in \Op(M)$ and $\psi \in \F(U)^{\bullet}$. It is obvious that $\psi|_{U \cap U_{\alpha}} \in \F(U \cap U_{\alpha})^{\bullet}$ for every $\alpha \in I$, hence $\psi|_{U \cap U_{\alpha}} = 0$ for every $\alpha \in I$, by assumption. Since $\F$ is a sheaf, this implies $\psi = 0$. Hence $\F$ is geometric. 
\end{proof}

The notion of geometric sheaves is justified by the following observation.
\begin{tvrz} \label{tvrz_sheaofsectionsisgeometric}
For every graded vector bundle $\E$, its sheaf of sections $\Gamma_{\E}$ is geometric. 
\end{tvrz}
\begin{proof}
Thanks to Lemma \ref{lem_geometricislocal}, it suffices to prove that sheaf of sections of a trivial graded vector bundle is geometric. We can thus assume that $\Gamma_{\E} = \C^{\infty}_{\M}[K]$ for some finite-dimensional graded vector space $K$, see (\ref{eq_trivialbundle}). For each $U \in \Op(M)$, every $\psi \in \Gamma_{\E}(U)$ can be \textit{uniquely} decomposed as $\psi \in \psi^{\lambda} \otimes \vartheta_{\lambda}$, if we fix some total basis $\{ \vartheta_{\lambda} \}_{\lambda=1}^{r}$ of $K$. Let $q \in \N$ and $a \in U$. It follows that $\psi \in \Gamma_{\E}(U)^{[q,a]}$, iff $\psi^{\lambda} \in (\J^{a}_{\M}(U))^{q}$ for each $\lambda \in \{1,\dots,r\}$. Consequently, one obtains
\begin{equation}
\Gamma_{\E}(U)^{\bullet} = \C^{\infty}_{\M}(U)^{\bullet} \otimes_{\R} K = 0,
\end{equation}
where we view $\C^{\infty}_{\M}$ as a sheaf of graded $\C^{\infty}_{\M}$-modules and use (\ref{eq_fzerocondition}). Hence $\Gamma_{\E}$ is geometric. 
\end{proof}
It turns out that there is a canonical procedure making graded presheaves of $\C^{\infty}_{\M}$-modules into geometric ones. This procedure is universal in the sense explained below. 
\begin{tvrz} \label{tvrz_geo}
To any presheaf $\F$ of graded $\C^{\infty}_{\M}$-modules, there is a geometric presheaf of graded $\C^{\infty}_{\M}$-modules $\Geo(\F)$, together with a $\C^{\infty}_{\M}$-linear presheaf morphism $\geo: \F \rightarrow \Geo(\F)$. 

It has the following universal property: To any geometric presheaf of graded $\C^{\infty}_{\M}$-modules $\G$ and any $\C^{\infty}_{\M}$-linear presheaf morphism $\varphi: \F \rightarrow \G$, there exists a unique presheaf morphism $\hat{\varphi}: \Geo(\F) \rightarrow \G$, such that $\hat{\varphi} \circ \geo = \varphi$. 

$\Geo(\F)$ is called the \textbf{geometrization} of $\F$. It is unique up to a $\C^{\infty}_{\M}$-linear presheaf isomorphism. If $\F$ is geometric, then $\geo: \F \rightarrow \Geo(\F)$ is an isomorphism. 
\end{tvrz}
\begin{proof}
The idea is to simply kill the unwanted submodule, hoping that this makes sense. Hence for any $U \in \Op(M)$, set
\begin{equation}
\Geo(\F)(U) := \F(U) / \F(U)^{\bullet}. 
\end{equation}
Obviously, we let $\geo_{U}: \F(U) \rightarrow \Geo(\F)(U)$ to be the canonical quotient map. For any $\psi \in \F(U)$, we shall write $[\psi]{}^{\bullet} := \geo_{U}(\psi)$ and call it the \textbf{geometric class} of $\psi$. There is a unique graded $\C^{\infty}_{\M}(U)$-module structure on $\Geo(\F)(U)$ making $\geo_{U}$ into a $\C^{\infty}_{\M}(U)$-linear map. 

Let $U \in \Op(M)$ and suppose $[\psi]^{\bullet} \in \Geo(\F)(U)^{\bullet}$. For every $q \in \N$ and $a \in U$, we thus have $[\psi]^{\bullet} = f^{\mu} \tr [\psi_{\mu}]^{\bullet} \equiv [f^{\mu} \tr \psi_{\mu}]^{\bullet}$ for some collection of $f^{\mu} \in (\J^{a}_{\M}(U))^{q}$ and $\psi_{\mu} \in \F(U)$. But this proves that $\psi - f^{\mu} \tr \psi_{\mu} \in \F(U)^{\bullet}$. In particular, this implies that $\psi \in \F(U)^{[q,a]}$. Since $q \in \N$ and $a \in U$ were arbitrary, we have $\psi \in \F(U)^{\bullet}$ and thus $[\psi]^{\bullet} = 0$. This proves that $\Geo(\F)(U)$ is a geometric graded $\C^{\infty}_{\M}(U)$-module. Since the restriction morphisms obviously map $\F(U)^{\bullet}$ to $\F(V)^{\bullet}$ for any $V \subseteq U$, there is a canonical way to make $\Geo(\F)$ into a presheaf of graded $\C^{\infty}_{\M}(U)$-modules and $\geo := \{ \geo_{U} \}_{U \in \Op(M)}$ into a $\C^{\infty}_{\M}$-linear presheaf morphism. Hence $\Geo(\F)$ is a geometric presheaf of graded $\C^{\infty}_{\M}$-modules

Suppose $\varphi: \F \rightarrow \G$ is a $\C^{\infty}_{\M}$-linear presheaf morphism into a geometric presheaf. For each $U \in \Op(M)$, $\hat{\varphi}_{U}$ must be given by $\hat{\varphi}_{U}[\psi]^{\bullet} := \varphi_{U}(\psi)$ for every $\psi \in \F(U)$. Since $\geo_{U}$ is surjective, one only has to verify that $\hat{\varphi}_{U}$ is well-defined. But this follows from the obvious property $\varphi_{U}( \F(U)^{\bullet}) \subseteq \G(U)^{\bullet} = 0$. Clearly $\hat{\varphi} := \{ \hat{\varphi}_{U} \}_{U \in \Op(M)}$ defines a $\C^{\infty}_{\M}$-linear presheaf morphism from $\Geo(\F)$ to $\G$. The rest of the claims follows easily. 
\end{proof}
\begin{rem}
The assignment $\F \mapsto \Geo(\F)$ can be viewed as a faithful functor from the category $\PSh^{\C^{\infty}_{\M}}$ of presheaves of graded $\C^{\infty}_{\M}$-modules into its full subcategory $\PSh^{\C^{\infty}_{\M}}_{\geo}$ of \textit{geometric} presheaves of graded $\C^{\infty}_{\M}$-modules. 
\end{rem}
\begin{rem}
Note that even if $\F$ is a sheaf, its geometrization is in general not a sheaf. Conversely, a sheafification of a geometric presheaf is not necessarily geometric. 
\end{rem}
There are some general statements about geometric presheaves, some of which will become useful in the next section.
\begin{tvrz} \label{tvrz_geoprops}
Let $\F,\G$ be any presheaves of graded $\C^{\infty}_{\M}$-modules. 

Then we have the following observations:
\begin{enumerate}[(i)]
\item Any restriction of a geometric presheaf is geometric.
\item If $\varphi: \F \rightarrow \G$ is an injective $\C^{\infty}_{\M}$-linear presheaf morphism and $\G$ is geometric, then so is $\F$.
\item If $\G$ is geometric, the presheaf $\ul{\PSh}^{\C^{\infty}_{\M}}( \F, \G)$ is geometric. 
\item The dual presheaf $\F^{\ast} := \ul{\PSh}^{\C^{\infty}_{\M}}(\F,\C^{\infty}_{\M})$ is always a geometric sheaf. 
\end{enumerate}
\end{tvrz}
\begin{proof}
The fact $(i)$ is obvious. Let us start by proving $(ii)$. For any $U \in \Op(M)$, one obviously has $\varphi_{U}( \F(U)^{\bullet}) \subseteq \G(U)^{\bullet}$. If $\G$ is geometric, we have $\varphi_{U}(\F(U)^{\bullet}) = 0$ and since $\varphi_{U}$ is injective, this proves $\F(U)^{\bullet} = 0$. Hence $\F$ is geometric. 

To prove $(iii)$, recall that $\ul{\PSh}^{\C^{\infty}_{\M}}(\F,\G)$ is a presheaf assigning to each $U \in \Op(M)$ the graded vector space of $\C^{\infty}_{\M}|_{U}$-linear presheaf morphisms from $\F|_{U}$ to $\G|_{U}$, with the obvious graded $\C^{\infty}_{\M}(U)$-module structure. The presheaf restrictions are restrictions of presheaf morphisms. If $\G$ is a sheaf, then it is also a sheaf. Hence suppose that $\varphi \in \ul{\PSh}^{\C^{\infty}_{\M}}(\F,\G)(U)^{\bullet}$. For any $q \in \N$ and $a \in U$, one can thus write it as some finite combination $\varphi = f^{\mu} \tr \varphi_{\mu}$, where $f^{\mu} \in (\J^{a}_{\M}(U))^{q}$ and $\varphi_{\mu}: \F|_{U} \rightarrow \G|_{U}$ are $\C^{\infty}_{\M}|_{U}$-linear presheaf morphisms. For any $V \in \Op(U)$ and $\psi \in \F(V)$, one thus has 
\begin{equation}
\varphi_{V}(\psi) = f^{\mu}|_{V} \tr (\varphi_{\mu})_{V}(\psi) \in \G(V)^{[q,a]},
\end{equation}
since $f^{\mu}|_{V} \in (\J^{a}_{\M}(V))^{q}$. In particular, we can choose any $q \in \N$ and $a \in V$, which proves that $\varphi_{V}(\psi) \in \G(V)^{\bullet}$. But since $\G$ is geometric, $\psi \in \F(V)$ and $V \in \Op(U)$ were arbitrary, we get $\varphi = 0$. 

The claim $(iv)$ follows from the fact that $\F^{\ast} \equiv \ul{\PSh}^{\C^{\infty}_{\M}}(\F, \C^{\infty}_{\M})$ and $\C^{\infty}_{\M}$ is a geometric \textit{sheaf} due to (\ref{eq_fzerocondition}). As already noted, $\F^{\ast}$ is automatically a sheaf. 
\end{proof}
\section{Graded jet bundles: construction} \label{sec_gJetconst}
To any graded vector bundle $\E$ and $k \in \N_{0}$, we intend to assign a $k$-th order graded jet bundle $\frJ^{k}_{\E}$. In the classical setting, this is usually done by constructing the $k$-th order jet space over each $m \in M$ and then obtaining the total space of the $k$-th jet bundle as a disjoint union of jet spaces over all points. This is a problematic approach in the graded setting. Instead, we will directly construct its sheaf of sections $\Jet^{k}_{\E} := \Gamma_{\frJ^{k}_{\E}}$. This turns out to be a bit complicated and requires one to employ the geometrization procedure described in the preceding section. 

We will do the construction in three steps. We will first construct a presheaf of graded $\C^{\infty}_{\M}$-modules denoted as $\pJet^{k}_{\E}$. Next, we will geometrize it to obtain a geometric presheaf $\gpJet^{k}_{\E}$. In the final step, we will sheafify it to obtain a sheaf of graded $\C^{\infty}_{\M}$-modules $\Jet^{k}_{\E}$, which will turn out to be locally freely and finitely generated of a constant graded rank. 

$U \in \Op(M)$ be arbitrary. Let us first consider a graded vector space
\begin{equation}
\cX(U) := \C^{\infty}_{\M}(U) \otimes_{\R} \Gamma_{\E}(U). 
\end{equation}
There are two graded $\C^{\infty}_{\M}(U)$-module actions on $\cX(U)$:
\begin{align}
f \tr (g \otimes \psi) := & \ (f \cdot g) \otimes \psi, \\
f \btr (g \otimes \psi) := & \ (-1)^{|f||g|} g \otimes (f \cdot \psi),
\end{align}
for all $g \otimes \psi \in \cX(U)$. In fact, with restricting morphisms defined in an obvious way, $\cX$ forms a presheaf of graded $\C^{\infty}_{\M}(U)$-modules. These two actions do not coincide. For each $f \in \C^{\infty}_{\M}(U)$, one can thus define a graded linear map $\delta_{f}: \cX(U) \rightarrow \cX(U)$ of degree $|f|$, given by
\begin{equation}
\delta_{f}(g \otimes \psi) := (L^{\tr}_{f} - L^{\btr}_{f})(g \otimes \psi) \equiv (f \cdot g) \otimes \psi - (-1)^{|f||g|} g \otimes (f \cdot \psi),
\end{equation}
for all $g \otimes \psi \in \cX(U)$, where $L^{\tr}_{f}$ and $L^{\btr}_{f}$ denote the left translations by $f$. It is straightforward to verify that $\delta_{f}$ is also graded $C^{\infty}_{\M}(U)$-linear with respect to \textit{both} $\tr$ and $\btr$. 

Now, let $k \in \N_{0}$ and let $\V_{(k)}(U) \subseteq \cX(U)$ be the graded subspace 
\begin{equation} \label{eq_VkU}
\V_{(k)}(U) := \R \{ \delta_{f_{1}} \dots \delta_{f_{k+1}}(g \otimes \psi) \; | \; f_{1},\dots,f_{k+1},g \in \C^{\infty}_{\M}(U), \; \psi \in \Gamma_{\E}(U) \}
\end{equation}
Thanks to the $\C^{\infty}_{\M}(U)$-linearity of $\delta_{f}$, it forms a graded $\C^{\infty}_{\M}(U)$-submodule of $\cX(U)$. In fact, one can view $\V_{(k)}$ as a presheaf of graded $\C^{\infty}_{\M}$-submodules of the presheaf $\cX$. This is true for both actions $\tr$ and $\btr$. Note that clearly $\V_{(k+1)}(U) \subseteq \V_{(k)}(U)$. Finally, let 
\begin{equation}
\pJet^{k}_{\E}(U) := \cX(U) / \V_{(k)}(U).
\end{equation}
Let $\natural^{(k)}_{U}: \cX(U) \rightarrow \pJet^{k}_{\E}(U)$ be the natural quotient map. There are two graded $\C^{\infty}_{\M}(U)$-module structures on $\pJet^{k}_{\E}(U)$ induced by $\tr$ and $\btr$, which we will denote by the same symbol as on $\cX(U)$. There are induced restrictions making $\pJet^{k}_{\E}$ into a presheaf of graded $\C^{\infty}_{\M}$-modules (with respect to two different actions) and $\natural^{(k)} := \{ \natural^{(k)}_{U} \}_{U \in \Op(M)}$ into a $\C^{\infty}_{\M}$-linear presheaf morphism. For every $f \in \C^{\infty}_{\M}(U)$ and $\psi \in \Gamma_{\E}(U)$, we will write 
\begin{equation}
f \otimes_{(k)} \psi := \natural^{(k)}_{U}(f \otimes \psi).
\end{equation}
For each $f \in \C^{\infty}_{\M}(U)$, one has $\delta_{f}(\V_{(k)}(U)) \subseteq \V_{(k)}(U)$. Consequently, there is an induced graded linear map $\delta_{f}: \pJet^{k}_{\E}(U) \rightarrow \pJet^{k}_{\E}(U)$, $\C^{\infty}_{\M}(U)$-linear with respect to both $\tr$ and $\btr$ actions, and again denoted by the same symbol. Explicitly, one finds
\begin{equation}
\delta_{f}(g \otimes_{(k)} \psi) = (f \cdot g) \otimes_{(k)} \psi - (-1)^{|f||g|} g \otimes_{(k)} (f \cdot \psi),
\end{equation}
for all generators $g \otimes_{(k)} \psi \in \pJet^{k}_{\E}(U)$. For every $j \in \N$ and all $f_{1},\dots,f_{j} \in \C^{\infty}_{\M}(U)$, we define 
\begin{equation}
\delta^{(j)}_{(f_{1},\dots,f_{j})} := \delta_{f_{1}} \circ \dots \circ \delta_{f_{j}}.
\end{equation}
The relations forced by the quotient can be then written as 
\begin{equation}
\delta_{(f_{1},\dots,f_{k+1})}^{(k+1)}(g \otimes_{(k)} \psi) = 0,
\end{equation}
for all $f_{1},\dots,f_{k+1},g \in \C^{\infty}_{\M}(U)$ and $\psi \in \Gamma_{\E}(U)$. It is not a coincidence that this resembles the equivalent definition of $\Dif^{k}_{\E}(U)$ discussed in Remark \ref{rem_localusingDkplus1}. This finishes the construction of the presheaf $\pJet^{k}_{\E}$. As already announced, we now define 
\begin{equation}
\gpJet^{k}_{\E} := \Geo( \pJet^{k}_{\E}).
\end{equation}
Before we proceed, recall that to any presheaf $\F$ of graded $\C^{\infty}_{\M}$-modules, there exists its sheafification $\Sff(\F)$, a sheaf of graded $\C^{\infty}_{\M}$-modules together with a $\C^{\infty}_{\M}$-linear presheaf morphism $\sff: \F \rightarrow \Sff(\F)$, having the following universal property: To any sheaf of graded $\C^{\infty}_{\M}$-modules $\G$ and any $\C^{\infty}_{\M}$-linear presheaf morphism $\varphi: \F \rightarrow \G$, there exists a unique presheaf morphism $\hat{\varphi}: \Sff(\F) \rightarrow \G$, such that $\hat{\varphi} \circ \sff = \varphi$. $\Sff(\F)$ is called the \textbf{sheafification of $\F$} and if $\F$ was already a sheaf, $\sff: \F \rightarrow \Sff(\F)$ is a sheaf isomorphism. In fact, this defines a faithful functor $\Sff: \PSh^{\C^{\infty}_{\M}} \rightarrow \Sh^{\C^{\infty}_{\M}}$ into the category of sheaves of graded $\C^{\infty}_{\M}$-modules. Let 
\begin{equation}
\Jet^{k}_{\E} := \Sff( \gpJet^{k}_{\E}). 
\end{equation}
\begin{definice}
Let $k \in \N_{0}$. For every $U \in \Op(M)$, the elements of $\Jet^{k}_{\E}(U)$ are called \textbf{$k$-th order jets of $\E$ over $U$}. $\Jet^{k}_{\E}$ is called the \textbf{sheaf of $k$-th order jets of $\E$}. 
\end{definice}

We will now spend the rest of this section proving that $\Jet^{k}_{\E}$ is locally freely and finitely generated of a constant graded rank, that is a sheaf of sections of a graded vector bundle which we will denote as $\frJ^{k}_{\E}$. To do so, suppose that $(U,\varphi)$ is a graded local chart on $\M$ and we have a local frame $\{ \Phi_{\lambda} \}_{\lambda=1}^{r}$ for $\E$ over $U$. Now, for every $j \in \N$, $f_{1},\dots,f_{j} \in \C^{\infty}_{\M}(U)$ and $\lambda \in \{1,\dots,r\}$, define
\begin{equation} \label{eq_fdelta}
\fdelta^{(j)}_{(f_{1},\dots,f_{j}),\lambda} := [ \delta^{(j)}_{(f_{1},\dots,f_{j})}(1 \otimes_{(k)} \Phi_{\lambda}) ]^{\bullet} \in \gpJet^{k}_{\E}(U).
\end{equation}

First, we obtain an analogue of Lemma \ref{lem_olDjprops}:
\begin{lemma}
For every $j \in \N$, one has 
\begin{equation} \label{eq_fdeltajgsymmetry}
\fdelta^{(j)}_{(f_{1},\dots,f_{i},f_{i+1}, \dots ,f_{j}),\lambda} = (-1)^{|f_{i}||f_{i+1}|} \fdelta^{(j)}_{(f_{1}, \dots, f_{i+1}, f_{i}, \dots, f_{j}),\lambda},
\end{equation}
for every $j$-tuple $(f_{1}, \dots, f_{j})$ of functions in $\C^{\infty}_{\M}(U)$ and $i \in \{1, \dots, j-1\}$. Moreover, for any $f,g \in \C^{\infty}_{\M}(U)$ and any $(j-1)$-tuple $(f_{2},\dots,f_{j})$ of functions in $\C^{\infty}_{\M}(U)$, one has
\begin{equation} \label{eq_fdeltagLeibniz}
\fdelta^{(j)}_{(f \cdot g, f_{2},\dots,f_{j}),\lambda} = f \tr \fdelta^{(j)}_{(g,f_{2},\dots,f_{j}), \lambda} + (-1)^{|f||g|} g \tr \fdelta^{(j)}_{(f,f_{2},\dots,f_{j}),\lambda} - \fdelta^{(j+1)}_{(f,g,f_{2},\dots,f_{j}), \lambda}.
\end{equation} 
\begin{proof}
The identity (\ref{eq_fdeltajgsymmetry}) follows immediately from the fact that $[\delta_{f},\delta_{g}] = 0$ for all $f,g \in \C^{\infty}_{\M}(U)$. The property (\ref{eq_fdeltagLeibniz}) can be obtained from the fact that for all $f,g \in \C^{\infty}_{\M}(U)$, one has
\begin{equation}
\delta_{f \cdot g} = L^{\tr}_{f} \circ \delta_{g} + (-1)^{|f||g|} L^{\btr}_{g} \circ \delta_{f}.
\end{equation}
This can be verified easily by evaluating both sides on the generators $h \otimes_{(k)} \psi \in \pJet^{k}_{\E}(U)$. Then use $\delta_{g} = L^{\tr}_{g} - L^{\btr}_{g}$, definitions, and (\ref{eq_fdeltajgsymmetry}), to arrive to (\ref{eq_fdeltagLeibniz}).  
\end{proof}
\end{lemma}
The above lemma is now vital for the proof of the following analogue of Proposition \ref{tvrz_olfDformula}. 
\begin{tvrz}
For each $j \in \{0, \dots, k-1\}$, one has the formula
\begin{equation} \label{eq_fdeltaformula}
\fdelta^{(k-j)}_{(f,g_{2},\dots,g_{k-j}), \lambda} = \sum_{q=1}^{j+1} (-1)^{q+1} \frac{1}{q!} \partial^{\tl}_{A_{1} \dots A_{q}}(f) \tr \fdelta^{(q+k-j-1)}_{(\bbz^{A_{1}}, \dots, \bbz^{A_{q}},g_{2},\dots,g_{k-j}),\lambda}
\end{equation}
for all $f,g_{2},\dots,g_{k-j} \in \C^{\infty}_{\M}(U)$ and $\lambda \in \{1, \dots, r\}$. 
\end{tvrz}
\begin{proof}
The proof is a line-to-line copy of the proof of Proposition \ref{tvrz_olfDformula} and we will not repeat it here. However, there is one notable difference we have to very carefully point out. In the proof of (\ref{eq_olfDformula}), we have used the fact that both its sides were in $\C^{\infty}_{\M}(U)$. In particular, we have used the criterion (\ref{eq_fzerocondition}) in the very last step of the proof. 

This time, both sides are elements of $\gpJet^{k}_{\E}(U)$. At the certain moment, one is able to show that the difference of two sides of (\ref{eq_fdeltaformula}) is in the graded $\C^{\infty}_{\M}(U)$-submodule $\gpJet^{k}_{\E}(U)^{\bullet}$. But since $\gpJet^{k}_{\E}(U)$ is geometric, it has to vanish. This is \textit{exactly} the reason why we geometrize the presheaf $\pJet^{k}_{\E}$, otherwise we were not able to deduce the formula (\ref{eq_fdeltaformula}). 
\end{proof}

We are almost ready to prove the main theorem of this section. However, we still need one more statement. It will also play an important role in the following section.
\begin{tvrz} \label{tvrz_diffopfactorization}
Let $k \in \N_{0}$ and $U \in \Op(M)$ be arbitrary. Let $D \in \Dif^{k}_{\E}(U)$. Define
\begin{equation}
\frj^{k}_{0}[D][f \otimes_{(k)} \psi]^{\bullet} := (-1)^{|D||f|} f \cdot D(\psi),
\end{equation}
for each $\psi \in \Gamma_{\E}(U)$ and $f \in \C^{\infty}_{\M}(U)$. Then $\frj^{k}_{0}[D]$ is a well-defined degree $|D|$ map 
\begin{equation}
\frj^{k}_{0}[D]: \gpJet^{k}_{\E}(U) \rightarrow \Gamma_{\E}(U),
\end{equation}
$\C^{\infty}_{\M}(U)$-linear with respect to the action $\tr$ on $\gpJet^{k}_{\E}(U)$ and the (given) one on $\Gamma_{\E}(U)$.
\end{tvrz}
\begin{proof}
It suffices to prove that $\hat{\frj}^{k}_{0}[D](f \otimes_{(k)} \psi) := (-1)^{|D||f|} f \cdot D(\psi)$ is a well-defined graded $\C^{\infty}_{\M}(U)$-linear map (with respect to $\tr$) from $\pJet^{k}_{\E}(U)$ to $\Gamma_{\E}(U)$ of degree $|D|$. Since $\Gamma_{\E}(U)$ is geometric, $\frj^{k}_{0}[D]$ is then obtained by the universal property of geometrization. To see that $\hat{\frj}^{k}_{0}[D]$ is well-defined, observe that 
\begin{equation}
\hat{\frj}^{k}_{0}[D] \circ \delta_{f} = -\hat{\frj}^{k}_{0}[D^{(1)}_{(f)}],
\end{equation} 
for any $f \in \C^{\infty}_{\M}(U)$. This is easily verified on generators. Whence
\begin{equation} \label{eq_factorizedoperatorondelta}
\hat{\frj}^{k}_{0}[D] \circ \delta^{(j)}_{(f_{1},\dots,f_{j})} = (-1)^{j} \hat{\frj}^{k}_{0}[D^{(j)}_{(f_{1},\dots,f_{j})}]
\end{equation}
for any $j \in \N$ and $f_{1},\dots,f_{j} \in \C^{\infty}_{\M}(U)$. In particular, $\hat{\frj}^{k}_{0}[D] \circ \delta^{(k+1)}_{(f_{1},\dots,f_{k+1})} = 0$, since $D \in \Dif^{k}_{\E}(U)$. 

This shows that $\hat{\frj}^{k}_{0}[D]$ is well-defined. The fact that it is $\C^{\infty}_{\M}(U)$-linear of degree $|D|$ is easily checked by a direct verification. 
\end{proof}
\begin{theorem} 
Let $\psi = \psi^{\lambda} \cdot \Phi_{\lambda} \in \Gamma_{\E}(U)$ be arbitrary. Then one has 
\begin{equation} \label{eq_1otimespsigclassdecomposition}
[1 \otimes_{(k)} \psi]^{\bullet} = \sum_{q=0}^{k} \sum_{\fI \in \ol{\N}{}^{n}(q)} (-1)^{q} \frac{1}{\fI!} \partial^{\tl}_{\fI}(\psi^{\lambda}) \tr \fdelta^{\fI}{}_{\lambda},
\end{equation}
where $\fdelta^{\fnula}{}_{\lambda} := [1 \otimes_{(k)} \Phi_{\lambda}]^{\bullet}$ and for every $q \in \N$ and $\fI \in \ol{\N}{}^{n}(q)$, one sets 
\begin{equation} \label{eq_fdeltafI}
\fdelta^{\fI}{}_{\lambda} := \fdelta^{(q)}_{(\bbz^{\fI}_{(q)}), \lambda}. 
\end{equation}
Finally, $\partial_{\fI}^{\tl} := (\partial^{\tl}_{1})^{i_{1}} \circ \dots \circ (\partial^{\tl}_{n})^{i_{n}}$, see also (\ref{eq_partialtl}). 
\end{theorem}
\begin{proof}
It follows from definitions and the formula (\ref{eq_fdeltaformula}) that 
\begin{equation}
\begin{split}
[1 \otimes_{(k)} \psi]^{\bullet} = & \  \psi^{\lambda} \tr [1 \otimes_{(k)} \Phi_{\lambda}]^{\bullet} - [\delta_{\psi^{\lambda}}(1 \otimes_{(k)} \Phi_{\lambda})]^{\bullet} \\
= & \ \psi^{\lambda} \tr \fdelta^{0}{}_{\lambda} + \sum_{q=1}^{k} (-1)^{q} \frac{1}{q!} \partial^{\tl}_{A_{1} \dots A_{q}} (\psi^{\lambda}) \tr \fdelta^{(q)}_{(\bbz^{A_{1}}, \dots, \bbz^{A_{q}}), \lambda} \\
= & \ \sum_{q=0}^{k} \sum_{\fI \in \ol{\N}{}^{n}(q)} (-1)^{q} \frac{1}{\fI!} \partial^{\tl}_{\fI}(\psi^{\lambda}) \tr \fdelta^{\fI}{}_{\lambda},
\end{split}
\end{equation}
where we have used (\ref{eq_fdeltajgsymmetry}) and the argument similar to (\ref{eq_twosummations}) to rewrite the sums in the last step. 
\end{proof}
This theorem is essential for the main statement of this section (and also of this paper).
\begin{theorem} \label{thm_generatorsgpJet}
The collection $\{ \fdelta^{\fI}{}_{\lambda}\}$ freely and finitely generates $\gpJet^{k}_{\E}|_{U}$.

Consequently, $\gpJet^{k}_{\E}|_{U}$ is a sheaf and $\Jet^{k}_{\E}$ is locally freely and finitely generated sheaf of graded $\C^{\infty}_{\M}$-modules of a constant graded rank. 
\end{theorem}
\begin{proof}
The formula (\ref{eq_1otimespsigclassdecomposition}) shows that a finite collection $\{ \fdelta^{\fI}{}_{\lambda} \}$, where $\fI$ goes over $\cup_{q = 0}^{k} \ol{\N}{}^{n}(q)$ and $\lambda \in \{1,\dots,r\}$, generates $\gpJet^{k}_{\E}|_{U}$. One only has to prove that it generates it freely. Let $V \in \Op(U)$ and suppose that 
\begin{equation} \label{eq_toprovethatitgeneratesfreely}
0 = \sum_{q=0}^{k} \sum_{\fI \in \ol{\N}^{n}(q)} (-1)^{q} \frac{1}{\fI!} f_{\fI}{}^{\lambda} \tr \fdelta^{\fI}{}_{\lambda}|_{V}
\end{equation}
is a zero element of $\gpJet^{k}_{\E}(V)$ of some degree $\ell \in \Z$ and some functions $f_{\fI}{}^{\lambda} \in \C^{\infty}_{\M}(V)$. Note that
\begin{equation}
|f_{\fI}{}^{\lambda}| = \ell - |\bbz^{\fI}| - |\vartheta_{\lambda}|,
\end{equation}
for each $\fI \in \cup_{q = 0}^{k} \ol{\N}{}^{n}(q)$ and each $\lambda \in \{1,\dots,r\}$. We must argue that all of those functions are zero. We will do so by utilizing the results of Proposition \ref{tvrz_diffopfactorization}. Indeed, let $D \in \Dif^{k}_{\E}(V)$ be arbitrary. Let us apply the operator $\frj^{k}_{0}[D]$ on both sides of the equation (\ref{eq_toprovethatitgeneratesfreely}). Since $\frj^{k}_{0}[D]$ is $\C^{\infty}_{\M}(V)$-linear of degree $|D|$, one finds
\begin{equation}
0 = \sum_{q=0}^{k} \sum_{\fI \in \ol{\N}{}^{n}(q)} (-1)^{q + |D|(\ell - |\bbz^{\fI}| - |\vartheta_{\lambda}|)} \frac{1}{\fI!} f_{\fI}{}^{\lambda} \cdot \frj^{k}_{0}[D]( \fdelta^{\fI}{}_{\lambda}|_{V}). 
\end{equation}
Now, it follows from the formula (\ref{eq_factorizedoperatorondelta}) that for each $\fI \in \ol{\N}{}^{n}(q)$ and $\lambda \in \{1, \dots, r\}$, one has 
\begin{equation}
\frj^{k}_{0}[D]( \fdelta^{\fI}{}_{\lambda}|_{V}) = (-1)^{q} [\fD^{\fI}]^{\mu}{}_{\lambda} \cdot \Phi_{\mu}|_{V}. 
\end{equation}
See Theorem \ref{thm_decompositionofdiffop} and above for the notation. We thus obtain the expression 
\begin{equation}
0 = \sum_{q=0}^{k} \sum_{\fI \in \ol{\N}{}^{n}(q)} (-1)^{|D|(\ell - |\bbz^{\fI}| - |\vartheta_{\lambda}|)} \frac{1}{\fI!} f_{\fI}{}^{\lambda} \cdot [\fD^{\fI}]^{\mu}{}_{\lambda} \cdot \Phi_{\mu}|_{V},
\end{equation}
for every $D \in \Dif^{k}_{\E}(V)$. Now, fix $\fJ \in \cup_{q=0}^{k} \ol{\N}{}^{n}(q)$ and $\kappa,\rho \in \{1,\dots,r \}$. Choose $D = \frP_{\fJ}{}^{\kappa}{}_{\rho}|_{V}$ and use the fact that then $[\fD^{\fI}]^{\mu}{}_{\lambda} = \fI! \delta_{\fJ}^{\fI} \delta^{\mu}_{\rho} \delta^{\kappa}_{\lambda}$. See (\ref{eq_operatorsframe}, \ref{eq_Kmapsdual}) for definitions and details. This gives
\begin{equation}
0 = (-1)^{(|\vartheta_{\rho}| - |\vartheta_{\kappa}| - |\bbz^{\fJ}|)(\ell - |\bbz^{\fJ}| - |\vartheta_{\kappa}|)} f_{\fJ}{}^{\kappa} \cdot \Phi_{\rho}|_{V}. 
\end{equation}
Note that there is no longer any summation on the right-hand side. Since the collection $\{ \Phi_{\lambda}|_{V} \}_{\lambda=1}^{r}$ freely generates $\Gamma_{\E}(V)$, this implies $f_{\fJ}{}^{\kappa} = 0$. Since $\fJ$ and $\kappa$ were arbitrary, this proves the claim. 

Now, since we have just proved that $\gpJet^{k}_{\E}|_{U}$ is a finitely and freely generated presheaf, it is in fact a sheaf. There is thus a canonical sheaf isomorphism $\Jet^{k}_{\E}|_{U} \cong \gpJet^{k}_{\E}|_{U}$. We will henceforth use this identification. This proves that $\Jet^{k}_{\E}$ is locally freely and finitely generated. Finally, it is of a constant graded rank, since if $(q_{j})_{j \in \Z} := \grk( \Jet^{k}_{\E}|_{U})$, one has 
\begin{equation}
q_{j} = \# \{ (\fI,\lambda) \in \cup_{q = 0}^{k} \ol{\N}{}^{n}(q) \times \{1, \dots, r \} \; | \; |\bbz^{\fI}| + |\vartheta_{\lambda}| = j \},
\end{equation}
which is a number independent of the used graded local chart $(U,\varphi)$ and a local frame $\{ \Phi_{\lambda} \}_{\lambda=1}^{r}$. 
\end{proof}

\begin{definice}
$\Jet^{k}_{\E}$ is a sheaf of sections a graded vector bundle $\frJ^{k}_{\E}$ called the \textbf{$k$-th order jet bundle of a graded vector bundle $\E$}.
\end{definice}

Similarly to the graded vector bundle of $k$-th order differential operators, one can view the assignment of the $k$-th order jet bundle as a functor in the category of graded vector bundles over a given graded manifold $\M$.

\begin{tvrz}
For each $k \in \N_{0}$, the assignment $\E \mapsto \frJ^{k}_{\E}$ defines a functor 
\begin{equation}
\frJ^{k}: \gVBun^{\infty}_{\M} \rightarrow \gVBun^{\infty}_{\M}.
\end{equation}
\end{tvrz}
\begin{proof}
Let $F: \E \rightarrow \E'$ be a graded vector bundle map, that is a $\C^{\infty}_{\M}$-linear sheaf morphism $F: \Gamma_{\E} \rightarrow \Gamma_{\E'}$ of any given degree $|F|$. We must produce a $\C^{\infty}_{\M}$-linear sheaf morphism 
\begin{equation}
\frJ^{k}F: \Jet^{k}_{\E} \rightarrow \Jet^{k}_{\E'}. 
\end{equation}
We do so by first producing a $\C^{\infty}_{\M}$-linear presheaf morphism $\frJ^{k}_{0}F: \pJet^{k}_{\E} \rightarrow \pJet^{k}_{\E}$. Let 
\begin{equation}
(\frJ^{k}_{0}F)_{U}(g \otimes_{(k)} \psi) := (-1)^{|F||g|} g \otimes_{(k)} F_{U}(\psi),
\end{equation} 
for all $U \in \Op(M)$, $g \in \C^{\infty}_{\M}(U)$ and $\psi \in \Gamma_{\E}(U)$. For each $f \in \C^{\infty}_{\M}(U)$, one has
\begin{equation}
(\frJ^{k}_{0}F)_{U} \circ \delta_{f} = (-1)^{|F||f|} \delta_{f} \circ (\frJ^{k}_{0}F)_{U}.
\end{equation}
This is easy to verify. This ensures that $(\frJ^{k}_{0}F)_{U}$ is well-defined. It is easily seen to be $\C^{\infty}_{\M}(U)$-linear of degree $|F|$ and natural in $V$. Moreover, one finds
\begin{equation}
\frJ^{k}_{0}\1_{\E} = \1_{\pJet^{k}_{\E}}, \; \; \frJ^{k}_{0}(G \circ F) = \frJ^{k}_{0}(G) \circ \frJ^{k}_{0}(F).
\end{equation}
This makes $\E \mapsto \pJet^{k}_{\E}$ into a functor from $\gVBun^{\infty}_{\M}$ to the category of presheaves of $\C^{\infty}_{\M}$-modules. Since $\Jet^{k}_{\E} = \Sff(\Geo(\pJet^{k}_{\E})$, where $\Geo$ and $\Sff$ are functors, let $\frJ^{k}F := \Sff(\Geo(\frJ^{k}_{0}F))$. 
\end{proof}
\begin{rem}
As already noted in Remark \ref{rem_totalspace}, one can construct a unique (up to a diffeomorphism) total space graded manifold $\frJ^{k}_{\E}$, together with a surjective submersion $\varpi_{k}: \frJ^{k}_{\E} \rightarrow \M$. In particular, one can use the trivial ``line bundle'' $\E = \M \times \R$ to obtain the graded manifold
\begin{equation}
\frJ^{k}_{[\M]} := \frJ^{k}_{\M \times \R},
\end{equation}
called the \textbf{$k$-th order jet manifold of $\M$}. Recall that $\Gamma_{\M \times \R} := \C^{\infty}_{\M}$. 
\end{rem}
\section{Graded jet bundles: properties} \label{sec_gJetprops}
So far it is not clear why $\Jet^{k}_{\E}$ should be called a sheaf of $k$-th order jets. We will now attempt to justify this construction by proving its properties which are obvious generalizations of their counterparts in ordinary geometry. Let us start by making the following list of expected results:
\begin{enumerate}[(1)]
\item For each $k \in \N_{0}$, there is a canonical ``jet prolongation'' sheaf morphism $\frj^{k}: \Gamma_{\E} \rightarrow \Jet^{k}_{\E}$.
\item $\frj^{0}$ is a $\C^{\infty}_{\M}$-linear sheaf isomorphism. Consequently, there is a canonical isomorphism $\E \cong \frJ^{0}_{\E}$. 
\item For each $k \in \N_{0}$, there is a canonical isomorphism $\frD^{k}_{\E} \cong \ul{\Hom}( \frJ^{k}_{\E}, \E)$. 
\item For any $\ell \leq k$, there is a surjective graded vector bundle map $\pi^{k,\ell}: \frJ^{k}_{\E} \rightarrow \frJ^{\ell}_{\E}$. 
\item The fiber of $\frJ^{k}_{\E}$ at any $a \in M$ can be interpreted as a graded set of equivalence classes of germs of sections whose components have the same Taylor polynomials of the order $k$ at $a$. 
\item For ordinary vector bundles over ordinary manifolds, this gives ordinary jet bundles.
\end{enumerate}
We will now discuss these items one by one.
\begin{enumerate}[(1)]
\item \textbf{Jet prolongations}. For each $k \in \N_{0}$, let us first construct a presheaf morphism $\hat{\frj}^{k}: \Gamma_{\E} \rightarrow \gpJet^{k}_{\E}$ as follows. For each $U \in \Op(M)$ and $\psi \in \Gamma_{\E}(U)$, set
\begin{equation} \label{eq_hatjmap}
\hat{\frj}^{k}_{U}(\psi) := [1 \otimes_{(k)} \psi]^{\bullet} 
\end{equation}
This map is $\C^{\infty}_{\M}(U)$-linear with respect to the $\btr$ action on $\gpJet^{k}_{\E}(U)$, but not with respect to the other action $\tr$. It is natural in $U$, whence it defines a presheaf morphism $\hat{\frj}^{k}: \Gamma_{\E} \rightarrow \gpJet^{k}_{\E}$. 

The sheaf morphism $\frj^{k}: \Gamma_{\E} \rightarrow \Jet^{k}_{\E}$ is then obtained as $\frj^{k} := \sff \circ \hat{\frj}^{k}$, where $\sff: \gpJet^{k}_{\E} \rightarrow \Jet^{k}_{\E}$ is a canonical presheaf morphism coming from the sheafification procedure. For each $\psi \in \Gamma_{\E}(U)$, $\frj^{k}_{U}(\psi)$ is called the \textbf{$k$-th order jet prolongation of the section $\psi$}. 

Note that the image of $\frj^{k}$ ``locally generates'' $\Jet^{k}_{\E}$. Indeed, for every point $m \in M$, there is $U \in \Op_{m}(M)$ where $\Jet^{k}_{\E}|_{U} \cong \gpJet^{k}_{\E}|_{U}$. For any $V \in \Op(U)$, every element of $\Jet^{k}_{\E}(V)$ can be then written as a combination $f^{j} \tr \frj^{k}_{V}(\psi_{j})$ for some finite collection $\{ \psi_{j} \}_{j} \subseteq \Gamma_{\E}(V)$ and functions $f^{j} \in \C^{\infty}_{\M}(V)$. 

\item \textbf{Degree zero jets.} First, observe that already $\pJet^{0}_{\E}$ is a geometric sheaf isomorphic to $\Gamma_{\E}$. To see this, note that for each $U \in \Op(M)$, one has 
\begin{equation}
\V_{(0)}(U) = \R \{ (f \cdot g) \otimes \psi - (-1)^{|f||g|} f \otimes (g \cdot \psi) \; | \; f,g \in \C^{\infty}_{\M}(U), \; \psi \in \Gamma_{\E}(U) \}.
\end{equation}
It follows that the quotient is precisely the tensor product of two graded $\C^{\infty}_{\M}(U)$-modules:
\begin{equation}
\pJet^{0}_{\E}(U) \equiv \cX(U) / \V_{(0)}(U) \equiv \C^{\infty}_{\M}(U) \otimes_{\C^{\infty}_{\M}(U)} \Gamma_{\E}(U). 
\end{equation}
Since $\C^{\infty}_{\M}(U)$ is the unit with respect to the monoidal product $\otimes_{\C^{\infty}_{\M}(U)}$, there must be a canonical isomorphism of the graded $\C^{\infty}_{\M}(U)$-modules $\Gamma_{\E}(U)$ and $\pJet^{0}_{\E}(U)$. It is given by
\begin{equation}
\hat{\frj}'^{0}_{U}(\psi) := 1 \otimes_{(0)} \psi.
\end{equation}
This is easily seen to define a $\C^{\infty}_{\M}$-linear presheaf morphism $\hat{\frj}'^{0}: \Gamma_{\E} \rightarrow \pJet^{0}_{\E}$. Note that $\tr$ and $\btr$ coincide. In particular, $\pJet^{0}_{\E}$ is in fact a geometric sheaf, see Proposition \ref{tvrz_geo} and Proposition \ref{tvrz_geoprops}. Hence $\hat{\frj}^{0} \equiv \geo \circ \hat{\frj}'^{0}$ is a $\C^{\infty}_{\M}$-linear sheaf isomorphism of $\Gamma_{\E}$ and $\gpJet^{0}_{\E}$ and $\frj^{0} \equiv \sff \circ \hat{\frj}^{0}$ is a $\C^{\infty}_{\M}$-linear sheaf isomorphism of $\Gamma_{\E}$ and $\Jet^{0}_{\E}$.
\begin{tvrz}
$\frj^{0}: \Gamma_{\E} \rightarrow \Jet^{0}_{\E}$ is a $\C^{\infty}_{\M}$-linear sheaf isomorphism, which can be viewed as a canonical graded vector bundle isomorphism $\frj^{0}: \E \rightarrow \frJ^{0}_{\E}$. 
\end{tvrz}

\item \textbf{Factorization of differential operators}. 
There is a crucial consequence of Proposition \ref{tvrz_diffopfactorization}. 
\begin{tvrz} \label{tvrz_diffopfactorization2}
Let $k \in \N_{0}$ and $U \in \Op(M)$ be arbitrary. Let $D \in \Dif^{k}_{\E}(U)$. Then there is a unique $\C^{\infty}_{\M}(U)$-linear map $\frj^{k}[D]: \Jet^{k}_{\E}(U) \rightarrow \Gamma_{\E}(U)$ satisfying $D = \frj^{k}[D] \circ \frj^{k}_{U}$. 
\end{tvrz}
\begin{proof}
Let $D \in \Dif^{k}_{\E}(U)$. Note that in Proposition \ref{tvrz_diffopfactorization}, we have in fact constructed a \textit{unique} $\C^{\infty}_{\M}(U)$-linear map $\frj^{k}_{0}[D]: \gpJet^{k}_{\E}(U) \rightarrow \Gamma_{\E}(U)$ satisfying the relation 
\begin{equation}
\frj^{k}_{0}[D] \circ \hat{\frj}^{k}_{U} = D. 
\end{equation}
One can define a $\C^{\infty}_{\M}|_{U}$-linear presheaf morphism $\ol{\frj^{k}_{0}[D]}: \gpJet^{k}_{\E}|_{U} \rightarrow \Gamma_{\E}|_{U}$ by declaring
\begin{equation}
\ol{\frj^{k}_{0}[D]}_{V} := \frj^{k}_{0}[D|_{V}],
\end{equation}
for each $V \in \Op(M)$. There thus exists a unique $\C^{\infty}_{\M}|_{U}$-linear sheaf morphism 
\begin{equation}
\ol{\frj^{k}[D]}: \Jet^{k}|_{U} \rightarrow \Gamma_{\E}|_{U}
\end{equation}
having the property $\ol{\frj^{k}[D]} \circ \sff|_{U} = \ol{\frj^{k}_{0}[D]}$. We declare $\frj^{k}[D] := \ol{\frj^{k}[D]}_{U}$. But then 
\begin{equation}
\frj^{k}[D] \circ \frj^{k}_{U} = \ol{\frj^{k}[D]}_{U} \circ (\sff_{U} \circ \hat{\frj}^{k}_{U}) = \ol{ \frj^{k}_{0}[D]}_{U} \circ \hat{\frj}^{k}_{U} = \frj^{k}_{0}[D] \circ \hat{\frj}^{k}_{U} = D. 
\end{equation}

Let $F: \Jet^{k}_{\E}(U) \rightarrow \Gamma_{\E}(U)$ be another $\C^{\infty}_{\M}(U)$-linear map having the required property. By a straightforward generalization of Proposition \ref{tvrz_Locglobalsheaf}, there exists a unique sheaf morphism $\ol{F}: \Jet^{k}_{\E}|_{U} \rightarrow \Gamma_{\E}|_{U}$ satisfying $F = \ol{F}_{U}$. We claim that for any $V \in \Op(U)$, one has $\ol{F}_{V} \circ \frj^{k}_{V} = D|_{V}$. First, observe that the left-hand side is in $\Loc_{\E}(V)$. This is because it is $\R$-linear and whenever $\psi|_{W} = 0$, one finds
\begin{equation}
(\ol{F}_{V} \circ \frj^{k}_{V})(\psi)|_{W} = (\ol{F}_{W} \circ \frj^{k}_{W})(\psi|_{W}) = 0. 
\end{equation}
Moreover, for every $\psi \in \Gamma_{\E}(U)$, one has 
\begin{equation}
(\ol{F}_{V} \circ \frj^{k}_{V})(\psi|_{V}) = (\ol{F}_{U} \circ \frj^{k}_{U})(\psi)|_{V} = (F \circ \frj^{k}_{U})(\psi)|_{V} = D(\psi)|_{V}. 
\end{equation}
The uniqueness of a local operator satisfying (\ref{eq_Flocalrestrictionproperty}) now implies $\ol{F}_{V} \circ \frj^{k}_{V} = D|_{V}$. Whence
\begin{equation}
(\ol{F}_{V} \circ \sff_{V}) \circ \hat{\frj}^{k}_{V} = D|_{V},
\end{equation}
which in turn implies $\ol{F}_{V} \circ \sff_{V} = \frj^{k}_{0}[D|_{V}]$ by the uniqueness claim of Proposition \ref{tvrz_diffopfactorization}. Since $V \in \Op(U)$ is arbitrary, this means that $\ol{F} \circ \sff|_{U} = \ol{\frj^{k}_{0}[D]}$ and thus $\ol{F} =  \ol{\frj^{k}[D]}$. This obviously implies $F = \frj^{k}[D]$ and the uniqueness claim follows. 
\end{proof}
One may utilize this statement to prove that this correspondence is one-to-one and well-behaved with respect to restrictions and graded module structures. 
\begin{tvrz} \label{tvrz_frDkashomfrJktoE}
For each $k \in \N_{0}$, there exists a canonical graded vector bundle isomorphism 
\begin{equation} 
\Psi^{\E}: \frD^{k}_{\E} \rightarrow \ul{\Hom}(\frJ^{k}_{\E},\E). 
\end{equation}
\end{tvrz}
\begin{proof}
A graded vector bundle isomorphism is a sheaf isomorphism of the corresponding sheaves of sections. For each $U \in \Op(M)$, one thus has to construct a $\C^{\infty}_{\M}(U)$-linear map 
\begin{equation}
\Psi^{\E}_{U}: \Dif^{k}_{\E}(U) \rightarrow \Lin^{\C^{\infty}_{\M}(U)}( \Jet^{k}_{\E}(U), \Gamma_{\E}(U)), 
\end{equation}
and prove that it is natural in $U$. For each $D \in \Dif^{k}_{\E}(U)$, we use Proposition \ref{tvrz_diffopfactorization2} and define
\begin{equation}
\Psi^{\E}_{U}(D) := \frj^{k}[D].
\end{equation}
First of all, one must argue that it is $\C^{\infty}_{\M}(U)$-linear in $D$. But this follows immediately from the uniqueness claim of Proposition \ref{tvrz_diffopfactorization2}. Next, let $V \in \Op(U)$. We claim that $\frj^{k}[D]|_{V} \circ \frj^{k}_{V} = D|_{V}$. It is easy to check that the left-hand side is in $\Loc_{\E}(V)$, and for any $\psi \in \Gamma_{\E}(U)$, one has 
\begin{equation}
(\frj^{k}[D]|_{V} \circ \frj^{k}_{V})(\psi|_{V}) = (\frj^{k}[D] \circ \frj^{k}_{U})(\psi)|_{V} = D(\psi)|_{V}.
\end{equation}
Hence, by Proposition \ref{tvrz_locsheaf}, we have $\frj^{k}[D]|_{V} \circ \frj^{k}_{V} = D|_{V}$. But it follows from the uniqueness claim of Proposition \ref{tvrz_diffopfactorization2} that necessarily $\frj^{k}[D|_{V}] = \frj^{k}[D]|_{V}$. This can be translated as 
\begin{equation}
\Psi^{\E}_{V}(D|_{V}) = \Psi^{\E}_{U}(D)|_{V},
\end{equation}
for all $D \in \Dif^{k}_{\E}(U)$. This proves that $\Psi^{\E}$ is a $\C^{\infty}_{\M}$-linear sheaf morphism. 

It remains to prove that is is an isomorphism. For each $U \in \Op(M)$, we will construct the inverse map $(\Psi^{\E}_{U})^{-1}$. If $F: \Jet^{k}_{\E}(U) \rightarrow \Gamma_{\E}(U)$ is $\C^{\infty}_{\M}(U)$-linear of degree $|F|$, we define
\begin{equation}
(\Psi^{\E}_{U})^{-1}(F) := F \circ \frj^{k}_{U}. 
\end{equation}
We must argue that the right-hand side is an element of $\Dif^{k}_{\E}(U)$. To do so, for any $\C^{\infty}_{\M}(U)$-linear map  $\hat{F}: \gpJet^{k}_{\E}(U) \rightarrow \Gamma_{\E}(U)$ of degree $|F|$, let $K[\hat{F}] := \hat{F} \circ \hat{\frj}^{k}_{U}: \Gamma_{\E}(U) \rightarrow \Gamma_{\E}(U)$. For any $f \in \C^{\infty}_{\M}(U)$, it is straightforward to prove the relation
\begin{equation}
K[\hat{F}]^{(1)}_{(f)} = -K[\hat{F} \circ \delta_{f}],
\end{equation}
This can be easily iterated to prove that for every $f_{1},\dots,f_{k+1} \in \C^{\infty}_{\M}(U)$, one has 
\begin{equation}
K[\hat{F}]^{(k+1)}_{(f_{1},\dots,f_{k+1})} = (-1)^{k+1} K( \hat{F} \circ \delta^{(k+1)}_{(f_{1},\dots,f_{k+1})}) = 0. 
\end{equation}
We mildly abuse the notation and use the symbol $\delta_{f}$ also for the induced map on $\gpJet^{k}_{\E}(U)$. This proves that $K[\hat{F}] \in \Dif^{k}_{\E}(U)$. Consequently, observe that $(\Psi^{\E}_{U})^{-1}(F) = (F \circ \sff_{U}) \circ \hat{\frj}^{k}_{U} = K[F \circ \sff_{U}] \in \Dif^{k}_{\E}(U)$. The fact that $(\Psi^{\E}_{U})^{-1}$ is a two-sided inverse to $\Psi_{U}^{\E}$ is obvious. 
\end{proof} 
\item \textbf{The inverse system over $\N_{0}$}. Suppose $\ell \leq k$ be two non-negative integers. We aim to construct a surjective $\C^{\infty}_{\M}$-linear sheaf morphism $\pi^{k,\ell}: \Jet^{k}_{\E} \rightarrow \Jet^{\ell}_{\E}$.

For each $U \in \Op(M)$, let us define a $\C^{\infty}_{\M}(U)$-linear map $\hat{\pi}^{k,\ell}_{U}: \gpJet^{k}_{\E}(U) \rightarrow \gpJet^{\ell}_{\E}(U)$ as
\begin{equation}
\hat{\pi}^{k,\ell}_{U}[ g \otimes_{(k)} \psi]^{\bullet} := [g \otimes_{(\ell)} \psi]^{\bullet},
\end{equation}
for each $g \in \C^{\infty}_{\M}(U)$ and $\psi \in \Gamma_{\E}(U)$. It is easy to check that it is well-defined, since $\V_{(k)}(U) \subseteq \V_{(\ell)}(U)$. It is $\C^{\infty}_{\M}(U)$-linear with respect to both $\tr$ and $\btr$, and it is natural in $U$. It defines a $\C^{\infty}_{\M}$-linear presheaf morphism fitting into the commutative diagram
\begin{equation}
\begin{tikzcd}
& \cX \arrow{ld}[swap]{\natural^{(k)}} \arrow{rd}{\natural^{(\ell)}} & \\
\pJet^{k}_{\E} \arrow{d}{\geo} & & \pJet^{\ell}_{\E} \arrow{d}{\geo} \\
\gpJet^{k}_{\E} \arrow[dashed]{rr}{\hat{\pi}^{k,\ell}} & & \gpJet^{\ell}_{\E} 
\end{tikzcd}.
\end{equation} 
Since all vertical arrows are surjective presheaf morphisms, so is $\hat{\pi}^{k,\ell}$. Finally, one can construct $\pi^{k,\ell}: \Jet^{k}_{\E} \rightarrow \Jet^{\ell}_{\E}$ as a unique $\C^{\infty}_{\M}$-linear morphism fitting into the commutative diagram
\begin{equation} \label{eq_piklfinaldiagram}
\begin{tikzcd}
\gpJet^{k}_{\E} \arrow{d}{\sff} \arrow{r}{\hat{\pi}^{k,\ell}} & \gpJet^{\ell}_{\E} \arrow{d}{\sff} \\
\Jet^{k}_{\E} \arrow[dashed]{r}{\pi^{k,\ell}} & \Jet^{\ell}_{\E} 
\end{tikzcd}
\end{equation}
One only has to argue that $\pi^{k,\ell}$ is surjective. Since for sheaves of graded $\C^{\infty}_{\M}$-modules, cokernels happen to be sheaves, one can argue that $\pi^{k,\ell}$ is surjective, iff the induced map $(\pi^{k,\ell})_{m}$ of stalks is surjective for every $m \in M$. Stalks of $\gpJet^{k}_{\E}$ and its sheafification $\Jet^{k}_{\E}$ are canonically identified via the isomorphism of the stalks induced by $\sff$. The claim thus follows from the surjectivity of the morphism $(\hat{\pi}^{k,\ell})_{m}$ for every $m \in M$, ensured by the surjectivity of $\hat{\pi}^{k,\ell}$. 

By definition, $\pi^{k,\ell}$ can be interpreted as a surjective degree zero graded vector bundle map $\pi^{k,\ell}: \frJ^{k}_{\E} \rightarrow \frJ^{\ell}_{\E}$ over the identity. We can examine additional properties of these maps.

\begin{tvrz}
For each $\ell \leq k$, let $\pi^{k,\ell}: \Jet^{k}_{\E} \rightarrow \Jet^{\ell}_{\E}$ be constructed as above. 
\begin{enumerate}[(i)]
\item For each $\ell \leq k$, $\pi^{k,\ell}$ fits into the equation
\begin{equation} \label{eq_pisandjs}
\pi^{k,\ell} \circ \frj^{k} = \frj^{\ell}. 
\end{equation}
\item $( \{ \Jet^{k}_{\E} \}_{k}, \{ \pi^{k,\ell} \}_{\ell \leq k} )$ forms an inverse system over $\N_{0}$ in the category of sheaves of graded $\C^{\infty}_{\M}$-modules. More explicitly, the ``transition maps'' $\pi^{k,\ell}$ satisfy the the conditions
\begin{equation} \label{eq_transitionmaps}
\pi^{k,k} = \1_{\Jet^{k}_{\E}}, \; \; \pi^{\ell,q} \circ \pi^{k,\ell} = \pi^{k,q} \text{ for any } q \leq \ell \leq k.
\end{equation}
\item There is a sheaf $\Jet^{\infty}_{\E}$ of graded $\C^{\infty}_{\M}$-modules, defined as an inverse limit of the above inverse system:
\begin{equation} \label{eq_Jetinfty}
\Jet^{\infty}_{\E} := \Ilim \Jet^{k}_{\E},
\end{equation}
together with a collection of surjective degree zero $\C^{\infty}_{\M}$-linear sheaf morphisms $\pi^{\infty,k}: \Jet^{\infty}_{\E} \rightarrow \Jet^{k}_{\E}$, such that for every $\ell \leq k$, one has  
\begin{equation}
\pi^{k,\ell} \circ \pi^{\infty,k} = \pi^{\infty,\ell}.
\end{equation}
\item For each $k \in \N$, one has a short exact sequence of graded vector bundles over $\M$ and vector bundle maps over $\1_{\M}$
\begin{equation} \label{eq_SESwithjets}
\begin{tikzcd}
0 \arrow{r} & \ul{\Hom}( S^{k}(T\M), \E) \arrow{r}{J_{(k)}} & \frJ^{k}_{\E} \arrow{r}{\pi^{k,k-1}} &[2em] \frJ^{k-1}_{\E} \arrow{r} & 0 
\end{tikzcd}.
\end{equation}
If one declares $\frJ^{-1}_{\E} = 0$, the statement is true also for $k = 0$. 
\end{enumerate}
\end{tvrz}
\begin{proof}
To prove $(i)$, note that for any $\ell \leq k$, one has $\hat{\pi}^{k,\ell} \circ \hat{\frj}^{k} = \hat{\frj}^{\ell}$, where $\hat{\frj}^{k}: \Gamma_{\E} \rightarrow \gpJet^{k}_{\E}$ are defined by (\ref{eq_hatjmap}). Since $\frj^{k}$ is uniquely determined by the equation $\frj^{k} = \sff \circ \hat{\frj}^{k}$ and $\pi^{k,\ell}$ are uniquely determined by (\ref{eq_piklfinaldiagram}), the equation (\ref{eq_pisandjs}) follows. 

The equations in (\ref{eq_transitionmaps}) follow immediately from the same ones satisfied by the presheaf morphisms $\hat{\pi}^{k,\ell}$, together with the uniqueness of the map completing the diagram (\ref{eq_piklfinaldiagram}). This proves $(ii)$. Next, it is easy to see that all products and equalizers exist in the category of sheaves of graded $\C^{\infty}_{\M}$-modules. This implies that all limits exist in this category. Consequently, the inverse limit (\ref{eq_Jetinfty}) exists. The collection $\{ \pi^{\infty,k} \}_{k \in \N_{0}}$ is then nothing but its universal limiting cone. 

We have to argue that $\pi^{\infty,k}$ are surjective. To do so, observe that one can describe $\Jet^{\infty}_{\E}$ explicitly. Indeed, for each $U \in \Op(M)$ and $i \in \Z$, one has 
\begin{equation}
(\Jet^{\infty}_{\E}(U))_{i} = \{ (\sigma_{\ell})_{\ell \in \N_{0}} \in \prod_{k \in \Z} (\Jet^{\ell}_{\E}(U))_{i} \; | \; \pi^{q,s}_{U}(\sigma_{q}) = \sigma_{s} \text{ for all } s \leq q \}.
\end{equation}
Moreover, one has $\pi^{\infty,k}_{U}[(\sigma_{\ell})_{\ell \in \Z}] := \sigma_{k}$ for each $k \in \N_{0}$. 

Since $\Jet^{\infty}_{\E}$ is a sheaf of graded $\C^{\infty}_{\M}$-modules, it suffices to argue that for each $m \in M$, there exists $U \in \Op_{m}(M)$, such that $\pi^{\infty,k}_{U}: \Jet^{\infty}_{\E}(U) \rightarrow \Jet^{k}_{\E}(U)$ is surjective. As noted in the last paragraph of (1), for each $m \in M$, there exists $U \in \Op_{m}(M)$, such that $\Jet^{k}_{\E}(U)$ is generated by the image of $\frj^{k}_{U}: \Gamma_{\E}(U) \rightarrow \Jet^{k}_{\E}(U)$. Let $\sigma_{k} \in \Jet^{k}_{\E}(U)$ be an arbitrary given section. We can thus find a finite collection of sections $\{ \psi_{j} \}_{j}$ in $\Gamma_{\E}(U)$, such that 
\begin{equation}
\sigma_{k} = f^{j} \tr \frj^{k}_{U}(\psi_{j}) 
\end{equation}
for some functions $f^{j} \in \C^{\infty}_{\M}(U)$ of degree $|f^{j}| = |\sigma_{k}| - |\psi_{j}|$. For every $\ell \in \N_{0}$, let us define 
\begin{equation}
\sigma_{\ell} := f^{j} \tr \frj^{\ell}_{U}(\psi_{j}) \in \Jet^{\ell}_{\E}(U). 
\end{equation}
Now, suppose that $s \leq q$ is a pair of arbitrary integers. Since $\pi^{q,s}_{U}$ is $\C^{\infty}_{\M}(U)$-linear, we can utilize the equation (\ref{eq_pisandjs}) to obtain 
\begin{equation}
\pi^{q,s}_{U}(\sigma_{q}) = \pi^{q,s}_{U}( f^{j} \tr \frj^{q}_{U}(\psi_{j})) = f^{j} \tr (\pi^{q,s}_{U} \circ \frj^{q}_{U})(\psi_{j}) = f^{j} \tr \frj^{s}_{U}(\psi_{j}) = \sigma_{s}. 
\end{equation}
This proves that $(\sigma_{\ell})_{\ell \in \N_{0}} \in \Jet^{\infty}_{\E}(U)$. By construction, one has $\pi^{\infty,k}_{U}[(\sigma_{\ell})_{\ell \in \Z}] = \sigma_{k}$. Since $\sigma_{k}$ was arbitrary, this proves that $\pi^{\infty,k}_{U}$ is surjective. This concludes the proof of $(iii)$. 

To prove $(iv)$, we have to construct a fiber-wise injective graded vector bundle map 
\begin{equation}
J_{(k)}: \ul{\Hom}(S^{k}(T\M), \E) \rightarrow \frJ^{k}_{\E},
\end{equation}
for every $k \in \N$. Suppose that $(U,\varphi)$ is a graded local chart for $\M$ and let $\{ \Phi_{\lambda} \}_{\lambda=1}^{r}$ be a local frame for $\E$ over $U$. We will construct a degree zero $\C^{\infty}_{\M}(U)$-linear map 
\begin{equation}
(J_{(k)})_{U}: \Lin^{\C^{\infty}_{\M}(U)}( \~\X_{\M}(U), \Gamma_{\E}(U)) \rightarrow \Jet^{k}_{\E}(U). 
\end{equation}
Recall that $\~\X_{\M}(U)$ is freely and finitely generated by a collection $\{ \partial^{\op}_{\fI} \}_{\fI \in \ol{\N}{}^{n}(k)}$, see (\ref{eq_partialIop}). Suppose $\phi: \~\X_{\M}(U) \rightarrow \Gamma_{\E}(U)$ is a given $\C^{\infty}_{\M}(U)$-linear map of degree $|\phi|$. We declare
\begin{equation} \label{eq_Jkmap}
(J_{(k)})_{U}(\phi) := \sum_{\fI \in \ol{\N}{}^{n}(k)} \frac{1}{\fI!} (-1)^{|\vartheta_{\lambda}|(1+|\phi|) + |\bbz^{\fI}|} \Phi^{\lambda}( \phi(\partial_{\fI}^{\op})) \tr \fdelta^{\fI}{}_{\lambda},
\end{equation}
where $\fdelta^{\fI}{}_{\lambda} \in \Jet^{k}_{\E}(U)$ are the generators defined by (\ref{eq_fdelta}) and (\ref{eq_fdeltafI}). See also Theorem \ref{thm_generatorsgpJet}. It is easy to see that $(J_{(k)})_{U}(\phi)$ is $\C^{\infty}_{\M}(U)$-linear in $\phi$. If $(U',\varphi')$ is a different graded local chart for $\M$ and $\{ \Phi'_{\lambda} \}_{\lambda=1}^{r}$ is a different local frame for $\E$ over $U'$, one has to verify that the definitions (\ref{eq_Jkmap}) agree on $U \cap U'$. To do so, one has to derive how $\partial_{\fI}^{\op}$ and $\fdelta^{\fI}{}_{\lambda}$ transform on $U \cap U'$ under the change of coordinates (and local frames). To derive the latter, one has to employ (\ref{eq_fdeltaformula}) together with (\ref{eq_fdeltajgsymmetry}). The calculation is straightforward, albeit a bit tedious, so we omit it here. 

Now, the local frame for $\ul{\Hom}(S^{k}(T\M), \E)$ over $U$ can be chosen to be the ``standard basis'' $\{ \frE^{\fI}{}_{\lambda} \}$, see (\ref{eq_frEstandardbasis}), acting on each $X = \frac{1}{\fI!} X^{\fI} \cdot \partial^{\op}_{\fI} \in \~\X^{k}_{\M}(U)$ as 
\begin{equation}
\frE^{\fI}{}_{\lambda}(X) = \frac{1}{\fI!} (-1)^{(|\vartheta_{\lambda}| - |\bbz^{\fI}|)(|X| + |\bbz^{\fI}|)} X^{\fI} \cdot \Phi_{\lambda}. 
\end{equation}
In other words, one has $\frE^{\fI}{}_{\lambda}(\partial_{\fJ}^{\op}) = \delta_{\fJ}^{\fI} \cdot \Phi_{\lambda}$, for every $\fI,\fJ \in \ol{\N}{}^{n}(k)$ and $\lambda \in \{1, \dots, r \}$. Plugging this into (\ref{eq_Jkmap}), one finds the expression
\begin{equation}
(J_{(k)})_{U}(\frE^{\fI}{}_{\lambda}) = \frac{1}{\fI!} (-1)^{|\bbz^{\fI}|(1 + |\vartheta_{\lambda}|)} \fdelta^{\fI}{}_{\lambda},
\end{equation}
for every $\fI \in \ol{\N}{}^{n}(k)$ and $\lambda \in \{1, \dots, r\}$. This proves that $J_{(k)}$ is fiber-wise injective and its image is precisely the kernel of $\pi^{k,k-1}$. The exactness of (\ref{eq_SESwithjets}) follows. 

Finally, for $k = 0$, we have $\~X^{0}_{\M} \cong \C^{\infty}_{\M}$ and a canonical vector bundle isomorphism $\E \cong \ul{\Hom}(S^{0}(T\M),\E)$, sending each $\psi \in \Gamma_{\E}(U)$ to a $\C^{\infty}_{\M}(U)$-linear map $\phi(f) := (-1)^{|\psi||f|} f \cdot \psi$. Observe that the only element of $\ol{\N}{}^{n}(0)$ is $\fnula$, we declare $\partial_{\fnula}^{\op} := 1$, and $\fdelta^{\fnula}{}_{\lambda} = \frj^{0}_{U}(\Phi_{\lambda})$. One has 
\begin{equation}
(J_{(0)})_{U}( \frE^{\fnula}{}_{\lambda}) = \frj^{0}_{U}(\Phi_{\lambda}),
\end{equation}
if we use (\ref{eq_Jkmap}) also for the $k = 0$ case. This means that $J_{(0)}$ is the unique vector bundle isomorphism fitting into the commutative diagram
\begin{equation}
\begin{tikzcd}
& \arrow{ld}[swap]{\cong} \E \arrow{rd}{\frj^{0}} & \\
\ul{\Hom}(S^{0}(T\M),\E) \arrow[dashed]{rr}{J_{(0)}} & & \frJ^{0}_{\E}
\end{tikzcd}
\end{equation}
Hence (\ref{eq_SESwithjets}) is exact also for $k = 0$, if we declare $\frJ^{-1}_{\E} = 0$. This finishes the proof. 
\end{proof}
\begin{rem}
For a non-trivial $\E$, $\Jet^{\infty}_{\E}$ is not a sheaf of sections of a graded vector bundle, except when $\M$ has only odd coordinates. Indeed, suppose $\Jet^{\infty}_{\E} = \Gamma_{\frJ^{\infty}_{\E}}$ for a graded vector bundle $\frJ^{\infty}_{\E}$. Then for each $k \in \N_{0}$, we have a surjective vector bundle map $\pi^{\infty,k}: \frJ^{\infty}_{\E} \rightarrow \frJ^{k}_{\E}$, which forces the inequality $\trk( \frJ^{\infty}_{\E}) \geq \trk( \frJ^{k}_{\E})$. On the other hand, one has
\begin{equation}
\trk( \frJ^{k}_{\E}) = \trk( \frJ^{k-1}_{\E}) + \trk( \ul{\Hom}(S^{k}(T\M),\E)),
\end{equation}
which follows from the exactness of (\ref{eq_SESwithjets}). This shows that $\trk(\frJ^{k}_{\E})$ grows strictly in $k$, iff $\trk( S^{k}(T\M))$ is non-zero for all $k$. If $\M$ has some coordinates of even degree (including $0$), this is the case. The above inequality then proves that $\frJ^{\infty}_{\E}$ cannot have a finite total rank, which cannot happen (in our setting).

The only exception is $\M$ having all coordinates odd. In particular, one has $M = \{ \ast \}$. Let $n := \sum_{j \in \Z} n_{j}$, where $(n_{j})_{j \in \Z} := \gdim(\M)$. Since $n_{j} = 0$ whenever $j$ is even, it is easy to see that $S^{k}(T\M) = 0$ for all $k > n$. It follows from the exactness of (\ref{eq_SESwithjets}) that $\pi^{k,n}: \Jet^{k}_{\E} \rightarrow \Jet^{n}_{\E}$ are $\C^{\infty}_{\M}$-linear sheaf isomorphisms for all $k > n$. But this ensures that $\pi^{\infty,n}: \Jet^{\infty}_{\E} \rightarrow \Jet^{n}_{\E}$ is a $\C^{\infty}_{\M}$-linear sheaf isomorphism. We see that in this case $\frJ^{\infty}_{\E}$ is a sheaf of sections of a graded vector bundle isomorphic to $\frJ^{n}_{\E}$. 
\end{rem}
\item \textbf{Fibers are jet spaces}. Let $a \in M$ be a fixed point of the underlying manifold $M$. 

Consider the Jacobson radical $\J_{\M,a} \subseteq \C^{\infty}_{\M,a}$ of the stalk of the sheaf $\C^{\infty}_{\M}$ at $a$, that is the unique maximal graded ideal, having the explicit form
\begin{equation}
\J_{\M,a} = \{ [f]_{a} \in \C^{\infty}_{\M,a} \; | \; f(a) = 0 \}.
\end{equation}
Let us denote the canonical action of $\C^{\infty}_{\M,a}$ on $\Gamma_{\E,a}$ as $\btr$, that is let $[f]_{a} \btr [\psi]_{a} := [f \cdot \psi]_{a}$, for all $[f]_{a} \in \C^{\infty}_{\M,a}$ and $[\psi]_{a} \in \Gamma_{\E,a}$. The reason for this will become clear later. There is another action $\tr$, given for each $[f]_{a} \in \C^{\infty}_{\M,a}$ and $[\psi]_{a} \in \Gamma_{\E,a}$ by
\begin{equation}
[f]_{a} \tr [\psi]_{a} := f(a) \cdot [\psi]_{a}. 
\end{equation}
For each $k \in \N_{0}$, consider the graded $\C^{\infty}_{\M,a}$-submodule $\Gamma_{\E,a}^{[k]}$ of the stalk $\Gamma_{\E,a}$ defined as 
\begin{equation}
\Gamma^{[k]}_{\E,a} := (\J_{\M,a})^{k+1} \btr \Gamma_{\E,a},
\end{equation}
that is the graded $\C^{\infty}_{\M,a}$-submodule (with respect to $\btr$) generated by elements of the form $[f]_{a} \btr [\psi]_{a}$, where $[f]_{a} \in (\J_{\M,a})^{k+1}$ and $[\psi]_{a} \in \Gamma_{\E,a}$. Note that $\Gamma^{[k]}_{\E,a}$ forms a graded submodule also with respect to the action $\tr$.

For each $k \in \N_{0}$, define the \textbf{$k$-th order jet space $\frJ^{k}_{\E,a}$ of $\E$ at $a$} as the quotient $\C^{\infty}_{\M,a}$-module
\begin{equation}
\frJ^{k}_{\E,a} := \Gamma_{\E,a} / \Gamma^{[k]}_{\E,a},
\end{equation}
and denote the respective quotient map as $\frj^{k}_{a}: \Gamma_{\E,a} \rightarrow \frJ^{k}_{\E,a}$. We say that the $[\psi]_{a}, [\psi']_{a} \in \Gamma_{\E,a}$ \textbf{have the same $k$-th order jet at $a$}, if 
\begin{equation}
\frj^{k}_{a}[\psi]_{a} = \frj^{k}_{a}[\psi']_{a}. 
\end{equation}
Note that $\frJ^{k}_{\E,a}$ inherits both $\C^{\infty}_{\M,a}$ actions, which we shall again denote as $\btr$ and $\tr$. 
Let us now argue that this definition can be restated in a more standard way.
\begin{tvrz} \label{tvrz_jetspacelocally}
Let $(U,\varphi)$ be a graded local chart for $\M$ over $U \in \Op_{a}(M)$, and let $\{ \Phi_{\lambda} \}_{\lambda =1}^{r}$ be a local frame for $\E$ over $U$. Let $k \in \N_{0}$.

Suppose $[\psi]_{a} \in \Gamma_{\E,a}$ is represented by $\psi \in \Gamma_{\E}(U)$ and write $\psi = \psi^{\lambda} \cdot \Phi_{\lambda}$ for unique functions $\psi^{\lambda} \in \C^{\infty}_{\M}(U)$. Then $[\psi]_{a} \in \Gamma^{[k]}_{\E,a}$, iff 
\begin{equation} \label{eq_vanishingderivativesatm}
(\partial_{\fI}^{\op} \psi^{\lambda})(a) = 0, 
\end{equation}
for all $\lambda \in \{1,\dots,r\}$ and $\fI \in \cup_{q=0}^{k} \ol{\N}{}^{n}(q)$. See Section \ref{sec_diffopVB} for the notation. 
\end{tvrz}
\begin{proof}
First, suppose that $[\psi]_{a} \in \Gamma^{[k]}_{\E,a}$. It follows that $[\psi^{\lambda}]_{a} \in (\J_{\M,a})^{k+1}$ for every $\lambda \in \{1,\dots,r\}$. This ideal is generated by the monomials in $[\bbz^{A}_{a}]_{a}$ of order $k+1$, where $\bbz^{A}_{a} := \bbz^{A} - \bbz^{A}(a)$. Consequently, there is $V \in \Op_{m}(U)$, such that $\psi^{\lambda}|_{V}$ is a finite linear combination of some terms of the form 
\begin{equation}
\bbz^{B_{1}}_{a} \cdots \bbz^{B_{k+1}}_{a} \cdot h 
\end{equation}
for some $h \in \C^{\infty}_{\M}(V)$. For any $\fI \in \cup_{q=0}^{k} \ol{\N}{}^{n}(q)$, we have $w(\fI) < k+1$ and thus clearly $(\partial^{\op}_{\fI} \psi^{\lambda})(a) = 0$. Conversely, suppose that $\psi$ representing $[\psi]_{a}$ satisfies (\ref{eq_vanishingderivativesatm}) for all $\lambda \in \{1,\dots,r\}$ and $\fI \in \cup_{q=0}^{k} \ol{\N}{}^{n}(q)$. This implies that for each $\lambda \in \{1,\dots,r\}$ , the $k$-th order Taylor polynomial of $\psi^{\lambda}$ at $a$ vanishes, $T_{a}^{k}(\psi^{\lambda}) = 0$. Consequently, one has 
\begin{equation}
\psi^{\lambda} = R_{a}^{k}(\psi^{\lambda}) \in (\J^{a}_{\M}(U))^{k+1},
\end{equation}
see Lemma 3.4 in \cite{Vysoky:2022gm}. This immediately implies $[\psi]_{a} = [\psi^{\lambda}]_{a} \cdot [\Phi_{\lambda}]_{a} \in \Gamma^{[k]}_{\E,a}$. 
\end{proof}
\begin{cor}
The collection $\{ \frj^{k}_{a}[ \bbz^{\fI}_{a} \cdot \Phi_{\lambda}]_{a} \}$, where $\fI$ goes over $\cup_{q=0}^{k} \ol{\N}{}^{n}(q)$ and $\lambda$ goes over $\{1, \dots, r\}$, forms a total basis for the graded vector space $\frJ^{k}_{\E,a}$. 

In particular, one has $\gdim( \frJ^{k}_{\E,a}) = \grk( \frJ^{k}_{\E})$. 
\end{cor}
\begin{proof}
The first statement follows easily from definitions and the previous proposition. One has to utilize a modification of (\ref{eq_partialIoponzJ}), namely the fact that
\begin{equation}
(\partial_{\fI}^{\op}(\bbz^{\fJ}_{a}))(a) = \fI! \cdot \delta_{\fI}^{\fJ},
\end{equation}
for all $\fI,\fJ \in \ol{\N}{}^{n}$. Next, note that $| \frj^{k}_{a}[ \bbz^{\fI}_{a} \cdot \Phi_{\lambda}]_{a}| = |\bbz^{\fI}| + |\vartheta_{\lambda}|$. This is the same degree as $|\fdelta^{\fI}{}_{\lambda}|$, see (\ref{eq_fdeltafI}). The second statement then immediately follows from Theorem \ref{thm_generatorsgpJet}. 
\end{proof}

Let us now argue that $\frJ^{k}_{\E,a}$ has a structure not dissimilar from the one constructed in Section \ref{sec_gJetconst}. First, to each $[f]_{a} \in \C^{\infty}_{\M,a}$, one can assign a $\C^{\infty}_{\M,a}$-linear (with respect to both $\btr$ and $\tr$) map $\delta'_{[f]_{a}}: \Gamma_{\E,a} \rightarrow \Gamma_{\E,a}$ of degree $|f|$ given by
\begin{equation} \label{eq_deltaprimeexplicit}
\delta'_{[f]_{a}}[\psi]_{a} := [f]_{a} \tr [\psi]_{a} - [f]_{a} \btr [\psi]_{a} \equiv [(f(a) - f) \cdot \psi]_{a},
\end{equation}
where we assume that the respective germs are represented by some $f \in \C^{\infty}_{\M}(U)$ and $\psi \in \Gamma_{\E}(U)$ for some $U \in \Op_{a}(M)$. We immediately see the following statement:
\begin{lemma}
One can write
\begin{equation}
\Gamma_{\E,a}^{[k]} = \R\{ \delta'_{[f_{1}]_{a}} \dots \delta'_{[f_{k+1}]_{a}}[\psi]_{a} \; | \; [f_{1}]_{a}, \dots, [f_{k+1}]_{a} \in \C^{\infty}_{\M,a}, [\psi]_{a} \in \Gamma_{\E,a} \}.
\end{equation}
Compare this formula to (\ref{eq_VkU}). 
\end{lemma}

\begin{tvrz} \label{tvrz_fibers}
For every $a \in M$, there is a canonical $\C^{\infty}_{\M,a}$-linear isomorphism of the fiber $(\frJ^{k}_{\E})_{a}$ of the graded vector bundle $\frJ^{k}_{\E}$ and the $k$-th order jet space $\frJ^{k}_{\E,a}$ of $\E$ at $a$. 
\end{tvrz} 
\begin{proof}
Recall that the fiber $\E_{a}$ of a graded vector bundle $\E$ at $a \in M$ is defined as 
\begin{equation}
\E_{a} := \R \otimes_{\C^{\infty}_{\M,a}} \Gamma_{\E,a},
\end{equation}
where the action of $\C^{\infty}_{\M,a}$ on $\R$ is given by $[f]_{a} \tr \lambda := f(a) \cdot \lambda$. It is easy to see that there is a canonical isomorphism 
\begin{equation} \label{eq_fiberasqotient}
\E_{a} \cong \Gamma_{\E,a} / (\J_{\M,a} \cdot \Gamma_{\E,a}).  
\end{equation}
For any $U \in \Op_{a}(M)$ and $\psi \in \Gamma_{\E}(U)$, \textbf{the value of $\psi$ at $a$} is the element $\psi|_{a} := 1 \otimes [\psi]_{a} \in \E_{a}$. Equivalently, this is precisely the equivalence class of $[\psi]_{a}$ in the quotient (\ref{eq_fiberasqotient}). 

To define a $\C^{\infty}_{\M,a}$-linear isomorphism $\Psi_{a}: (\frJ^{k}_{\E})_{a} \rightarrow \frJ^{k}_{\E,a}$, we thus start by defining
\begin{equation}
\Psi_{a}^{0}: \gpJet^{k}_{\E,a} \rightarrow \frJ^{k}_{\E,a}. 
\end{equation}
Fix any $U \in \Op_{a}(M)$. Let $[h \otimes_{(k)} \psi]^{\bullet}_{a}$ be the germ of the section $[h \otimes_{(k)} \psi]^{\bullet} \in \gpJet^{k}_{\E}(U)$, where $h \in \C^{\infty}_{\M}(U)$ and $\psi \in \Gamma_{\E}(U)$. We declare
\begin{equation}
\Psi_{a}^{0}[h \otimes_{(k)} \psi]^{\bullet}_{a} := h \tr \frj^{k}_{a}[\psi]_{a} \equiv h(a) \cdot \frj^{k}_{a}[\psi]_{a}. 
\end{equation}
Observe that $\Psi^{0}_{a}$ will be $\C^{\infty}_{\M,a}$-linear with respect to the actions $\tr$ on $\gpJet^{k}_{\E,a}$ and $\tr$ on $\frj^{k}_{\E,a}$, respectively. One has to check that this is a well-defined map. The definition clearly does not depend on the choice of the section $[h \otimes_{(k)} \psi]^{\bullet}$ representing the element $[h \otimes_{(k)} \psi]^{\bullet}_{a}$. It also does not depend on the representative $h \otimes_{(k)} \psi$ of the geometric class $[h \otimes_{(k)} \psi]^{\bullet}$, as every element of $\pJet^{k}_{\E}(U)^{\bullet}$ is also in $\J^{a}_{\M}(U) \tr \pJet^{k}_{\E}(U)$ and the rest is taken care of by the $\C^{\infty}_{\M,a}$-linearity of $\Psi^{0}_{a}$. Moreover, for every $f \in \C^{\infty}_{\M}(U)$, one has 
\begin{equation} \label{eq_psiadeltacommutes}
\Psi^{0}_{a}[ \delta_{f}(h \otimes_{(k)} \psi)]^{\bullet}_{a} = \delta'_{[f]_{a}}( \Psi^{0}_{a}[h \otimes_{(k)} \psi]^{\bullet}_{a}).
\end{equation}
This ensures that the image of a germ of a geometric class of an element of $\V_{(k)}(U)$ vanishes. Hence $\Psi^{0}_{a}$ is a well-defined $\C^{\infty}_{\M,a}$-linear map from $\gpJet^{k}_{\E,a}$ to $\frJ^{k}_{\E,a}$. Moreover, observe that 
\begin{equation}
\Psi^{0}_{a}( \J_{\M,a} \tr \gpJet^{k}_{\E,a}) = 0. 
\end{equation}
Finally, since $\Jet^{k}_{\E}$ is just a sheafification of $\gpJet^{k}_{\E}$, there is a canonical isomorphism $\Jet^{k}_{\E,a} \cong \gpJet^{k}_{\E,a}$ of the respective stalks, and $\Psi^{0}_{a}$ induces a degree zero $\C^{\infty}_{\M,a}$-linear map 
\begin{equation}
\Psi_{a}: (\frJ^{k}_{\E})_{a} \equiv \Jet^{k}_{\E,a} / (\J_{\M,a} \tr \Jet^{k}_{\E,a}) \rightarrow \frJ^{k}_{\E,a}
\end{equation}
The inverse to $\Psi_{a}$ is for each $[\psi]_{a} \in \Gamma_{\E,a}$ defined by 
\begin{equation}
\Psi_{a}^{-1}( \frj^{k}_{a}[\psi]_{a}) := \{ (\frj^{k})_{a}[\psi]_{a} \}|_{a},
\end{equation}
where $(\frj^{k})_{a}: \Gamma_{\E,a} \rightarrow \Jet^{k}_{\E,a}$ is the $\R$-linear mapping induced by a sheaf morphism $\frj^{k}: \Gamma_{\E} \rightarrow \Jet^{k}_{\E}$. One has to check that it is well-defined. 

Suppose $[\psi]_{a} \in \Gamma^{(k)}_{\E,a}$. But this means that there is $U \in \Op_{a}(M)$, such that $\psi \in \Gamma_{\E}(U)$ and it is a finite linear combination of sections of the form
\begin{equation}
f_{1} \cdots f_{k+1} \cdot \psi',
\end{equation}
where $f_{1},\dots,f_{k+1} \in \J^{a}_{\M}(U)$. It suffices to check that the value of the $k$-th jet prolongation $\frj^{k}_{U}( f_{1} \cdots f_{k+1} \cdot \psi')$ at $a$ vanishes. One can certainly choose $U$ so that $\Jet^{k}_{\E}(U) \cong \gpJet^{k}_{\E}(U)$ and 
\begin{equation}
\frj^{k}_{U}(f_{1} \cdots f_{k+1} \cdot \psi') = [1 \otimes_{(k)} (f_{1} \cdots f_{k+1} \cdot \psi')]^{\bullet}.
\end{equation}
By expanding the relation $\delta^{(k+1)}_{(f_{1},\dots,f_{k+1})} [1 \otimes_{(k)} \psi]^{\bullet} = 0$, it is easy to see that 
\begin{equation}
[1 \otimes_{(k)} (f_{1} \cdots f_{k+1} \cdot \psi')]^{\bullet}_{a} \in \J_{\M,a} \tr \Jet^{k}_{\E,a},
\end{equation}
which proves the claim. Hence $\Psi_{a}^{-1}$ is well-defined. It is straightforward to prove that it is a two-sided inverse to $\Psi_{a}$. This finishes the proof. 
\end{proof}
\item \textbf{Ordinary limit}. We have to check that for an ordinary vector bundle $E$ over an ordinary manifold $M$, $\frJ^{k}_{E}$ constructed in this paper is isomorphic to the usual notion of $k$-th order jet bundle. See e.g. the standard reference \cite{saunders1989geometry}. Let us denote the ``standard'' version as $\ol{\frJ}{}^{k}_{\E}$. Recall that its total space is defined to be the disjoint union of its jet spaces:
\begin{equation}
\ol{\frJ}{}^{k}_{E} := \bigsqcup_{a \in M} \frJ^{k}_{E,a}.
\end{equation}
The projection $\ol{\varpi}: \ol{\frJ}{}^{k}_{E} \rightarrow M$ is defined in an obvious way. Choose some local chart $(U,\varphi)$ and a local frame $\{ \Phi_{\lambda} \}_{\lambda=1}^{r}$ for $E$. The induced local trivialization chart $\phi: \ol{\varpi}{}^{-1}(U) \rightarrow U \times \R^{\ell}$, where $\ell := \rk(\ol{\frJ}{}^{k}_{E}) = r \cdot \binom{n+k}{k}$ is defined for each $a \in U$ by
\begin{equation}
\phi(\frj^{k}_{a}[\psi]_{a}) := (a, (( \partial^{\op}_{\fI} \psi^{\lambda})(a))_{\lambda,\fI}),
\end{equation}
where $[\psi]_{a}$ is represented by some $\psi = \psi^{\lambda} \cdot \Phi_{\lambda} \in \Gamma_{E}(U)$, and the indices in the sequence go over all $\lambda \in \{1,\dots,r\}$ and all $\fI \in \cup_{q=0}^{k} \ol{\N}{}^{n}(q)$. By Proposition \ref{tvrz_jetspacelocally}, this is well-defined. 

\begin{tvrz}
There is a canonical vector bundle isomorphism $\Psi: \frJ^{k}_{E} \rightarrow \ol{\frJ}{}^{k}_{E}$ over $\1_{M}$ for each $k \in \N_{0}$. 
\end{tvrz}
\begin{proof}
Since both $\frJ^{k}_{E}$ and $\ol{\frJ}{}^{k}_{E}$ are ordinary vector bundles, $\Psi$ is uniquely determined by its restriction to each fiber. For each $a \in M$, we thus have to define a linear isomorphism 
\begin{equation}
\Psi_{a}: (\frJ^{k}_{E})_{a} \rightarrow (\ol{\frJ}^{k}_{E})_{a} \equiv \frJ^{k}_{E,a}.
\end{equation}
But we constructed such a map for each $a \in M$ in Proposition \ref{tvrz_fibers}. One only has to prove that $\Psi$ is smooth. Let $(U,\varphi)$ be a local chart and $\{ \Phi_{\lambda} \}_{\lambda=1}^{r}$ a local frame for $E$ over $U$. As noted in Theorem (\ref{thm_generatorsgpJet}), the collection $\{ \fdelta^{\fI}{}_{\lambda}\}$, where $\lambda \in \{1,\dots,r\}$ and $\fI \in \cup_{q=0}^{k} \ol{\N}{}^{n}(q)$, forms a local frame for $\frJ^{k}_{\E}$ over $U$. It is easy to see from (\ref{eq_deltaprimeexplicit}, \ref{eq_psiadeltacommutes}) that one has
\begin{equation}
\Psi_{a}( \fdelta^{\fI}{}_{\lambda} |_{a}) = \frj^{k}_{a}[ \bbz^{\fI}_{a} \cdot \Phi_{\lambda}]_{a},
\end{equation}
for each $a \in U$, $\lambda \in \{1,\dots,r\}$ and $\fI \in \bigcup_{q=0}^{k} \ol{\N}{}^{n}(q)$. Composing this with the above local trivialization chart $\phi$ then shows that $\Psi$ is smooth. 
\end{proof}
\end{enumerate}

To conclude this paper, let us discuss a bit the uniqueness of $\frJ^{k}_{\E}$. First, up to a graded vector bundle isomorphisms, it is determined uniquely by the collection of short exact sequences (\ref{eq_SESwithjets}). Since short exact sequences split in the category $\gVBun^{\infty}_{\M}$ of graded vector bundles over $\M$, we get the isomorphism $\frJ^{k}_{\E} \cong \frJ^{k-1}_{\E} \oplus \ul{\Hom}(S^{k}(T\M),\E)$ for each $k \in \N_{0}$. We can thus iteratively construct $\frJ^{k}_{\E}$, starting with $\frJ^{0}_{\E} \cong \ul{\Hom}(S^{0}(T\M),\E) \cong \E$. 

Second, one can easily generalize Section \ref{sec_diffop} and Section \ref{sec_diffopVB} to construct a graded vector bundle $\frJ^{k}_{\E,\E'}$ of $k$-th order differential operators from $\E$ to $\E'$. It is then straightforward to modify the proof of Proposition \ref{tvrz_frDkashomfrJktoE} to obtain a canonical graded vector bundle isomorphism
\begin{equation} \label{eq_PsiEEprime}
\Psi^{\E,\E'}: \frD^{k}_{\E,\E'} \rightarrow \ul{\Hom}(\frJ^{k}_{\E},\E'). 
\end{equation}
natural in both $\E$ and $\E'$. Since the ``internal'' Yoneda functor $\E \mapsto \ul{\Hom}(\E,\cdot)$ is in a sense fully faithful, the existence of the isomorphism (\ref{eq_PsiEEprime}) determines the ``representing'' object $\frJ^{k}_{\E}$ up to a graded vector bundle isomorphism. 

\bibliography{bib}

 \end{document}